\documentclass[11pt]{amsart}
\usepackage{mathtools}
\usepackage{amssymb}
\usepackage{mathrsfs}
\usepackage[pagebackref,colorlinks=true,linkcolor=blue,citecolor=blue]{hyperref}
\usepackage{tikz-cd}

\UseRawInputEncoding

\usepackage{scalerel,stackengine}

\stackMath
\newcommand\reallywidehat[1]{%
\savestack{\tmpbox}{\stretchto{%
  \scaleto{%
    \scalerel*[\widthof{\ensuremath{#1}}]{\kern-.6pt\bigwedge\kern-.6pt}%
    {\rule[-\textheight/2]{1ex}{\textheight}}
  }{\textheight}%
}{0.5ex}}%
\stackon[1pt]{#1}{\tmpbox}%
}

\tikzset{
  symbol/.style={
    draw=none,
    every to/.append style={
      edge node={node [sloped, allow upside down, auto=false]{$#1$}}}
  }
}

\newcommand{\Z}{\mathbb{Z}}
\newcommand{\Q}{\mathbb{Q}}
\newcommand{\R}{\mathbb{R}}
\newcommand{\BC}{\mathbb{C}}

\renewcommand{\H}{\mathbb{H}}

\newcommand{\OO}{{\mathcal O}}

\newcommand{\ol}{\overline}

\newcommand{\Speh}{\mathrm{Speh}}

\newcommand{\WFP}{\mathfrak{p}^{m}}

\newcommand{\SL}{\mathrm{SL}}
\newcommand{\GL}{\mathrm{GL}}
\newcommand{\SO}{\mathrm{SO}}
\newcommand{\rO}{\mathrm{O}}
\newcommand{\Sp}{\mathrm{Sp}}
\newcommand{\St}{\mathrm{St}}
\newcommand{\RU}{\mathrm{U}}

\newcommand{\JL}{\mathrm{JL}}
\newcommand{\LJ}{\mathrm{LJ}}

\newcommand{\Fr}{\mathrm{Fr}}
\newcommand{\RG}{\mathrm{G}}

\newcommand{\Ind}{\mathrm{Ind}}

\newcommand{\WFN}{\overline{\mathfrak{n}}^{m}}

\newcommand{\half}[1]{\frac{#1}{2}}
\newcommand{\ul}[1]{\underline{\vphantom{p}}_{#1}}

\newcommand{\comment}[1]{}

\newcommand{\CO}{\mathcal{O}}

\newtheorem{thm}{Theorem}[section]
\newtheorem{cor}[thm]{Corollary}
\newtheorem{lemma}[thm]{Lemma}
\newtheorem{prop}[thm]{Proposition}
\newtheorem {conj}[thm]{Conjecture}

\newtheorem {ques/conj}[thm]{Question/Conjecture}

\newtheorem{defn}[thm]{Definition}
\newtheorem{remark}[thm]{Remark}

\newtheorem{exmp}[thm]{Example}

\newtheorem{assumption}[thm]{Working Hypotheses}

\newtheorem*{globalcond*}{Global Condition}
\newtheorem*{localcond*}{Local Condition}

\newtheorem*{globalconj*}{Global Conjecture}
\newtheorem*{localconj*}{Local Conjecture}

\newtheorem*{nonzero*}{Conjecture on the non-vanishing of the normalized intertwining operators}
\newtheorem*{holo*}{Conjecture on the holomorphicity of the normalized intertwining operators}

\DeclareMathOperator{\Aut}{Aut}

\DeclareMathOperator{\Gal}{Gal}

\DeclareMathOperator{\disc}{disc}
\numberwithin{equation}{section}

\let\oldbullet\bullet
\renewcommand{\bullet}{{\vcenter{\hbox{\tiny$\oldbullet$}}}}

\begin{document}
\title[Upper bound of wavefront sets]{On the upper bound of wavefront sets for representations of $p$-adic groups}

\author[A. Hazeltine]{Alexander Hazeltine
}
\address{Department of Mathematics\\
University of Michigan\\
Ann Arbor, MI 48109, USA}
\email{ahazelti@umich.edu}

\author[B. Liu]{Baiying Liu}
\address{Department of Mathematics\\
Purdue University\\
West Lafayette, IN, 47907, USA}
\email{liu2053@purdue.edu}

\author[C.-H. Lo]{Chi-Heng Lo}
\address{Department of Mathematics\\
National University of Singapore\\
119076, Singapore}
\email{{ch\_lo@nus.edu.sg}}

\author[F. Shahidi]{Freydoon Shahidi}
\address{Department of Mathematics\\
Purdue University\\
West Lafayette, IN, 47907, USA}
\email{freydoon.shahidi.1@purdue.edu}

\subjclass[2000]{Primary 11F70, 22E50; Secondary 11F85, 22E55}

\date{\today}


\keywords{Admissible Representations, Local Arthur Packets, Local Arthur Parameters, Nilpotent Orbits, Wavefront Sets}

\thanks{The research of the second named author is partially supported by the NSF Grant DMS-1848058 and the Simons Foundation: Travel Support for Mathematicians. The research of the fourth named author is partially supported by the NSF Grant DMS-2135021.}

\begin{abstract}
    In this paper, we study the upper bound of wavefront sets of irreducible admissible representations of connected reductive groups defined over non-Archimedean local fields of characteristic zero. We formulate a new conjecture on the upper bound and show that it can be reduced to that of anti-discrete series representations, namely, those whose Aubert-Zelevinsky duals are discrete series. 
    Then, we show that this conjecture is equivalent to the Jiang conjecture on the upper bound of wavefront sets of representations in local Arthur packets and also equivalent to an analogous conjecture on the upper bound of wavefront sets of representations in local ABV packets. 
\end{abstract}

\maketitle


\section{Introduction}

Let $F$ be a non-Archimedean field of characteristic zero. In this paper, we let $\RG$ denote a connected reductive algebraic groups defined over $F$, and put $G=\RG(F)$. We let $\Pi(G)$ denote the set of equivalence classes of irreducible admissible representations of $G$, and let $\Pi_{temp}(G)$ (resp. $\Pi_{2}(G)$) denote the subset of $\Pi(G)$ consisting of tempered (resp. discrete series) representations.
Given an irreducible {admissible} representation $\pi$ of $G$, one important invariant is a set $\mathfrak{n}(\pi)$ which is defined to be all the $F$-rational nilpotent orbits $\CO$ in the Lie algebra $\mathfrak{g}$ of $G$ such that the coefficient $c_{\CO}(\pi)$ in the Harish-Chandra-Howe local expansion of the character $\Theta(\pi)$ of $\pi$ is nonzero (see \cite{HC78, MW87}). Let $\mathfrak{n}^m(\pi)$ be the subset of $\mathfrak{n}(\pi)$ consisting of maximal nilpotent orbits, under the Zariski closure ordering of nilpotent orbits. We let $\ol{\mathfrak{n}}(\pi)$ be the sets of (geometric) nilpotent orbits over $\ol{F}$ corresponding to orbits in $\mathfrak{n}(\pi)$, and let $\ol{\mathfrak{n}}^m(\pi)$ be the subset of maximal elements under the dominance ordering. 
The set $\ol{\mathfrak{n}}^m(\pi)$ is called the (geometric) wavefront set of $\pi$.

In general, it is known that the structures of wavefront sets of representations are closely related to theory of Langlands functorial lifting and descent, and the Gan-Gross-Prasad conjecture. 
It is an interesting and long-standing question to study the structures of the sets $\mathfrak{n}(\pi)$,  $\mathfrak{n}^m(\pi)$, $\ol{\mathfrak{n}}(\pi)$ and $\ol{\mathfrak{n}}^m(\pi)$. Recently, Tsai (\cite{Tsa24}) constructed an example of representations of $\RU_7(\mathbb{Q}_3)$ showing that the wavefront set $\ol{\mathfrak{n}}^m(\pi)$ may not be a singleton, which implies the complication of its study in general.
In this paper, inspired by the work of \cite{CMBO22}, we consider the following conjecture on the upper bound of the wavefront set $\ol{\mathfrak{n}}^m(\pi)$ based on the local Langlands parameter of its 
Aubert-Zelevinsky dual $\widehat{\pi}$. 
{Note that 
one can attach nilpotent orbits to $L$-parameters as follows. Let $\widehat{G}(\BC)$ be the complex dual group of $\RG$ and $\widehat{\mathfrak{g}}(\BC)$ denote its Lie algebra. For an $L$-parameter $\phi$ of $G$, let $\{H,X,Y\}$ be an $\mathfrak{sl}_2$-triple of $\widehat{\mathfrak{g}}(\BC)$ associated to the morphism $\phi|_{\SL_2(\BC)}: \SL_2(\BC) \to \widehat{G}(\BC)$. We define $\OO_{\phi}$ to be the nilpotent orbit of $\widehat{\mathfrak{g}}(\BC)$ containing $X$ (see \S\ref{lp and lap} for more details). }

\begin{conj}[Upper Bound Conjecture, Conjecture \ref{conj bound of WF}]\label{conj bound of WF intro}
Assume the local Langlands correspondence for irreducible admissible representations of $G$. 
Let $\pi$ be an irreducible admissible representation of $G$. For any ${\CO} \in \ol{\mathfrak{n}}^m(\pi)$, 
the following inequality holds
\[{\CO} \leq d_{BV}( {\CO}_{\phi_{\widehat{\pi}} }),\]
where $\widehat{\pi}$ is the Aubert-Zelevinsky dual of $\pi$, $\phi_{\widehat{\pi}}$ is the local Langlands parameter of $\widehat{\pi}$. 
Here, 
$d_{BV}$ is the Barbasch-Vogan duality map from nilpotent orbits in $\widehat{\mathfrak{g}}(\mathbb{C})$ to nilpotent orbits in $\mathfrak{g}(\mathbb{C})$ which are naturally identified with nilpotent orbits in $\mathfrak{g}(\ol{F})$ (see \cite{Spa82, Lus84, BV85, Ach03} and \S \ref{wfs}). 
\end{conj}

Following the idea of M{\oe}glin and Waldspurger (see \cite[\S II.2]{MW87}), for the cases of $\GL_n(F)$ and its inner form, we compute the wavefront set for any irreducible representation and show that the Conjecture \ref{conj bound of WF intro} is true (see Theorem \ref{thm main GL}). 
In this paper, we focus on the case of the quasi-split classical groups
$\RG_n=\RU_n, \Sp_{2n}, \SO_{2n+1}, \SO^{\alpha}_{2n}$, and their pure inner forms (though the mains results naturally extend to general connected reductive groups, see Remark \ref{general thm}), where $\alpha$ is a square class and $n$ is any positive integer. Here, we identify a square class with the corresponding quadratic character of the Weil group $W_F$ via the local class field theory.  Let $G_n = \RG_n(F)$. 
An irreducible admissible representation of $G_n$ is called anti-tempered (resp. anti-discrete) if its Aubert-Zelevinsky involution is tempered (resp. discrete series).  
Our first main result is as follows.

\begin{thm}[Theorem \ref{thm main classical groups}]\label{thm main classical groups intro}
The following statements are equivalent.
\begin{enumerate}
\item Conjecture \ref{conj bound of WF intro} holds for all admissible representations of $G_m$, $m \leq n$.
\item Conjecture \ref{conj bound of WF intro} holds for all representations of $G_m$ of Arthur type, $m \leq n$.
\item Conjecture \ref{conj bound of WF intro} holds for all anti-tempered representations of $G_m$, $m \leq n$.
\item Conjecture \ref{conj bound of WF intro} holds for all anti-discrete representations of $G_m$, $m \leq n$.
\end{enumerate}
\end{thm}

We remark that the computation of $\phi_{\widehat{\pi}}$ or $\widehat{\pi}$ from $\pi$ is in general complicated. For general linear groups, a recursive algorithm is given in \cite{MW86}, and for $\Sp_{2n}(F)$ and split $\SO_{2n+1}(F)$, a recursive algorithm is given in \cite{AM20}. Currently, there is no algorithm for other groups. However, there are several special cases that $d_{BV}(\OO_{\phi_{\widehat{\pi}}})$ can be computed explicitly. Here are two important families.
\begin{enumerate}
    \item [$\oldbullet$] If $\pi$ is a generic representation of $G_n$, then $\OO_{\phi_{\widehat{\pi}}}$ is the zero orbit (see \cite{LLS24b}), and hence $d_{BV}(\OO_{\phi_{\widehat{\pi}}})$ is the regular orbit (see \cite{BV85, Ach03}).
    Thus, the wavefront sets of generic representations achieve the conjectural upper bound in Conjecture \ref{conj bound of WF intro}. 
    \item [$\oldbullet$] If $\pi$ is anti-tempered or anti-discrete, then $\phi_{\widehat{\pi}}$ is the unique open $L$-parameter in the associated Vogan variety (see \cite[Proposition 5.6]{CFMMX22}). Therefore, the conjectural upper bound $d_{BV}(\OO_{\phi_{\widehat{\pi}}})$ can be easily computed from the $L$-parameter (or infinitesimal parameter) of $\pi$.
\end{enumerate}
Thus, Theorem \ref{thm main classical groups intro} allows us to reduce Conjecture \ref{conj bound of WF intro} to a case that the conjectural upper bound $d_{BV}(\OO_{\phi_{\widehat{\pi}}})$ can be computed explicitly.

Our second main result is that Conjecture \ref{conj bound of WF intro} is equivalent to the Jiang conjecture on the upper bound of wavefront sets of representations in local Arthur packets (see Conjecture \ref{conj Jiang}). More precisely, the well-known Shahidi conjecture states that  tempered $L$-packets of any quasi-split reductive group have generic members. The Jiang conjecture (see \cite[Conjecture 4.2]{Jia14} and \cite[Conjecture 1.7]{LS23}) generalizes the Shahidi conjecture to arbitrary local Arthur packets as follows. 

\begin{conj}[Jiang Conjecture, Conjecture \ref{conj Jiang}] \label{conj Jiang intro}
Let $\RG$ be a connected reductive group over $F$ and let $G=\RG(F)$. Assume that there is a local Arthur packets theory for $G$ as conjectured in \cite[Conjecture 6.1]{Art89}. Then for any $\psi \in \Psi(G)$, the following holds.
\begin{enumerate}
    \item [(i)] For any $\pi\in \Pi_{\psi}$ and any $\OO \in \WFN(\pi)$, we have
    $ \OO \leq d_{BV}(\OO_\psi). $
    \item [(ii)] If $G$ is quasi-split over $F$, there exists at least one member $\pi \in \Pi_{\psi}$ having the property that $\overline{\mathfrak{n}}^m(\pi)=\{d_{BV}(\OO_\psi)\}$.
\end{enumerate}
Here, $\Psi(G)$ is the set of local Arthur parameters of $G$ (see \S \ref{lp and lap}). Given $\psi \in \Psi(G)$, $\Pi_{\psi}$ is the corresponding local Arthur packet (see \S \ref{lap ABV Jiang}), $\OO_\psi$
is the corresponding geometric nilpotent orbit (see \S \ref{lp and lap}), and
$d_{BV}(\OO_\psi)$ is the nilpotent orbit obtained via applying the Barbasch-Vogan duality map to $\OO_\psi$.
\end{conj}

Note that $d_{BV}(\OO_\psi)=d_{BV}(\OO_{\phi_{\widehat{\psi}}})$, where $\widehat{\psi}$ is the dual of $\psi$ (see \eqref{eq psi hat}), and $\phi_{\widehat{\psi}}$ is the $L$-parameter associated to $\widehat{\psi}$ (see \eqref{phi_psi}). 

In \cite{LS23}, the second and the fourth named authors studied Conjecture \ref{conj Jiang intro} for quasi-split classical groups, using the matching method in endoscopic lifting and partially reduced it to certain properties of the wavefront sets of bitorsor representations.
Recently,
\cite{AC26} claims a proof of Conjecture \ref{conj Jiang intro} for split classical groups under certain hypothesis, applying similar matching method.

In this paper, we focus on Part (i) of the Jiang conjecture and specialize to pure inner forms of classical groups or general linear groups. We rephrase it as follows.

\begin{conj}[Conjecture \ref{conj Jiang 2}]\label{conj Jiang 2 intro}
Assume there is a local Arthur packets theory for $G$ as conjectured in \cite[Conjecture 6.1]{Art89}. 
For any $\pi \in \Pi_A(G)$, the subset of $\Pi(G)$ that consists of representations of Arthur type, 
and $\OO \in \overline{\mathfrak{n}}^m(\pi)$, we have $$\OO \leq d_{BV}(\OO_\psi),$$
for any $\psi \in \Psi(\pi):= \{ \psi \in \Psi(G) \ | \ \pi \in \Pi_{\psi}  \}.$
\end{conj}

Our second main result is as follows. 

\begin{thm}[Theorem \ref{thm main Jiang}]\label{thm main Jiang intro}
    Assume there is a local Arthur packets theory for $G$ as conjectured in \cite[Conjecture 6.1]{Art89} and the closure ordering relation holds (see Working Hypothesis \ref{assu closure ordering A-packet}) for any local Arthur packet of $G_m$, $m \leq n$. Then the following statements are equivalent.
\begin{enumerate}
\item Conjecture \ref{conj bound of WF intro} holds for all admissible representations of $G_m$, $m \leq n$.
\item Conjecture \ref{conj Jiang 2 intro} holds for all representations of $G_m$ of Arthur type, $m \leq n$
\item Conjecture \ref{conj Jiang 2 intro} holds for all anti-tempered representations of $G_m$, $m \leq n$.
\item Conjecture \ref{conj Jiang 2 intro} holds for all anti-discrete representations of $G_m$, $m \leq n$.
\end{enumerate}
\end{thm}

We remark that for $\Sp_{2n}(F)$ and split $\SO_{2n+1}(F)$, the closure ordering relation has been proved in {\cite[Theorem 1.3]{HLLZ22}}.
We also remark that assuming Working Hypothesis \ref{assu closure ordering A-packet}, for any $\pi \in \Pi_{\psi}$, we have $d_{BV}(\OO_{\phi_{\widehat{\pi}}}) \leq d_{BV}(\OO_\psi)$, and the inequality can be strict (see \eqref{eq sharper upper bound} and Example \ref{exmp strict inequality}). 

In this paper, we also consider the following generalization of the Shahidi conjecture and the Jiang conjecture to the local ABV-packets defined in \cite{CFMMX22}. 

\begin{conj}[Generalized Shahidi Conjecture on local ABV packets, Conjecture \ref{conj Jiang ABV}]\label{conj Jiang ABV intro}
Let $\RG$ be a connected reductive group over $F$ that has a quasi-split pure inner form and let $G=\RG(F)$. For any $\phi \in \Phi(G)$, the following holds.
\begin{enumerate}
    \item [(i)] For any $\pi\in \Pi_{\phi}^{\textrm{ABV}}$ and any $\OO \in \WFN(\pi)$, we have
    $ \OO \leq d_{BV}(\OO_{\widehat{\phi}}). $
    \item [(ii)]  If $G$ is quasi-split over $F$, there exists at least one member $\pi \in \Pi_{\phi}^{\textrm{ABV}}$ having the property that $\WFN(\pi)=\{d_{BV}(\OO_{\widehat{\phi}})\}$.
\end{enumerate}
Here, $\Phi(G)$ is the set of local $L$ parameters of $G$ (see \S \ref{lp and lap}). Given $\phi \in \Phi(G)$, $\Pi_{\phi}^{ABV}$ is the corresponding ABV-packet (see \S \ref{lap ABV Jiang}). The $L$-parameter $\widehat{\phi}$ is the Pyasetskii involution of $\phi$ (see \cite[\S 6.4]{CFMMX22} and \S \ref{abv}) and 
$\OO_{\widehat{\phi}}$
is the corresponding geometric nilpotent orbit (see \S \ref{lp and lap}).
\end{conj}

Again, we shall focus on Part (i) of the above conjecture for $G_n$, which can be rephrased as follows.

\begin{conj}[Conjecture \ref{conj Jiang ABV 2}]\label{conj Jiang ABV 2 intro}
For any $\pi \in \Pi(G)$ and any $\OO \in \WFN(\pi)$, we have 
$$\OO \leq d_{BV}(\OO_{\widehat{\phi}}),$$
for any $\phi \in \Phi(\pi):= \{ \phi \in \Phi(G) \ | \ \pi \in \Pi_{\phi}^{\textrm{ABV}}  \}.$
\end{conj}

Our third main result is as follows. 

\begin{thm}[Theorem \ref{thm main Jiang ABV}]\label{thm main Jiang ABV intro}
    Assume the Aubert-Zelevinsky involution preserves local ABV packets (see Working Hypothesis \ref{assu ABV AZ dual}) for any local Arthur packet of $G_m$, $m \leq n$. Then the following statements are equivalent.
\begin{enumerate}
\item Conjecture \ref{conj bound of WF intro} holds for all admissible representations of $G_m$, $m \leq n$.
\item Conjecture \ref{conj Jiang ABV 2 intro} holds for all admissible representations of $G_m$, $m \leq n$.
\item Conjecture \ref{conj Jiang ABV 2 intro} holds for all anti-tempered representations of $G_m$, $m \leq n$.
\item Conjecture \ref{conj Jiang ABV 2 intro} holds for all anti-discrete representations of $G_m$, $m \leq n$.
\end{enumerate}
\end{thm}

In a forthcoming paper (\cite{LLSZ26}), the last three named authors and Zhang will prove Conjecture \ref{conj bound of WF intro} for all anti-tempered representations of $G_n$, unconditionally, hence fully prove Conjectures \ref{conj bound of WF intro}, \ref{conj Jiang 2 intro}, and
\ref{conj Jiang ABV 2 intro}. 

We also show that the first three statements in Theorems \ref{thm main classical groups intro}, \ref{thm main Jiang intro} can be reduced from a family of groups $G_m$ with $m\leq n$ to a single group $G_n$ (see Theorem \ref{thm reduction to a single gp}). We expect that the same argument should work for Theorem \ref{thm main Jiang ABV intro}. However, the reduction seems infeasible for the statements for anti-discrete representations (see Remark \ref{rmk reduction to a single group}).

Theorems \ref{thm main classical groups intro}, \ref{thm main Jiang intro}, and \ref{thm main Jiang ABV intro} are proved uniformly in \S \ref{sec proof classical}, via applying a key lemma (Lemma \ref{lem goal}) that the Barabsch-Vogan duality is compatible with induction. This lemma was stated in \cite[Proposition A.2(c)]{BV85}, with a sketched proof. To be complete, we include the detailed proof in \S \ref{sec proof of key lemma} for classical groups. For the case of $\GL_n$, since local Arthur packets are singletons and Working Hypothesis \ref{assu ABV AZ dual} has been verified in \cite[Proposition 3.2.1]{CFK22}, as a corollary of Theorem \ref{thm main GL}, we can completely prove Conjectures \ref{conj Jiang 2 intro} and \ref{conj Jiang ABV 2 intro} (see Corollary \ref{cor GLn}).

{Let $\RG$ be a general connected reductive algebraic group defined over $F$ and let $G=\RG(F)$. Assume the local Langlands correspondence (Conjecture \ref{conj LLC}), the theories of local Arthur packets and ABV-packets for $G$, then all the results in Theorems \ref{thm main classical groups intro}, \ref{thm main Jiang intro}, and \ref{thm main Jiang ABV intro} can be naturally extended, applying \cite[Proposition A2(c)]{BV85} instead of Lemma \ref{lem goal}. For the convenience of future references, we state it in the following remark. However, for non-pure inner forms, the upper bound may not be sharp. In this case, we expect that the upper bound can be improved by modifying the definition of $\OO_{\phi}$, as in the case for inner forms of $\GL_n$.
For more detailed discussion, see Remark \ref{remarks}.}

\begin{remark}\label{general thm}
    All the results in Theorems \ref{thm main classical groups intro}, \ref{thm main Jiang intro}, and \ref{thm main Jiang ABV intro} naturally extend to general connected reductive algebraic groups $G$. More precisely,
    with the same assumptions as in Theorems \ref{thm main Jiang intro} and \ref{thm main Jiang ABV intro}, and assuming the theories of local Langlands correspondence, local Arthur packets, and local ABV-packets for $G$, then, 
    Conjectures \ref{conj bound of WF intro}, \ref{conj Jiang 2 intro}, and \ref{conj Jiang ABV 2 intro} hold for $G$ if and only if they hold for any anti-discrete representation of any Levi subgroup of $G$.
\end{remark}

The proofs of our main results in Theorems \ref{thm main classical groups intro}, \ref{thm main Jiang intro}, and \ref{thm main Jiang ABV intro} can be summarized in the following diagram. 

\begin{equation}\label{eq diagram intro}
        \begin{tikzcd}[column sep=huge]
   (\Phi) \ar[dd, Rightarrow] &(\Pi)\ar[l, Leftrightarrow, "\text{(D)}"',"\text{Lemma }\ref{lem (D) (E) (G)}"]  \ar[d, Rightarrow]& \\
    &(\Pi_A) \ar[d, Rightarrow] \ar[r, Rightarrow, "\text{(B)}", "\text{Lemma }\ref{lem (B)}"'] & (\Psi_A) \ar[d, Rightarrow]\\
    (\Phi_{\widehat{temp}}) \ar[d, Rightarrow]&(\Pi_{\widehat{temp}}) \ar[l, Leftrightarrow, "\text{(E)}"',"\text{Lemma }\ref{lem (D) (E) (G)}"]  \ar[d, Leftrightarrow, "\text{(F)}", "\text{Theorem }\ref{thm (F)}"'] \ar[uu, bend left=60, Rightarrow,"\text{(A)}"', " 
\text{Theorem }\ref{thm (A)}"] & (\Psi_{\widehat{temp}})\ar[d, Rightarrow] \ar[l, Rightarrow, "\text{(C)}"', "\text{Lemma }\ref{lem (C) (H)}"]\\
        (\Phi_{\widehat{2}})&(\Pi_{\widehat{2}})\ar[l, Leftrightarrow, , "\text{(G)}"',"\text{Lemma }\ref{lem (D) (E) (G)}"] \ar[r, Leftarrow, , "\text{(H)}",  "\text{Lemma }\ref{lem (C) (H)}"'] & (\Psi_{\widehat{2}})
    \end{tikzcd}
\end{equation}
Here let 
\[(\Xi, -) \in \{(\Pi, \ref{conj bound of WF intro}), (\Psi, \ref{conj Jiang 2 intro}), (\Phi, \ref{conj Jiang ABV 2 intro})\}.\]
We consider the collection of statements
\begin{align}\label{eq Xi ast intro}
    \tag{$\Xi_{\ast}$} \text{Conjecture } - \text{ holds for any }\pi \in \Pi_{\ast}(G_m) \text{ for any }m \leq n,
\end{align}
where the subscript $\ast \in \{\emptyset, A, \widehat{temp}, \widehat{2}\}$. $\Xi_{\emptyset}$ is understood as $\Xi$. 
$\Pi(G_m)$ is the set of all irreducible admissible representations of $G_m$, $\Pi_{A}(G_m)$ is the set of all representations of $G_m$ of Arthur type,  $\Pi_{\widehat{temp}}(G_m)$ is the set of all anti-tempered representations of $G_m$, and $\Pi_{\widehat{2}}(G_m)$ is the set of all anti-discrete representations of $G_m$. 
The vertical implications downward immediately follow from the following chain of containments 
\begin{align}\label{eq chain of containment intro}
    \Pi(G_m) \supseteq \Pi_A(G_m) \supseteq \Pi_{\widehat{temp}}(G_m) \supseteq \Pi_{\widehat{2}}(G_m).
\end{align}
We remark that Working Hypothesis \ref{assu closure ordering A-packet} is used in direction (B) and Working Hypothesis \ref{assu ABV AZ dual} is used in directions (D), (E) and (G).

From the diagram \eqref{eq diagram intro}, we can see that to prove Conjectures \ref{conj bound of WF intro}, \ref{conj Jiang 2 intro}, and \ref{conj Jiang ABV 2 intro}, we just need to prove Conjecture \ref{conj Jiang 2 intro} for anti-discrete representations, which is a work in progress of the authors. We hope the recent progress on the explicit computation of Aubert-Zelevinsky involution in \cite{AM20} could play important roles here. 
On the other hand, we remark that our method indeed has limitation towards proving Conjectures \ref{conj bound of WF intro}, \ref{conj Jiang 2 intro}, and \ref{conj Jiang ABV 2 intro} for anti-discrete representations, see the discussion and examples given in \S \ref{limitation}.

In this paper, we also discuss the status of Conjectures \ref{conj bound of WF intro} and \ref{conj Jiang  intro} for unipotent representations. More precisely, based on recent progresses in \cite{CMBO22, CMBO23, Wal18, Wal20} and the results in \cite{HLLZ22}, we prove Conjecture \ref{conj Jiang intro} in special cases (see Theorems \ref{thm Jiang upper bound unipotent} and \ref{thm upper bound SO(2n+1) unip}). We also show that \cite[Conjecture 1.4.3]{CMBO23} (see Conjecture \ref{conj CMBO}) implies Conjecture \ref{conj bound of WF intro} for unipotent representations (see Theorem \ref{thm equiv CMBO conj}). Note that \cite[Conjecture 1.4.3]{CMBO23} has already been proved if $\pi$ has real infinitesimal parameter in the same paper. We expect that for unipotent tempered representations of classical groups, the upper bounds given in Conjecture \ref{conj bound of WF intro} would agree with the arithmetic wavefront sets introduced by Jiang-Liu-Zhang (\cite{JLZ22, CJLZ23}) and with the wavefront sets computed by Waldspurger (for $\SO_{2n+1}$, \cite{Wal20}). We verify this expectation in two families (see Example \ref{exmp SO odd}). We remark that recently, La (\cite{La24}) studied the relation between the wavefront sets computed by Waldspurger and the Iwahori–Matsumoto involution on affine Hecke
algebras, which is expected to match the Aubert-Zelevinsky dual. See Remark \ref{rmk La} for more details.

Given an irreducible automorphic representation $\pi = \otimes_v \pi_v$, the upper bound given in Conjecture \ref{conj bound of WF intro} for each $\pi_v$ will provide an upper bound for the wavefront set of $\pi$. An interesting question is whether the minimum of all these upper bounds would occur in the wavefront set of $\pi$.

We remark that the upper bound in Conjecture \ref{conj bound of WF intro} may not be always sharp for an individual representation $\pi$. That is, it is possible that for any $\OO \in \WFN(\pi)$, $\OO < d_{BV}(\OO_{\phi_{\widehat{\pi}}}).$
We thank Cheng-Chiang Tsai for helpful discussions on the first two of the following examples. 

\begin{exmp}\label{exmp non-achieve}\ 
    \begin{enumerate}
    \item [1.] Let $n \in \Z_{>0}$, $G$ be the split $\SO_{4n+1}(F)$ and $P=MN$ be the Siegel parabolic subgroup of $G$. Let $\rho$ be a selfdual supercuspidal representation of $M$ whose $L$-parameter $\phi_{\rho}$ is symplectic. The parabolic induction $\Ind_{M}^G \rho$ is reducible and decomposes into a direct sum $\pi_{1}\oplus \pi_2$. They form an $L$-packet of $G$. That is, $\Pi_{\phi}=\{\pi_1,\pi_2\}$ where $\phi:= \phi_{\rho}+ \phi_{\rho}$. In this case, $\widehat{\pi_1}=\pi_2$ and $\widehat{\pi_2}=\pi_1$. Thus $d_{BV}(\OO_{\phi_{\widehat{\pi_i}}})= d_{BV}(\OO_{\phi})$, which is the regular orbit. On the other hand, the Whittaker model of $\Ind_{M}^G \rho$ is one dimensional, and hence exactly one of $\pi_1, \pi_2$ is generic. We conclude that the other one does not achieve the conjectural upper bound in Conjecture \ref{conj bound of WF intro}.  
        \item [2.]Let $\pi$ be the epipelagic supercuspidal representation of $\RU_7(\Q_3)$ constructed in \cite{Tsa24} whose wavefront set is not a singleton. We have $\pi=\widehat{\pi}$, and $\phi_{\widehat{\pi}}$ is trivial when restricted to $\SL_2(\BC)$. As a consequence, $\OO_{\phi_{\widehat{\pi}}}$ is the zero orbit and neither of the orbits in $\WFN(\pi)$ can achieve $d_{BV}(\OO_{\phi_{\widehat{\pi}}})$, which is the regular orbit. More generally, if $\pi$ is any representation such that $\WFN(\pi)$ is not a singleton and Conjecture \ref{conj bound of WF intro} holds for $\pi$, then none of the nilpotent orbits in $\WFN(\pi)$ can achieve the conjectural upper bound $d_{BV}(\OO_{\phi_{\widehat{\pi}}})$.
        \item [3.] In \cite[Example 1.4.2]{CMBO23}, they gave an example of a unipotent representation of $\textrm{E}_7$ whose wavefront set is strictly smaller than the upper bound in Conjecture \ref{conj bound of WF intro}.
    \end{enumerate}
\end{exmp}

However, for quasi-split connected reductive groups and their pure inner forms, we expect that the conjectural upper bound in Conjecture \ref{conj bound of WF intro} is sharp for a ``packet" of representations. As generalizations of the Shahidi conjecture, in \cite{LLS24a}, the last three named authors formulated several conjectures on the sharpness of the conjectural upper bound of wavefront sets of representations in arbitrary local $L$-packets and considered similar reductions for groups $G_n$ as in this paper. More precisely, based on Conjecture \ref{conj bound of WF intro}, the first version of generalizations is as follows. 

\begin{conj}[Generalized Shahidi Conjecture on $L$-packets, {\cite[Conjecture 1.7]{LLS24a}}]\label{generalized Shahdi's conjecture 1st attempt}
    Let $\mathrm{G}$ be a connected reductive group and $G=\mathrm{G}(F)$. 
Assume the local Langlands correspondence for $G$. Let $\phi$ be a local $L$-parameter of $G$ and $\Pi_{\phi}$ be the corresponding local $L$-packet. Let 
\[\mathrm{UB}(\phi):=\max \{d_{BV}( \OO_{\phi_{\widehat{\pi}}}) \ |\  \pi \in \Pi_{\phi}\} \ \ \ \ \text{(not necessarily a singleton)}.\]
Then the followings hold. 
\begin{enumerate}
    \item [(i)] For any representation $\pi$ in ${\Pi}_{\phi}$ and any nilpotent orbit $\OO \in \ol{\mathfrak{n}}^m(\pi)$, there exists $\OO' \in \mathrm{UB}(\phi)$ such that  
    \[ \OO \leq \OO'.\]
    \item [(ii)] If $G$ is quasi-split, then for any $\OO' \in \mathrm{UB}(\phi)$, there exists a representation $\pi \in {\Pi}_{\phi}$ such that 
    \[ \ol{\mathfrak{n}}^m(\pi)= \{ \OO'\}.\]
\end{enumerate}
\end{conj}

Note that Conjecture \ref{generalized Shahdi's conjecture 1st attempt}(i) is slightly weaker than Conjecture \ref{conj bound of WF intro}.
One issue about Conjecture \ref{generalized Shahdi's conjecture 1st attempt} is that given a local $L$-parameter $\phi$, $\mathrm{UB}(\phi)$ may not be a singleton. See \cite[Section 5.1]{LLS24a} for an example. This makes Conjecture \ref{generalized Shahdi's conjecture 1st attempt} very hard to consider. 
One way to deal with the difficulty is to consider the {ABV-packets} as in  Conjecture \ref{conj Jiang ABV intro}.
A closer look at Conjecture \ref{conj bound of WF intro} leads us to a modification of Conjecture \ref{generalized Shahdi's conjecture 1st attempt}. 
More precisely, we have a conjecture as follows which is more feasible.

\begin{conj}[Generalized Shahidi Conjecture on duals of $L$-packets, {\cite[Conjecture 1.7]{LLS24a}}]\label{generalized Shahdi's conjecture}
    Let $\mathrm{G}$ be a connected reductive group and $G=\mathrm{G}(F)$. 
Assume that the local Langlands correspondence for $G$ holds. Let $\phi$ be a local $L$-parameter of $G$ and $\Pi_{\phi}$ be the corresponding local $L$-packet. Then the followings hold. 
\begin{enumerate}
    \item [(i)] For any representation $\pi$ in $\widehat{\Pi}_{\phi}$ and any nilpotent orbit $\OO \in \ol{\mathfrak{n}}^m(\pi)$, we have 
    \[ \OO \leq d_{BV}( \OO_{\phi}).\]
    \item [(ii)] If $G$ is quasi-split, then there exists a representation $\pi \in \widehat{\Pi}_{\phi}$ such that
    \[ \ol{\mathfrak{n}}^m(\pi)= \{ d_{BV}( \OO_{\phi})\}.\]
\end{enumerate}
Here, $\widehat{\Pi}_{\phi}:= \{ \pi \ | \ \widehat{\pi} \in \Pi_{\phi} \}.$
Note that $\phi = \phi_{\widehat{\pi}}$.
\end{conj}

Note that Conjecture \ref{generalized Shahdi's conjecture}(i) is also equivalent to Conjecture \ref{conj bound of WF intro}. The novelty of Conjecture \ref{generalized Shahdi's conjecture} is the point of view that to consider the wavefront sets of admissible representations and to obtain a unique upper bound of the wavefront sets, one should consider the Aubert-Zelevinsky dual of the local $L$-packet $\Pi_{\phi}$, not the local $L$-packet itself. Note that using the the Aubert-Zelevinsky dual of the local $L$-packets, we still have the following disjoint decomposition of the irreducible admissible dual of $G$ as follows. 
$$\mathrm{Irr}(G)=\cup_{\phi} \ \widehat{\Pi}_{\phi}.$$

In \cite{LLS24a}, we focused on Conjecture \ref{conj Jiang intro}(ii) and 
Conjecture \ref{generalized Shahdi's conjecture}(ii), and proved similar reductions as in  Theorems \ref{thm main classical groups intro}, \ref{thm main Jiang intro}, and \ref{thm main Jiang ABV intro}.

We remark that Ciubotaru and Kim (\cite{CK25}) also made and studied Conjecture \ref{conj bound of WF intro}, at the same time as the current paper, and proved the equivalence of Theorem \ref{thm main classical groups intro} Parts (1) and (3), independently. 
They verified it
in some cases, including the depth-zero supercuspidal representations of classical groups and all the irreducible representations of $G_2$.
On the other hand, our paper addresses certain other aspects as well, especially the relation of Conjecture \ref{conj bound of WF intro} with the Jiang conjecture on wavefront sets of local Arthur packets and the generalized Shahidi conjecture on wavefront sets of local ABV packets (see Theorems \ref{thm main Jiang intro} and \ref{thm main Jiang ABV intro} above). In particular, we include a detailed proof of a key ingredient (\cite[Proposition A.2(c)]{BV85}) for classical groups (see Lemma \ref{lem goal}).

The following is the structure of this paper. In \S 
\ref{sec groups}, we introduce all the groups considered in this paper. In \S \ref{wfs} and \S \ref{lp and lap}, we introduce the notation on wavefront set and partitions, local $L$-parameters and local Arthur parameters. In \S \ref{llc and lc}, we recall the local Langlands correspondence and the Langlands classification for general linear groups and the groups $G_n$, and the closure ordering relation for local $L$-parameters. In \S \ref{lap ABV Jiang}, we recall the expected properties of local Arthur packets and ABV-packets and state the Jiang conjecture. 
In \S \ref{main results}, we state all the main results (Theorems \ref{thm main classical groups intro}, \ref{thm main Jiang intro}, \ref{thm main Jiang ABV intro}) 
precisely, which will be proved in \S \ref{sec proof classical}. In \S \ref{sec GLn(A) main}, following the idea of M{\oe}glin and Waldspurger, we prove the $\GL_n$ case of Conjectures \ref{conj bound of WF intro}, \ref{conj Jiang 2 intro}, and \ref{conj Jiang ABV intro} (Theorem \ref{thm main GL} and Corollary \ref{cor GLn}). In \S \ref{limitation}, we discuss the limitation of the method used in this paper for anti-discrete representations and give explicit examples. In \S \ref{sec unipotent}, we consider the unipotent cases of Conjectures \ref{conj bound of WF intro} and \ref{conj Jiang  intro} (Theorems \ref{thm Jiang upper bound unipotent} and \ref{thm upper bound SO(2n+1) unip}). In \S \ref{sec proof of key lemma}, we prove in details the key Lemma \ref{lem goal}.

\subsection*{Acknowledgements} 
The authors would like to thank Dihua Jiang for the constant support and encouragement. The authors are grateful to Cheng-Chiang Tsai for helpful comments and suggestions, especially for the first two of Example \ref{exmp non-achieve}, and to Lei Zhang and Cheng Chen for the helpful discussion on arithmetic wavefront sets of representations.
The authors also would like to thank Ciubotaru and Kim for informing us of  their related independent work on the upper bound conjecture. 

\section{Groups}\label{sec groups}
In this section, we give notations for the groups considered in this paper, which are inner forms of general linear groups and pure inner forms of classical groups.

First, we consider inner forms of $\GL_n(F)$. An inner form of $\GL_n(F)$ is of the form $\GL_m(A)$, where $A$ is a central division algebra of dimension $d_A^2$ over $F$ with $n=md_A$. We let $|\cdot|_A$ denote the reduced norm of $A$. We also regard it as a character on $\GL_{m}(A)$ by composing the determinant. If $\pi$ is a representation of $\GL_m(A)$ and $x \in \R$, we let $\pi |\cdot|_A^{x}$ denote the representation acting on the same vector space of $\pi$ by $\pi|\cdot|_A^x (g):= \pi(g) |\det(g)|_A^x$.

Let $G= \GL_m(A)$. There is a one-to-one correspondence between ordered partitions of $m$ and the conjugacy classes of parabolic subgroups of $G$. Let $\underline{n}=( n_1,\ldots, n_f)$ be an ordered partition of $m$ and let $P_{\underline{n}}$ be the corresponding parabolic subgroup. The Levi subgroup $M_{\underline{n}}$ of $P_{\underline{n}}$ is isomorphic to
\[ \GL_{n_1}(A) \times \cdots \times \GL_{n_f}(A).\]
If $\tau_i$'s are representations of $\GL_{n_i}(A)$ respectively, then we regard $\tau:=\sigma_1 \otimes \cdots \otimes \tau_f$ as a representation of $M_{\underline{n}}$ and denote the normalized parabolic induction $\Ind_{M_{\underline{n}}}^G \tau$ by
\[ \tau_1 \times \cdots \times \tau_f.\]

Next, we consider pure inner forms of classical groups, which are symplectic, special orthogonal and unitary groups. Let $E$ be a non-Archimedean local field of characteristic zero such that $[E:F]\leq 2$, and let $\sigma$ be the involution of $E$ with fixed field $F$. That is, if $\sigma=1,$ then $E=F$. Otherwise, $\sigma$ is the non-trivial element in $\Gal(E/F).$

Let $\epsilon\in\{\pm1\}$ and $V$ be an $\epsilon$-Hermitian space over $E$ of dimension $n$. We let $\langle\cdot,\cdot\rangle$ denote the non-degenerate, $\sigma$-sesquilinear form on $V$ and $\mathfrak{r}$ denote the Witt index of $V$. Thus, $V$ is isometric to $ V_{an}\oplus \H^\mathfrak{r},$ where $V_{an}$ is anisotropic $\epsilon$ Hermitian space over $E$, and $\H \cong E_{v}\oplus E_{v^{\ast}}$ is the hyperbolic plane defined by $\langle v, v\rangle = \langle v^{\ast}, v^{\ast}\rangle=0$ and $\langle v, v^{\ast} \rangle =1$. Conversely, fixing an anisotropic $\epsilon$-Hermitian space $V_{an}$ over $E$, we let $V_{an,r}:=  V_{an} \oplus \H^{r}$. We shall consider the Witt tower associated to $V_{an}$
\[ \mathcal{V}:= \{ V_{an,r}\ | \ r \geq 0\}. \]

The pure inner forms of classical groups can be described by the identity component of the isometry groups of an $\epsilon$-Hermitian space $V$. Let
\[ G=G(V):=\textrm{Isom}(V)^{\circ}:= \{T \in \Aut_E(V)\ | \ \langle Tv, Tw \rangle= \langle v,w \rangle ,\  v,w \in V\}^{\circ}.\]
To be explicit, if $E=F$, $n$ is even, and $\epsilon=1,$ then $G=\Sp(V)$ is a symplectic group. If $E=F$ and $\epsilon=1$, then $G=\SO(V)$ is a special orthogonal group. If $E\neq F$, then $G=\RU(V)$ is a unitary group. Fixing the $\epsilon$-Hermitian space $V$ and write $V=V_{an,\mathfrak{r}}$ and let $G=G(V)$.
We shall call $G_r:= G(V_{an,r})$ a group of the same type as $G$. Note that the index $r$ gives the $F$-rank of $G_r$. We set $G(V)$ to be the trivial group if $\dim(V)=0$.

There is a surjection from the conjugacy classes of parabolic subgroups of $G=G(V_{an, \mathfrak{r}})$ to the union of ordered partitions of $r \leq \mathfrak{r}$. The surjection is also an injection unless $E=F$, $\epsilon=1$ and $n$ is even, in which case the fiber of the surjection may have cardinality 2. Let $\underline{r}=(n_1,\ldots, n_f)$ be an ordered partition of $|\underline{r}| \leq \mathfrak{r}$ and let $P_{\underline{r}}$ be in the fiber of $\underline{r}$ under the surjection. The Levi subgroup $M_{\underline{r}}$ of $P_{\underline{r}}$ is isomorphic to
\[ \GL_{n_1}(E) \times \cdots \times \GL_{n_f}(E) \times G(V_{an, \mathfrak{r}-|\underline{r}| }).\]
If $\tau_i$'s are representations of $\GL_{n_i}(E)$ and $\sigma$ is a representation of $ G(V_{an, \mathfrak{r}-|\underline{r}| })$, then we regard $ \pi:=\tau_1 \otimes \cdots \otimes \tau_f \otimes \sigma$ as a representation of $M_{\underline{r}}$ and denote the normalized parabolic induction $\Ind_{M_{\underline{r}}}^G \pi$ by 
\[ \tau_1 \times \cdots \times \tau_f \rtimes \sigma.\]

Finally, we recall the definition of Aubert-Zelevinsky involution. Let $\RG$ be a connected reductive algebraic group defined over $F$, $G=\RG(F)$, and let $\mathcal{R}(G)$ be the Grothendieck group of smooth representations of finite length of $G$. If $\pi$ is a smooth representation of finite length of $G$, we let $[\pi]$ denote its image in $\mathcal{R}(G)$. If $P$ is a parabolic subgroup of $G$, we let $\mathrm{Ind}_{P}^G$ denote the normalized parabolic induction and let $\mathrm{Jac}_P$ denote the Jacquet module.

Aubert (\cite{Aub95}) showed that for any representation $\pi$ of $\Pi(G)$, there exists $\varepsilon\in\{\pm 1\}$ and an irreducible representation $\widehat{\pi} \in \Pi(G)$ such that
\begin{align*}
[\widehat{\pi}]:=\varepsilon\sum_P (-1)^{\mathrm{dim}(A_P)}[\mathrm{Ind}_{P}^{G} \circ \mathrm{Jac}_P(\pi)].
\end{align*}
 Here the sum is taken over all standard parabolic subgroups $P$ of $G$ and $A_P$ is the maximal split torus of the center of the Levi subgroup of $P.$ Moreover, the map $\pi \mapsto \widehat{\pi}$ is an involution on $\Pi(G)$. We call $\widehat{\pi}$ the Aubert-Zelevinsky involution of $\pi.$

\section{Wavefront sets and partitions}
\label{wfs}
Let $\RG$ be a connected reductive group defined over $F$ and $G=\RG(F)$. Let $\pi$ be an irreducible admissible representation of $G$ and let $\Theta(\pi)$ be the character of $\pi.$ The Harish-Chandra-Howe local expansion states that, for a sufficiently small neighborhood of $0$ in $\mathfrak{g}(F)$, $\Theta(\pi)$ is a linear combination of Fourier transform of nilpotent orbits (\cite{HC78}). That is,
\begin{equation}\label{eqn Harish-Chandra-Howe local char exp}
(\Theta(\pi))(\exp(X))=\sum_{\mathcal{O}}c_\mathcal{O}(\pi)\hat{\mu}_{\mathcal{O}}(X),    
\end{equation}
where the sum is over the nilpotent orbits in $\mathfrak{g}(F)$ and $X$ is a regular element of $\mathfrak{g}(F)$ and lies in a sufficiently small neighborhood of $0.$

Recall that $\mathfrak{n}^m(\pi)$ is the set of maximal $F$-rational nilpotent orbits $\CO$ in the Lie algebra $\mathfrak{g}(F)$ of $G$ such that the coefficient $c_{\CO}(\pi)$ in \eqref{eqn Harish-Chandra-Howe local char exp} is nonzero (\cite{HC78, MW87}) and  $\ol{\mathfrak{n}}^m(\pi)$ is the set of corresponding nilpotent orbits over $\ol{F}$.
The set $\ol{\mathfrak{n}}^m(\pi)$ is called the (geometric) wavefront set of $\pi$. 

Let $Q = MN$ be a parabolic
subgroup of $G$. Let $\mathfrak{m}(F)$, $\mathfrak{n}(F),$ and $\mathfrak{q}(F)$ be the Lie algebra of $M, N$ and $Q$ respectively. We also let $\CO$ be a nilpotent orbit in $\mathfrak{m}(F)$. Then there exists a unique nilpotent orbit, denoted by $\Ind_{\mathfrak{q}}^{\mathfrak{g}} \OO$, such that $\Ind_{\mathfrak{q}}^{\mathfrak{g}} \OO \cap (\CO+\mathfrak{n}(F))$ is an open dense set in $\CO+\mathfrak{n}(F)$. This orbit is connected to the wavefront sets of induced representations as follows.

\begin{prop} [{\cite[Section II.1.3]{MW87}}]\label{prop induced orbit}
Let $G$ be a reductive group defined over a non-Archimedean local field $F$, and $Q = MN$ be a parabolic
subgroup of $G$. Let $\delta$ be an irreducible admissible representation of $M$.
Then
\[\overline{\mathfrak{n}}^{m}(\Ind_{Q}^{G} \delta)= \{ \Ind_{\mathfrak{q}}^{\mathfrak{g}} \OO \ | \ \OO \in \overline{\mathfrak{n}}^{m}(\delta) \},\]
where $\mathfrak{q}$ and $\mathfrak{g}$ are the Lie algebras of $Q$ and $G$, respectively.
\end{prop}
{
\begin{remark}\label{rmk MW87 ind p=2}
   Though \cite{MW87} assumes that the residual characteristic is not $2$ throughout the paper, the proof of the above results does not require this assumption.
\end{remark}
}

Now we restrict to the groups $G$ considered in \S \ref{sec groups}. We have
\begin{align*}
    \textrm{Lie}(\GL_m(A))(\overline{F})= \mathfrak{gl}_{md_A}(\overline{F})&,\ \textrm{Lie}(\SO(V))(\overline{F})= \mathfrak{so}_{\dim(V)}(\overline{F}),\\ \textrm{Lie}(\Sp(V))(\overline{F})= \mathfrak{sp}_{\dim(V)}(\overline{F})&,\ \textrm{Lie}(\RU(V))(\overline{F})= \mathfrak{gl}_{\dim(V)}(\overline{F}).
\end{align*}
The set of nilpotent orbits of $\mathfrak{g}(\overline{F})$ surjects to certain collection of partitions, which we describe explicit now.



First, we introduce several notations on partitions. We denote the set of partitions of $n$ by $\mathcal{P}(n)$. We shall express a partition $\underline{p} \in \mathcal{P}(n)$ in one of the following forms.
\begin{enumerate}
    \item [(i)] $\underline{p}=[p_1,\dots, p_N]$, such that $p_i$'s are non-increasing and $\sum_{i=1}^N p_i= n$. We assume $p_i>0$ unless specified. We denote the the \emph{length} of $\underline{p}$ by $l(\underline{p})= |\{1 \leq i \leq N \ | \ p_i >0\}|$.
    \item [(ii)] $\underline{p}=[p_1^{r_1},\dots, p_N^{r_N}]$, such that $p_i$'s are decreasing and $\sum_{i=1}^N r_i p_i=n$. We assume $r_i>0$ unless specified.
\end{enumerate}
For a partition $\underline{p} \in \mathcal{P}(n)$, we denote $|\underline{p}|=n$. Let $\geq$ denote the dominance order on $\mathcal{P}(n)$. That is, if $\underline{p}=[p_1,\ldots, p_r], \underline{q}=[q_1,\ldots, q_s] \in \mathcal{P}(n)$, then $\underline{p} \geq \underline{q}$ if $\sum_{i=1}^k p_i \geq \sum_{i=1}^k q_i$ for any $1 \leq k \leq r$. If $\Sigma$ is a subset of $\mathcal{P}(n)$ and $\underline{p}\in \mathcal{P}(n)$, we write 
\[ \Sigma \leq \underline{p},\]
if $\underline{q} \leq \underline{p}$ for any $\underline{q} \in \Sigma$.

We say a partition is of type $B$, $C$ and $D$ according to the following definition.
\begin{defn}
For $\epsilon\in \{ \pm 1\}$,  define
\[ \mathcal{P}_{\epsilon}(n)= \{ [p_1^{r_1},\dots, p_N^{r_N}]\in \mathcal{P}(n) \ | \ r_i \text{ is even for all }p_i \text{ with } (-1)^{p_i}=\epsilon \}. \]
Then we say
\begin{enumerate}
    \item $\underline{p}\in \mathcal{P}(n)$ is of type $B$ if $n$ is odd and $\underline{p}\in \mathcal{P}_{1}(n)$.
    \item $\underline{p}\in \mathcal{P}(n)$ is of type $C$ if $n$ is even and $\underline{p}\in \mathcal{P}_{-1}(n)$.
    \item $\underline{p}\in \mathcal{P}(n)$ is of type $D$ if $n$ is even and $\underline{p}\in \mathcal{P}_{1}(n)$.
\end{enumerate}
For $X\in \{B,C,D\},$ let $\mathcal{P}_X(n)$ denote the set of partitions of $n$ of type $X$ and let $\mathcal{P}_A(n):= \mathcal{P}(n)$.
\end{defn}

Denote the set of nilpotent orbits of $\mathfrak{gl}_n(\overline{F}), \mathfrak{so}_{2n+1}(\overline{F}), \mathfrak{sp}_{2n}(\overline{F}),\mathfrak{so}_{2n}(\overline{F})$ by $\mathcal{N}_A(n),\mathcal{N}_B(n),$ $\mathcal{N}_C(n)$ and $\mathcal{N}_D(n)$ respectively. Also, for $X \in \{A,B,C,D\}$, let
 \[\mathcal{N}_X =\bigcup_{n \geq 0} \mathcal{N}_{X}(n).  \]
For $(X,N)\in \{(A,n) ,(B,2n+1 ),(C,2n ), (D,2n ) \}$, there is a surjection
\begin{align*}
    \mathcal{N}_X(n) & \longrightarrow \mathcal{P}_{X}(N),\\
    \OO& \longmapsto \underline{p}(\OO).
\end{align*}
The fiber of $\underline{p}= [p_1^{m_1},\dots,p_r^{m_r}] \in \mathcal{P}_X(N)$ under this map is a singleton, which we denote by $\{\OO_{\underline{p}}\}$, except when $X=D$ and $\underline{p}$ is ``very even"; i.e., $p_i$'s are all even. In this case, the fiber consists of two nilpotent orbits, which we denote by $\OO_{\underline{p}}^{I}$ and $\OO_{\underline{p}}^{II}$.

The surjection $\OO \mapsto \underline{p}(\OO)$ carries the closure ordering on $\mathcal{N}_X(n)$ to the dominance order on $\mathcal{P}_X(N)$ in the sense that $\OO > \OO'$ if and only if $\underline{p}(\OO)> \underline{p}(\OO')$. Note that when $\underline{p}$ is very even, $\OO_{\underline{p}}^{I}$ and $\OO_{\underline{p}}^{II}$ are not comparable.

Next, we would like to describe the induced orbit $\Ind_{\mathfrak{q}}^{\mathfrak{g}} \OO$ in Proposition \ref{prop induced orbit} in terms of partitions. Thus, we need the following operation on partitions.

\begin{defn} Suppose $\underline{p} \in \mathcal{P}(n_1)$ and $\underline{q} \in \mathcal{P}(n_2)$. 
\begin{enumerate}
    \item [(i)] Write $\underline{p}=[p_1^{r_1},\dots, p_N^{r_N}]$ and $\underline{q}=[p_1^{s_1},\dots, p_N^{s_N}]$, where we allow $r_i=0$ or $s_i=0$. Then we define
    \[ \underline{p}\sqcup \underline{q}= [p_1^{r_1+s_1},\dots, p_N^{r_N+s_N}]\in \mathcal{P}(n_1+n_2). \]
    \item [(ii)] Write $\underline{p}=[p_1,\dots, p_N]$ and $\underline{q}=[q_1,\dots, q_N]$, where we allow $p_i=0$ or $q_i=0$. Then we define
    \[ \underline{p}+ \underline{q}= [p_1+q_1,\dots, p_N+q_N]\in \mathcal{P}(n_1+n_2). \]
    \item [(iii)] Write $\underline{p}=[p_1,\dots, p_N]$. We define
    \begin{align*}
    \underline{p}^+&= [p_1+1,p_2,\dots, p_N] \in \mathcal{P}(n_1+1),\\
        \underline{p}^-&= [p_1,\dots, p_{N-1}, p_N-1] \in \mathcal{P}(n_1-1).
\end{align*}
\end{enumerate}
\end{defn}

Let $n$ be a positive integer and let $X=B$ if $n$ is odd and $X\in \{C,D\}$ if $n$ is even. For any $\underline{p}\in \mathcal{P}(n)$, there exists a unique maximal partition $\underline{p}_X\in \mathcal{P}(n)$ of type $X$ such that $\underline{p}_X \leq \underline{p}$. We call $\underline{p}_X$ the \emph{$X$-collapse} of $\underline{p}$. For the computation of collapse, see \S \ref{sec collapse}. By convention, we let $\underline{p}_A:= \underline{p}$. Now we can describe the induced orbit in terms of partitions.

\begin{prop}[{\cite[\S 7]{CM93}}]\label{prop computation of induced orbit}
Let $\mathfrak{q}, \mathfrak{q}_1, \mathfrak{q}_2$ be parabolic subalgebras of $\mathfrak{g}$ with Levi subalgebras $\mathfrak{m},\mathfrak{m}_1, \mathfrak{m}_2$.
\begin{enumerate}
    \item Suppose $\mathfrak{q}_1\subseteq \mathfrak{q}_2$. Then for any nilpotent orbit $\OO$ of $\mathfrak{m}_1(\overline{F})$, we have
    \[ \Ind_{\mathfrak{p}_1}^\mathfrak{g} \OO=\Ind_{\mathfrak{p}_2}^\mathfrak{g}(\Ind_{\mathfrak{p}_1}^{\mathfrak{p}_2} \OO)\]
     \item Suppose $(\mathfrak{g}, \mathfrak{m})=(\mathfrak{gl}_{n_1+n_2}, \mathfrak{gl}_{n_1} \oplus \mathfrak{gl}_{n_2} )$. Write a nilpotent orbit $\OO$ of $\mathfrak{m}(\overline{F})$ as $\OO_{1}\oplus \OO_{2}$, where $\OO_{i}\in \mathcal{N}_A(n_i)$. Then 
    \[ \underline{p}(\Ind_{\mathfrak{q}}^{\mathfrak{g}} \OO)= \underline{p}(\OO_1)+\underline{p}(\OO_2). \]
     \item Let $(\mathfrak{g}_n, X)\in \{(\mathfrak{so}_{2n+1},B),(\mathfrak{sp}_{2n},C),(\mathfrak{so}_{2n},D)\}$. Suppose 
    $(\mathfrak{g}, \mathfrak{m})=(\mathfrak{g}_{n_1+n_2}, \mathfrak{gl}_{n_1} \oplus \mathfrak{g}_{n_2})$. Write a nilpotent orbit $\OO$ of $\mathfrak{m}(\overline{F})$ as $\OO_{1}\oplus \OO_{2}$, where $\OO_{1}\in \mathcal{N}_A(n_1)$ and $ \OO_2 \in \mathcal{N}_{X}(n_2)$. Then 
    \[ \underline{p}(\Ind_{\mathfrak{q}}^{\mathfrak{g}} \OO)= (\underline{p}(\OO_1)+\underline{p}(\OO_1)+\underline{p}(\OO_2) )_X. \]
\end{enumerate}
\end{prop}

Finally, we recall the definition of Barbasch-Vogan dual of partitions of type $X$ following \cite{Spa82, Lus84, BV85, Ach03}. Let $\RG$ be a reductive algebraic group and $\widehat{\RG}$ be its dual group. The Barbasch-Vogan duality is a map that sends a nilpotent orbit $\OO$ of $\textrm{Lie}(G)(\overline{F})$ to a nilpotent orbit $d_{BV}(\OO)$ of $\textrm{Lie}(\widehat{G})(\overline{F})$. Thus, for $(X,X')\in \{(A,A),(B,C),(C,B),(D,D)\}$, this induces a map from $\mathcal{P}_{X}$ to $\mathcal{P}_{X'}$ by
\[ d_{BV}(\underline{p}(\OO)):= \underline{p}(d_{BV}(\OO))\]
for any $\OO \in \mathcal{N}_X$. To describe this map in terms of partition, we recall the definition of transpose (or conjugation) of partitions.

\begin{defn}\label{def transpose}
For $\underline{p}=[p_1,\dots, p_N]\in \mathcal{P}(n)$, we define $\underline{p}^{\ast}=[p_1^{\ast}, \dots, p_{N'}^{\ast}]\in \mathcal{P}(n)$ by
\[ p_i^{\ast}=|\{ j \ | \ p_j \geq i   \}|. \]
\end{defn}

Now we describe the Barbasch-Vogan duality map on partitions case by case.

\begin{defn} 
\begin{enumerate}
\item[(i)] For $\underline{p}\in \mathcal{P}_A(2n+1)$, we define $d_{BV}(\underline{p}):= \underline{p}^{\ast}$.
    \item [(ii)]For $\underline{p}\in \mathcal{P}_B(2n+1)$, we define $d_{BV}(\underline{p}):= ((\underline{p}^{-})\underline{\vphantom{p}}_C)^{\ast}$, which is in $\mathcal{P}_C(2n)$.
    \item [(iii)]For $\underline{p}\in \mathcal{P}_C(2n)$, we define $d_{BV}(\underline{p}):= ((\underline{p}^{+})\underline{\vphantom{p}}_B)^{\ast}$, which is in $\mathcal{P}_B(2n+1)$.
    \item [(iv)] For $\underline{p}\in \mathcal{P}_D(2n)$, we define $d_{BV}(\underline{p}):= (\underline{p}^{\ast})\underline{\vphantom{p}}_D$, which is in $\mathcal{P}_D(2n)$.
\end{enumerate}
\end{defn}

We need the following properties of the Barbasch-Vogan duality.

\begin{prop}\label{prop dBV}
Let $(X,N)\in \{(A,n),(B,2n+1),(C,2n),(D,2n)\}$ and $\underline{p}, \underline{q} \in \mathcal{P}_X(N)$.
\begin{enumerate}
    \item If $\underline{p} \geq \underline{q}$, then $d_{BV}(\underline{p}) \leq d_{BV}(\underline{q})$.
    \item We have $d_{BV}^3(\underline{p})= d_{BV}(\underline{p})$.
\end{enumerate}
\end{prop}

Remark that the Barbasch-Vogan duality is not an injection unless $X=A$. See \cite{LLS23} for a study of the fiber of $d_{BV}$. 

Finally, it is stated in \cite[Proposition A.2(c)]{BV85}, with a sketch of the proof, that the Barbasch-Vogan duality is compatible with induction. This is indeed a crucial observation for the purpose of this paper. Thus, we state it explicitly in terms of partitions below, and give a complete combinatorial proof in \S \ref{sec proof of key lemma}.

\begin{lemma}\label{lem goal}
Suppose $(X,X')\in \{ (B,C),(C,B),(D,D) \}$. Let $\underline{p}$ be a partition of type $X$ and $b,d$ be positive integers. Then the following equality holds
\begin{align}\label{eq goal partition}
    d_{BV}( [b^{2d}]\sqcup \underline{p}) = ([(2d)^{b}] +d_{BV}(\underline{p}))_{X'}. 
\end{align}
\end{lemma}
Note that the analogue of \eqref{eq goal partition} for $(X,X')=(A,A)$ is $(\underline{p} \sqcup \underline{q})^{\ast}=\underline{p}^{\ast} + \underline{q}^{\ast}$, which is a direct consequence of the definitions.

\section{\texorpdfstring{Local $L$-parameters and local Arthur parameters}{}}
\label{lp and lap}
In this section, we recall the notation of $L$-parameters and local Arthur parameters of the groups we consider, and their decompositions. Then we define the nilpotent orbits and partitions associated to them.

Let $\RG$ be a connected reductive algebraic group defined over $F$ and $G=\RG(F)$. The $L$-group of $\RG$ is given by ${}^L G :=\widehat{G} \rtimes W_F$. The action of $W_F$ on $\widehat{G}$ factors through the quotient $W_F/W_K \cong \Gal(K/F)$, where $K$ is the splitting field of $G$. Thus, we replace ${}^L G $ by $ \widehat{G} \rtimes \Gal(K/F)$ occasionally. Following \cite[\S 7]{GGP12}, we identify the $L$-groups for inner forms of $\GL_n(F)$ and pure inner form of quasi-split classical groups $G(V)$ as follows. Here $\disc(V)$ is the discriminant of $V$. 

\begin{center}
    \begin{tabular}{|c|c|c|c|c|}
     \hline 
     \rule{0pt}{3ex} $(E,\epsilon, \dim_E(V))$&  $G$& $\widehat{G}$  & ${}^L G$\\
     \hline 
      \rule{0pt}{3ex} &  $\GL_m(A)$& $\GL_{md_A}(\BC)$  & $ \GL_{md_A}(\BC)$\\
     \hline 
    \rule{0pt}{3ex}    $(E=F, 1, 2n+1)$& $\SO(V)$ & $\Sp_{2n}(\BC)$ & $\Sp_{2n}(\BC)$  \\
     \hline
 \rule{0pt}{3ex}     $(E=F,1, 2n)$& $\SO(V)$ & $\SO_{2n}(\BC)$ & $\SO_{2n}(\BC)$ if $\disc(V)\in (F^{\times})^2$ \\
     &  &  &  $\rO_{2n}(\BC)$ if $\disc(V)\not\in (F^{\times})^2$ \\
     \hline
     \rule{0pt}{3ex}     $(E=F,-1, 2n)$& $\Sp(V)$ & $\SO_{2n+1}(\BC)$ & $\SO_{2n+1}(\BC)$ \\
     \hline
      \rule{0pt}{3ex}     $(E\neq F, \pm 1, n)$& $\RU(V)$ & $\GL_{n}(\BC)$&  $\GL_{n}(\BC) \rtimes \Gal(E/F)$    \\
     \hline
\end{tabular}
\end{center}
For classical groups $G=G(V)$, if $E=F$, we fix an embedding $\xi_V:{}^L G \hookrightarrow \GL_N(\BC)$ where $N \in \{2n ,2n+1\}$. If $E \neq F$, we fix an embedding $\xi_V:\widehat{G} \hookrightarrow \GL_N(\BC)$ where $N=n$.

Now we recall the definitions of $L$-parameters and local Arthur parameters of $G$.
\begin{defn}\label{def L-par}
  An $L$-parameter $[\phi]$ of $G$ is a $\widehat{G}(\BC)$-conjugacy class of an admissible homomorphism
 \[ \phi: W_F \times \SL_2(\BC) \to {}^L G.\]
 That is, $\phi$ is continuous, and
 \begin{enumerate}
\item $\phi$ commutes with the projections $W_F \times \SL_2(\BC) \to W_F$ and ${}^L G \to W_F$; 
\item the restriction of $\phi$ to $W_F$ consists of semi-simple elements;

\item the restriction of $\phi$ to $\SL_2(\BC)$ is analytic;

 \item $\phi$ is $G$-relevant. That is, if the image of $\phi$ is contained in the Levi subgroup of some parabolic subgroup ${}^L P$ of ${}^L G$, then $P$ is relevant for $G$ (see \cite[8.2(ii)]{Bor79} for notation).

\end{enumerate}
By abuse of notation, we don't distinguish $[\phi]$ and $\phi$. We let $\Phi(G)$ denote the equivalence class of $L$-parameters of $G$.
\end{defn}

\begin{defn}\label{def A-par}
    A local Arthur parameter $[\psi]$ of $G$ is a $\widehat{G}(\BC)$-conjugacy class of a continuous homomorphism
    \[ \psi: W_F \times \SL_2^D(\BC) \times \SL_2^A(\BC) \to {}^LG, \]
    such that
     \begin{enumerate}
    \item The restriction of $\psi$ to $W_F$ has bounded image;
    \item the restriction of $\psi$ to $\SL_2^D(\BC)$ and $\SL_2^A(\BC)$ are both analytic;
    \item the parameter $\psi$ is $G$-relevant.
\end{enumerate}
By abuse of notation, we don't distinguish $[\psi]$ and $\psi$. We let $\Psi(G)$ denote the equivalence class of local Arthur parameters of $G$.
\end{defn}

Next, we discuss the decompositions of $L$-parameters and local Arthur parameters of classical groups following \cite[\S 8]{GGP12}. Let $\varphi$ be an $L$-parameter of $\GL_N(E)$. Equivalently, we regard $\varphi$ as a representation
\[ \varphi: W_E \times \SL_2(\BC) \to \GL(M)\]
where $M \cong \BC^N$. Take an $s \in W_{E}$ that generates the quotient $W_E/W_F$. We define ${}^s\varphi$, another $L$-parameter of $\GL_N(E)$, from $\varphi$ by  
\[ {}^s \varphi(w,x):= \varphi(sws^{-1},x).\]
The equivalence class $[{}^s \varphi]$ is independent of the choice of $s$, and hence we denote it by ${}^\sigma \varphi$. We say $\varphi$ is $\sigma$-selfdual if $\varphi$ is equivalent to ${}^\sigma \varphi^{\vee}$. Equivalently, $\varphi$ is $\sigma$-selfdual if there exists a non-degenerate bilinear form $B$ on $M$ and $b(\varphi) \in \{\pm 1\}$ such that for any $m_1,m_2 \in M$, $w \in W_E$ and $x \in \SL_2(\BC)$,
\begin{align}\label{eq conjugate self-dual}
\begin{cases}
    B( \varphi(w,x) m_1, {}^s\varphi(w,x) m_2 )= B(m_1,m_2),\\
    B(m_1, m_2)= b(\varphi) B(m_2, \varphi(s^2,1) m_1).
\end{cases}
\end{align}
We call $b(\varphi)$ the sign of $\varphi$.

Now let $G=G(V)$ and $\phi \in \Phi(G)$. We associate an $L$-parameter $\phi_{\GL}$ of $\GL_N(E)$ by
\[ \phi_{\GL}:=  \xi_V \circ \phi|_{W_E \times \SL_2(\BC)}.\]
The map $\phi \mapsto \phi_{\GL}$ gives a surjection from $\Phi(G)$ onto a subset $\Phi(\GL_N(E))_V$ of $\Phi(\GL_N(E))$ consisting of $L$-parameters $\varphi$ with following conditions.
\begin{enumerate}
    \item [(i)] $\varphi$ is  $\sigma$-selfdual with sign $b(\varphi)=\widehat{\epsilon_V}$, where 
\begin{align*}
    \widehat{\epsilon_V}=\begin{cases}
        (-1)^{\dim(V)} & \text{ if }E=F,\\
        (-1)^{\dim(V)+1} & \text{ if }E\neq F.
    \end{cases}
\end{align*}
\item [(ii)] If $G(V)= \Sp(V)$, then $\det(\varphi)=1$. If $G(V)= \SO(V)$ and $\dim(V)$ is even, then the quadratic character $\det(\varphi)$ corresponds to the square class $\disc(V)$. 
\end{enumerate}
Moreover, this map is an injection unless $G(V)=\SO(V)$ and $\dim(V)$ is even, in which case each fiber has cardinality at most 2. See \cite[Theorem 8.1]{GGP12} for a proof of these facts.

Now we decompose $\phi_{\GL}$ into a direct sum of irreducible representations of $W_E \times \SL_{2}(\BC)$. Due to the $\sigma$-selfduality of $\phi_{\GL}$, we may write
\begin{align}\label{eq decomp phi}
    \phi_{\GL}= \bigoplus_{i \in I_{=0}} \rho_i \otimes S_{a_i} + \bigoplus_{i \in I_{\neq 0}} (\rho_i|\cdot|^{x_i} \otimes S_{a_i}+ {}^\sigma \rho_i^{\vee}|\cdot|^{-x_i} \otimes S_{a_i}),
\end{align}
where $S_{a_i}$ is the unique $a_i$-dimensional irreducible representation of $\SL_2(\BC)$, each $\rho_i$ is an irreducible representation of $W_E$ with bounded image for $i \in I_{=0} \sqcup I_{\neq 0}$, and $x_i \in \R \setminus \{0\}$ if $i \in I_{\neq 0}$. We say $\phi$ is tempered if $I_{\neq 0}$ is empty. We say $\phi$ is discrete if it is tempered and $\rho_i \otimes S_{a_i}$ is not isomorphic to $\rho_j \otimes S_{a_j}$ for any $i \neq j \in I_{=0}$.

The above discussion works for local Arthur parameters after an obvious modification. For $\psi \in \Psi(G)$, write
\begin{align}\label{eq decomp psi}
    \psi_{\GL}= \bigoplus_{i \in I} \rho_i \otimes S_{a_i}\otimes S_{b_i}.
\end{align}
We say $\psi$ is tempered if $\psi|_{\SL_2^{A}}$ is trivial, or equivalently, $b_i=1$ for any $i \in I$. We say $\psi$ is discrete if $\psi$ is tempered and the decomposition of $\psi_{\GL}$ in \eqref{eq decomp psi} is multiplicity free. Let $\Psi_{temp}(G)$ (resp. $\Psi_{2}(G)$) denote the set of tempered (resp. discrete) local Arthur parameters of $G$.

For each $\psi \in \Psi(G)$, we associate an $L$-parameter $\phi_{\psi}$ by
\begin{align}\label{phi_psi}
    \phi_{\psi}(w,x):= \psi\left( w,x , \begin{pmatrix}
        |w|^{1/2} & \\ & |w|^{-1/2}
    \end{pmatrix}\right).
\end{align}
Due to the boundedness of the image of $\psi|_{W_F}$ in Condition (1) of Definition \ref{def A-par}, the map $\psi \mapsto \phi_{\psi}$ gives an injection from $\Psi(G)$ to $\Phi(G)$. Note that if $\psi$ is tempered, then $\psi=\phi_{\psi} \otimes S_1$, i.e.,
$\phi_{\psi}(w,x)= \psi(w,x,1)$.  We may also associate another local Arthur parameter $\widehat{\psi}$ from $\psi$ by swapping the Arthur-$\SL_2(\BC)$ and Deligne-$\SL_2(\BC)$. Namely, $\widehat{\psi}$ is defined by
\begin{align}\label{eq psi hat}
    \widehat{\psi}(w,x,y) = \psi(w,y,x).
\end{align}

Finally, for each $\phi \in \Phi(G)$ and $\psi \in \Psi(G)$, we define partitions $\underline{p}(\phi)$ and $\underline{p}(\psi)$ as follows.

\begin{defn}\label{def p(phi), p(psi)}
Let $\phi \in \Phi(G)$ and $\psi \in \Psi(G)$ with decompositions \eqref{eq decomp phi} and \eqref{eq decomp psi}. Define
\begin{align*}
    \underline{p}(\phi)&:= \bigsqcup_{i \in I_{=0}} [a_i^{\dim(\rho_i)}] \sqcup \bigsqcup_{i \in I_{\neq 0}} [a_i^{2\dim(\rho_i)}],\\
    \underline{p}(\psi)&:= \bigsqcup_{i \in I_{=0}} [b_i^{a_i\dim(\rho_i)}] \sqcup \bigsqcup_{i \in I_{\neq 0}} [b_i^{2a_i\dim(\rho_i)}].
\end{align*}
\end{defn}

For an $L$-parameter $\phi$ (resp. a local Arthur parameter $\psi$) of general reductive algebraic group $G$, let $\{H,X,Y\}$ be a $\mathfrak{sl}_2$-triple of $\widehat{\mathfrak{g}}(\BC)$ associated to the morphism $\phi|_{\SL_2(\BC)}: \SL_2(\BC) \to \widehat{G}(\BC)$ (resp. $\psi|_{\SL_2^A(\BC)}: \SL_2(\BC) \to \widehat{G}(\BC)$). We may generalize Definition \ref{def p(phi), p(psi)} to general groups by defining $\OO_{\phi}$ (resp. $\OO_{\psi}$) to be the nilpotent orbit of $\widehat{\mathfrak{g}}(\BC)$ containing $X$. Note that when $G$ is one of the groups considered in \S \ref{sec groups}, we have $\underline{p}(\OO_{\phi})= \underline{p}(\phi)$ (resp. $\underline{p}(\OO_{\psi})=\underline{p}(\psi)$). From this view point or from direct computation, one can see that $\underline{p}(\psi)= \underline{p}(\phi_{\widehat{\psi}}).$

\section{Local Langlands Correspondence and Langlands classification}\label{llc and lc}
Let $\RG$ be a connected reductive algebraic group defined over $F$ and let $G=\RG(F)$. The Langlands classification for $\Pi(G)$ (\cite[Theorem 3.5]{Kon03}) and $\Phi(G)$ (\cite{SZ18}) gives canonical bijections
\[ \pi \leftrightarrow (P,\pi_{temp}, \nu), \ \phi \leftrightarrow (P, \phi_{temp}, \nu), \]
where
\begin{enumerate}
\item [$\oldbullet$] $\pi\in \Pi(G)$ and $\phi \in \Phi(G)$;
    \item [$\oldbullet$] $P$ is a parabolic subgroup of $G$ with Levi subgroup $M$;
    \item [$\oldbullet$] $\nu$ is an unramified character of $M$ in a fixed positive Weyl chamber;
    \item [$\oldbullet$] $\pi_{temp} \in \Pi(M)$ and $\phi_{temp} \in \Phi(M)$ are both tempered. 
\end{enumerate}
Under this bijection, $\pi$ is the unique irreducible subrepresentation of $\Ind_{P}^G \pi_{temp} \otimes \nu^{-1}$. We need the following version of local Langlands correspondence.

\begin{conj}\label{conj LLC}
    For any Levi subgroup $M$ of $G$, there is a canonical map
    \begin{align*}
      \textrm{LLC}_{M}:  \Pi(M) &\to \Phi(M)\\
        \pi & \mapsto \phi_{\pi}
    \end{align*}
    such that if $\pi\in \Pi(G)$ and $\pi \leftrightarrow (P, \pi_{temp}, \nu) $, then $\phi_{\pi} \leftrightarrow (P,\phi_{temp}, \nu)$ with $\phi_{temp}= \phi_{\pi_{temp}}$.
\end{conj}

The above conjecture is proved for the groups considered in \S \ref{sec groups}, except for even special orthogonal groups in which case we only have a weaker version. We give a more explicit description for these groups in the following subsections.

\subsection{General linear groups}
The local Langlands correspondence  for $\GL_n(F)$ (\cite{Hen00, HT01, Sch13}) gives a bijection between $\Phi(\GL_n(F))$ and $\Pi(\GL_n(F))$. Under this bijection, an irreducible representation $\rho$ of $W_F$ corresponds to a supercuspidal representation of $\GL_{\dim(\rho)}(F)$, which we also denote by $\rho$. Fixing $\rho$ and $a \in \Z_{>0}$, the parabolic induction 
\[ \rho|\cdot|^{\half{a-1}} \times \rho|\cdot|^{\half{a-3}} \times \cdots \times \rho|\cdot|^{\half{1-a}},\]
has a unique irreducible subrepresentation (resp. quotient), which we denote by $\St(\rho,a)$ (resp. $\Speh(\rho,a)$). The Aubert-Zelevinsky involution sends $\St(\rho,a)$ and $\Speh(\rho,a)$ to each other. Any essentially discrete series representation of $\GL_n(F)$ is of the form $\St(\rho,a)$ with $n=a \dim(\rho)$.

The Langlands classification states that any $\pi \in \Pi(\GL_{n}(F))$ can be realized as the unique irreducible subrepresentations of a parabolic induction
\[ M(\pi)= \St(\rho_1,a_1)|\cdot|^{x_1} \times \cdots \times \St(\rho_f,a_f)|\cdot|^{x_f},\]
where 
\begin{enumerate}
    \item each $\rho_i$ is a unitary supercuspidal representation of $\GL_{d_i}(F)$ and $\sum_{i=1}^f a_i d_i=n$;
    \item each $x_i$ is a real number and $x_1 \leq \cdots \leq x_f$.
\end{enumerate}
We shall call $M(\pi)$ the standard module of $\pi$. In this case, $\phi_{\pi}$, the $L$-parameter corresponding to $\pi$, is given by
\begin{align*}
    \phi_{\pi}= \rho_1 |\cdot|^{x_1} \otimes S_{a_1}+ \cdots+\rho_f |\cdot|^{x_f} \otimes S_{a_f}.
\end{align*}

Next, we consider the case over division algebra $\GL_m(A)$. The Jacquet-Langlands Correspondence (\cite{JL70,Rog83, DKV84, Bad02}) gives a bijection 
\[ \JL: \Pi_2(\GL_{k}(A)) \to \Pi_2(\GL_{k d_A}(F)),\]
where $\Pi_2(G)$ is the set of equivalence classes of essentially discrete series representations of $G$. Thus, by requiring $\textrm{LLC}_{\GL_{k}(A)}(\pi')= \textrm{LLC}_{\GL_{kd_A}(F)}(\JL(\pi'))$ for any $\pi' \in \Pi_2(\GL_k(A))$ and $k \leq m$, this uniquely determines the map $\textrm{LLC}$ for $\GL_m(A)$ from the desiderata in Conjecture \ref{conj LLC}. We give more details for the map $\JL$ and the Langlands classification below.

 If $\rho'$ is a cuspidal representation of $\GL_{k}(A)$, then $\JL(\rho')$ is of the form $\St(\rho,s)$, and we define $s(\rho'):=s$. It is known that $s(\rho')$ always divides $d_A$ and $kd_A$ is the least common multiple of $\dim(\rho)$ and $d_A$. For each $a' \in \Z_{>0}$, the parabolic induction 
\[ \rho'|\cdot|_{A}^{s(\rho') \cdot \half{a'-1}}\times \rho'|\cdot|_{A}^{s(\rho')\cdot \half{a'-3}}   \times \cdots \times \rho'|\cdot|_{A}^{s(\rho')\cdot \half{1-a'}}\]
has a unique irreducible subrepresentation (resp. quotient), which we denote by $\St(\rho',a')$ (resp. $\Speh(\rho',a')$). The representation $\St(\rho',a')$ is essentially discrete series, and any essentially discrete series representations of $\GL_{k}(A)$ is of this form. If $\JL(\rho')= \St(\rho, s(\rho'))$, then $\JL( \St(\rho', a') )= \St(\rho, s(\rho')a')$.
The representation $\Speh(\rho',a')$ is the Aubert-Zelevinsky involution of $\St(\rho',a')$. If $\rho'$ is unitary, both $\St(\rho', a')$ and $\Speh(\rho',a')$ are unitary.

The Langlands classification for $\GL_m(A)$ states that any $\pi' \in \Pi(\GL_{m}(A))$ can be realized as the unique irreducible subrepresentations of a parabolic induction
\[ M(\pi')= \St(\rho_1',a_1')|\cdot|_A^{x_1} \times \cdots \times \St(\rho_f',a_f')|\cdot|_A^{x_f},\]
where 
\begin{enumerate}
    \item each $\rho_i'$ is a unitary cuspidal representation of $\GL_{d_i}(A)$ and $\sum_{i=1}^f a_i' d_i=m$;
    \item each $x_i$ is a real number and $x_1 \leq \cdots \leq x_f$.
\end{enumerate}
We shall call $M(\pi')$ the standard module of $\pi'$. In this case, write $\JL(\rho_i')= \rho_i \otimes S_{s(\rho_i')}$. Then $\phi_{\pi'}$, the $L$-parameter corresponding to $\pi'$, is given by
\begin{align*}
    \phi_{\pi}= \rho_1 |\cdot|^{x_1} \otimes S_{a_1' s(\rho_1')}+ \cdots+\rho_f |\cdot|^{x_f} \otimes S_{a_f' s(\rho_f')}.
\end{align*}

We may regard $\Phi(\GL_m(A))$ as a subset of $\Phi(\GL_{md_A}(F))$ that is $\GL_m(A)$-relevant. Equivalently, an $L$-parameter $\phi \in \Phi(\GL_{md_A}(F))$ lies in $\Phi(\GL_m(A))$ if one of the following equivalent conditions hold.
\begin{enumerate}
    \item [(i)] Every irreducible subrepresentation of $\phi$ has dimension divisible by $d_A$.
    \item [(ii)] $\phi=\phi_{\pi}$ for some $\pi \in \Pi(\GL_m(A))$.
\end{enumerate}

\subsection{Classical groups}
 For quasi-split classical groups that are not even special orthogonal groups, the local Langlands correspondence has been established by \cite{Art13, Mok15}. The extension of the local Langlands correspondence to pure inner forms, as conjectured by Vogan in \cite{Vog93}, is also established in \cite{Art13,KMSW14, MR18, Ish23}. In particular, the map
  \begin{align*}
      \textrm{LLC}:  \Pi(G) &\to \Phi(G)\\
        \pi & \mapsto \phi_{\pi}
    \end{align*}
    is well-defined. If $G(V)$ is an even special orthogonal group, currently we only have a weaker version
    \[ \textrm{WLLC}: (\Pi(G)/\sim) \to (\Phi(G)/\sim),\]
    where the equivalence relation $\sim$ is defined as follows:
    \begin{enumerate}
        \item [$\oldbullet$] Fix an $\epsilon \in \textrm{Isom}(V)\setminus G(V)$. For $\pi \in \Pi(G)$, let $\pi^{\epsilon}(g):= \pi(\epsilon g \epsilon^{-1})$. Then we define $\pi \sim \pi'$ if $\pi'\cong \pi$ or $\pi' \cong \pi^{\epsilon}$.
        \item [$\oldbullet$] We define $\phi_1 \sim \phi_2$ if $(\phi_{1})_{\GL}$ is equivalent to $(\phi_2)_{\GL}$ as $L$-parameters of $\GL_{\dim(V)}(F)$.
    \end{enumerate}
    See \cite[\S 3.5]{AG17} for more details. In the rest of this paper, if $G=G(V)$ is an even special orthogonal group, then a representation $\pi$ and an $L$-parameter $\phi$ of $G$ is understood by their equivalence class in $\Pi(G)/\sim$ and  $\Phi(G)/\sim$. This will not affect the main results of this paper. See Remark \ref{rmk even orthogonal}.

Now we give more details on $\textrm{LLC}$ or $\textrm{WLLC}$ for $G=G(V)$. Write $V=V_{an, \mathfrak{r}}$. The Langlands classification for $G$ states that any $\pi \in \Pi(G)$ can be realized as the unique irreducible subrepresentations of a parabolic induction
\begin{align*}
    M(\pi)= \St(\rho_1,a_1)|\cdot|^{x_1} \times \cdots \times \St(\rho_f, a_f) |\cdot|^{x_f} \rtimes \pi_{temp},
\end{align*}
where
\begin{enumerate}
    \item each $\rho_i$ is a unitary cuspidal representation of $\GL_{d_i}(E)$, $\pi_{temp}$ is a tempered representation of $G(V_{an,r})$ and $r+\sum_{i=1}^f a_i d_i = \mathfrak{r}$;
    \item each $x_i$ is a real number and $x_1 \leq \cdots \leq x_f<0$.
\end{enumerate}
We shall call $M(\pi)$ the standard module of $\pi$. In this case, we have
\begin{align*}
    (\phi_{\pi})_{\GL}= (\phi_{\pi_{temp}}
)_{\GL}+ \bigoplus_{i =1}^f (\rho_i|\cdot|^{x_i} \otimes S_{a_i}+{}^\sigma \rho_i^{\vee}|\cdot|^{-x_i} \otimes S_{a_i}).
\end{align*}

\subsection{Closure ordering}
At the end of this section, we recall a partial ordering $\geq_C$ on $\Phi(G)$, called the closure ordering. We refer the readers to \cite[\S 4]{CFMMX22} for more details. Let $\phi \in \Phi(G)$. We associate a (conjugacy class of) homomorphism $\lambda_{\phi}:W_F \to {}^L G$ by
\[ \lambda_{\phi}(w):= \phi\left( w, \begin{pmatrix}
    |w|^{1/2} & \\ & |w|^{-1/2}
\end{pmatrix} \right), \]
which is an infinitesimal parameter of $G$. Conversely, fixing an infinitesimal parameter $\lambda $ of $G$, consider the set 
\[ \Phi(G)_{\lambda}:= \{ \phi \in \Phi(G)\ | \ \lambda_\phi=\lambda\},\]
which is finite and in bijection with the set of orbits of the Vogan variety $V_{\lambda}$ associated to $\lambda$. We denote the bijection by $\phi \mapsto C_{\phi}$ and define the closure ordering $\geq_C$ on $\Phi(G)$ by 
\[ \phi_1 \geq_C \phi_2 \text{ if } \lambda_{\phi_1}= \lambda_{\phi_2} \text{ and } \overline{C_{\phi_1}} \supseteq \overline{C_{\phi_2}}.
\] 
Fixing any infinitesimal parameter $\lambda$ of $G$, the Vogan variety $V_{\lambda}$ has a unique open (resp. closed) orbit $C^0$ (resp. $C_0$), and we say the corresponding $L$-parameter is open (resp. closed). This gives the unique maximal (resp. minimal) element of $\Phi(G)_{\lambda}$ under $\geq_C$.

For pure inner forms of classical groups, an equivalent definition of the closure ordering is that $\phi_1 \gneq_C \phi_2$ if $\phi_1 \neq \phi_2$ and $\pi_{(\phi_{2})_{\GL}}$ is a subquotient of $M( \pi_{(\phi_{1})_{\GL}})$, where $\pi_{(\phi_i)_{\GL}}$ is the unique irreducible representation in the $L$-packet $\Pi_{(\phi_i)_{\GL}}$ of $\GL_N(E)$ (see \cite[10.2.1]{CFMMX22}, \cite{Zel80}). Also in this case, the closure ordering implies the dominance ordering for the associated partitions.

\begin{prop}[{\cite[Corollary 5.10(2)]{HLLZ22}}]\label{prop C implies P}
Let $\phi_1, \phi_2 \in \Phi(G(V))$. If $\phi_1 >_C \phi_2$, then $\underline{p}(\phi_1) > \underline{p}(\phi_2)$.    
\end{prop}

Finally, we need the following closure ordering relation for subquotient of standard modules.

\begin{thm}[{\cite{BW80, Kon03}}]\label{thm closure ordering std module}
    Let $\pi \in \Pi(G)$. If $\pi'$ is a subquotient of $M(\pi)$, then $\phi_{\pi'} \geq_{C} \phi_{\pi}$. Moreover, the equality holds if and only if $\pi'=\pi$.
\end{thm}

\section{Local Arthur packets, ABV-packets, and the Jiang conjecture}\label{lap ABV Jiang}
In this section, we recall the expected properties of local Arthur packets and ABV-packets we need, and state the Jiang conjecture for upper bounds of wavefront set for these packets.

\subsection{Local Arthur packets}\label{lap}
For each local Arthur parameter $\psi \in \Psi(G)$, it is conjectured in \cite[Conjecture 6.1]{Art89} that there is a finite multi-set of irreducible representations $\Pi_{\psi}$, called the local Arthur packets, that should parameterize local components of global discrete automorphic representations. When $G$ is quasi-split, a local characterization of $\Pi_{\psi}$ using the theory of regular and twisted endoscopy is given in \cite[Theorem 2.2.1]{Art13}, and the existence of $\Pi_{\psi}$ is proved in \cite{Art13, Mok15}. For non-quasi-split classical groups $G$, a conjectural local characterization of $\Pi_{\psi}$ of $G$ given in \cite[Conjecture 9.4.2]{Art13} and \cite[Theorem* 1.6.1]{KMSW14} when $G$ is an inner form or a pure inner form of a quasi-split classical group $G^{\ast}$ respectively. Currently, for pure inner forms of classical groups, the existence of local Arthur packets $\Pi_{\psi}$ is only established in the case that $\psi$ is generic. See \cite{KMSW14,MR18, Ish23}. In the following proposition, we assume the existence of local Arthur packet $\Pi_{\psi}$ for any $\psi \in \Psi(G)$, and state the properties of local Arthur packets for pure inner forms of classical groups.

\begin{prop}\label{prop A-packet}
    Let $G=G(V)$ be a pure inner form of classical groups. Then the following holds.
    \begin{enumerate}
        \item [(a)] If $\psi$ is tempered, then $\Pi_{\psi}= \Pi_{\phi_\psi}$ and
        \[ \Pi_{temp}(G)= \bigsqcup_{\psi \in \Psi_{temp}(G)} \Pi_{\psi},\ \ \Pi_{2}(G)= \bigsqcup_{\psi \in \Psi_{2}(G)} \Pi_{\psi}.\]
        \item [(b)] Let $\psi$ be a tempered local Arthur parameter of $G$. If $\psi= \xi_M \circ \psi_M$, where $M$ is a  Levi subgroup $M$ of $G$, $\psi_M$ is a discrete local Arthur parameter of $M$ and $\xi_M: {}^L M \hookrightarrow {}^L G$ is the associated embedding. Then 
        \[ \bigoplus_{\pi \in \Pi_{\psi}}\pi= \bigoplus_{\pi_M \in \Pi_{\psi_M}} \Ind_M^G \pi_M.\]
        \item [(c)] Assume that there is a local Arthur packets theory for $G$ as conjectured in \cite[Conjecture 6.1]{Art89}. Then for any $\psi \in \Psi(G)$, we have $\Pi_{\widehat{\psi}}=\{\widehat{\pi} \ | \ \pi \in \Pi_{\psi}\}$, where $\widehat{\psi}$ is defined in \eqref{eq psi hat}.
    \end{enumerate}
\end{prop}

Parts (a) and (b) are proved for pure inner forms of classical groups in \cite{Art13, Mok15, KMSW14, MR18}. If $\Pi_{\psi}$ and $\Pi_{\widehat{\psi}}$ satisfy the local characterization from endoscopic theory, then Part (c) is a consequence of the compatibility of Aubert-Zelevinsky involution with endoscopic transfer, which is proved in \cite{Hir04} for the regular case and in \cite[\S A]{Xu17b} for the twisted case.

We also need the following Working Hypothesis on the closure ordering for local Arthur packets.

\begin{assumption}[{\cite[Conjecture 2.1]{Xu21b}}]\label{assu closure ordering A-packet}
Let $\RG$ be a connected reductive group over $F$ and let $G=\RG(F)$. Assume that there is a local Arthur packets theory for $G$ as conjectured in \cite[Conjecture 6.1]{Art89}. Then for any $\psi \in \Psi(G)$ and $\pi \in \Pi_{\psi}$, we have
\begin{align*} 
\phi_{\pi} \geq_C \phi_{\psi}.   
\end{align*}
\end{assumption}

This Working Hypothesis is studied in \cite{HLLZ22}. In particular, the Working Hypothesis is verified for symplectic and split odd special orthogonal groups.

\begin{thm}[{\cite[Theorem 1.3]{HLLZ22}}]\label{thm closure ordering A-packet SpSO}
    Working Hypothesis \ref{assu closure ordering A-packet} is verified for $\Sp_{2n}(F)$ and split $\SO_{2n+1}(F)$.
\end{thm}

\subsection{ABV-packets}\label{abv}
In \cite{CFMMX22}, Cunningham, Fiori, Moussaoui, Mracek, and Xu extended the work of \cite{ABV92} to define an ABV-packet $\Pi_\phi^{\mathrm{ABV}}$ over $p$-adic fields using  micro-local vanishing cycle functors, for any $L$-parameter $\phi$ of any pure inner form of quasi-split $p$-adic reductive group. As in the real reductive groups cases, it is expected that $\Pi_\phi^{\mathrm{ABV}}=\Pi_\psi$ if $\phi=\phi_\psi$ for classical group $G$ when $\Pi_\psi$ is defined, see \cite[Section 8.3, Conjecture 1]{CFMMX22}. We recall the properties of $\Pi_{\phi}^{\textrm{ABV}}$ that we need and refer to \cite{CFMMX22} for the proof.
\begin{prop}\label{prop ABV}
    Let $\lambda$ be an infinitesimal parameter of $G$, $\phi \in \Phi(G)_{\lambda}$, and $\pi \in \Pi_{\phi}^{\textrm{ABV}}$. 
    \begin{enumerate}
        \item [(a)] The $L$-packet $\Pi_{\phi}$ is contained in the ABV-packet $\Pi_{\phi}^{\textrm{ABV}}$.
        \item [(b)] We have $\phi_{\pi} \geq_C \phi$. In particular, $ \underline{p}(\phi_{\pi}) \geq \underline{p}(\phi).$
    \end{enumerate}
\end{prop}

We need a further assumption that ABV-packets are also compatible with Aubert-Zelevinsky involution as local Arthur packets. To be explicit, there is a well-defined involution
        \begin{align*}
            \Phi(G_n)_{\lambda} &\to \Phi(G_n)_{\lambda},\\
            \phi& \mapsto \widehat{\phi},
        \end{align*}
called the Pyasetskii involution. We refer to \cite[\S 6.4]{CFMMX22} for precise definition. Note that $ \widehat{\phi_{\psi}}= \phi_{\widehat{\psi}} $ for any local Arthur parameter $\psi$. The following Working Hypothesis is expected (see \cite[\S 10.3.4]{CFMMX22}).
 
\begin{assumption}\label{assu ABV AZ dual}
    If $\pi\in \Pi_{\phi}^{\textrm{ABV}}$, then $\widehat{\pi} \in \Pi_{\widehat{\phi}}^{\textrm{ABV}}$.
\end{assumption}

The Working Hypothesis is proved for $G=\GL_n(F)$ in \cite[Proposition 3.2.1]{CFK22}. 

Finally, we remark that though \cite{CFMMX22} only defined ABV-packets for pure inner forms of quasi-split groups, they suggested a way to generalize for all inner forms in \cite[\S 11.2.7]{CFMMX22}.

\subsection{The Jiang Conjecture}
We say $\pi \in \Pi(G)$ is of Arthur type if $\pi \in \Pi_{\psi}$ for some $\psi \in \Psi(G)$, and let $\Pi_A(G)$ denote the subset of $\Pi(G)$ that consists of representations of Arthur type. For each $\psi \in \Psi(G)$, the Jiang conjecture predicts an upper bound of the wavefront set for representations in $\Pi_{\psi}$.

\begin{conj}[{\cite[Conjecture 4.2]{Jia14}} and {\cite[Conjecture 1.7]{LS23}}]\label{conj Jiang}
Let $\RG$ be a connected reductive group over $F$ and let $G=\RG(F)$. Assume that there is a local Arthur packets theory for $G$ as conjectured in \cite[Conjecture 6.1]{Art89}. Then for any $\psi \in \Psi(G)$, the following holds.
\begin{enumerate}
    \item [(i)] For any $\pi\in \Pi_{\psi}$, we have
    $ \WFN(\pi) \leq d_{BV}(\OO_\psi). $
    \item [(ii)] If $G$ is quasi-split over $F$, there exists at least one member $\pi \in \Pi_{\psi}$ such that $ \overline{\mathfrak{n}}^m(\pi)=\{d_{BV}(\OO_\psi)\}$.
\end{enumerate}
\end{conj}

In this paper, we focus on Part (i) of the above conjecture and specialize to pure inner forms of classical groups or general linear groups. Therefore, we rephrase the statement in terms of partitions as follows.

\begin{conj}\label{conj Jiang 2}
For any $\pi \in \Pi_A(G)$ and any $\underline{p} \in \WFP(\pi)$, we have $$\underline{p} \leq d_{BV}(\underline{p}(\psi)),$$
for any $\psi \in \Psi(\pi):= \{ \psi \in \Psi(G) \ | \ \pi \in \Pi_{\psi}  \}.$
\end{conj}

Though not needed in this paper, it is natural to ask if the collection of the conjectural upper bounds 
\[ \{ d_{BV}(\underline{p}(\psi)) \ | \ \psi \in \Psi(\pi)\} \]
admit a unique minimal partition. For $\Sp_{2n}(F)$ and split $\SO_{2n+1}(F)$, the first three named authors proved that there is a distinguished member $\psi^{min}(\pi) \in \Psi(\pi)$ that gives a minimal partition in the above set in \cite[\S 11.1]{HLL22}. We expect that there are similar phenomenons for other classical groups.

In this paper, we also consider the following generalization of the Shahidi conjecture to ABV-packets.

\begin{conj}\label{conj Jiang ABV}
Let $\RG$ be a connected reductive group over $F$ that has a quasi-split pure inner form and let $G=\RG(F)$. For any $\phi \in \Phi(G)$, the following holds.
\begin{enumerate}
    \item [(i)] For any $\pi\in \Pi_{\phi}^{\textrm{ABV}}$, we have
    $ \WFN(\pi) \leq d_{BV}(\OO_{\widehat{\phi}}). $
    \item [(ii)]  If $G$ is quasi-split over $F$, there exists at least one member $\pi \in \Pi_{\phi}^{\textrm{ABV}}$ such that $ \WFN(\pi)=\{d_{BV}(\OO_{\widehat{\phi}})\}$.
\end{enumerate}
\end{conj}

Again, we shall focus on Part (i) of the above conjecture for the groups considered in \S \ref{sec groups}, which can be rephrased as follows.

\begin{conj}\label{conj Jiang ABV 2}
For any $\pi \in \Pi(G)$ and  for any $ \phi \in \Phi(\pi):= \{ \phi \in \Phi(G) \ | \ \pi \in \Pi_{\phi}^{\textrm{ABV}}  \},$ we have 
\[\WFP(\pi) \leq d_{BV}(\underline{p}(\widehat{\phi})).\]
\end{conj}

If Working Hypothesis \ref{assu ABV AZ dual} holds, then $\widehat{\phi_{\widehat{\pi}}} \in \Phi(\pi)$, so the collection of conjectural upper bounds 
\[ \{ d_{BV}(\underline{p}(\widehat{\phi})) \ | \ \phi \in \Phi(\pi)\} \]
has a unique minimal partition $d_{BV}(\underline{p}(\phi_{\widehat{\pi}}))$ by Proposition \ref{prop ABV}(2).

We end this section by giving some remarks on Part (ii) of Conjectures \ref{conj Jiang}, \ref{conj Jiang ABV}.

\begin{remark} \label{rmk Jiang quasi-split}We focus on Part (ii) of  Conjecture \ref{conj Jiang}. Similar discussion applies to  Conjecture \ref{conj Jiang ABV}(ii).
    \begin{enumerate}
        \item Part (ii) of Conjectures \ref{conj Jiang} fails without the assumption that $G$ is quasi-split. For example, let $\psi$ be a tempered local Arthur parameter of $G$. Then $d_{BV}(\OO_{\psi})$ is the regular orbit. Thus, $  \WFN(\pi)=\{d_{BV}(\OO_{\psi})\} $ if and only if $\pi$ is generic with respect to some Whittaker datum, which can not be the case if $G$ is not quasi-split.
        \item A weaker version of Part (ii) of Conjecture \ref{conj Jiang} is to consider the Vogan version of local Arthur packets (see \cite{GGP20}) as follows. 
        \begin{enumerate}
            \item [(ii)']There exists a 
        \[\pi\in \bigsqcup_{G'} \Pi_{\psi}(G') \]
        such that $\WFN(\pi)=\{d_{BV}(\OO_{\psi})\}$. Here the disjoint union is taken over all pure inner twists $G'$ of $G$ and $\Pi_{\psi}(G')$ is the local Arthur packet of $\psi$ for $G'$ if $\psi$ is $G'$-relevant, and empty set otherwise. 
        \end{enumerate}
        Clearly, Conjecture \ref{conj Jiang}(ii) implies (ii)'. It would be interesting if there exists an example that (ii)' holds but Conjecture \ref{conj Jiang}(ii) fails.
    \end{enumerate} 
\end{remark}

\section{Main results}\label{main results}
In this section, we state the main results of this paper. Let $\RG$ be a connected reductive algebraic group defined over $F$ and let $G= \RG(F)$. Assuming that the local Langlands correspondence is established for $G$, we propose a new conjecture (Conjecture \ref{conj bound of WF}) on upper bounds of all admissible representations for $G.$ We prove the new conjecture for $\GL_m(A)$ (Theorem \ref{thm main GL}), and reduce it to anti-discrete series case for classical groups (Theorem \ref{thm main classical groups}). Then we show the equivalence of the new conjecture and Conjectures \ref{conj Jiang 2}, \ref{conj Jiang ABV 2} under certain assumptions.

Now we state the new conjecture.

\begin{conj}\label{conj bound of WF general}
Assume the local Langlands correspondence for $G$ (Conjecture \ref{conj LLC}). For any $\pi \in \Pi(G)$ and any nilpotent orbit $\OO \in \overline{\mathfrak{n}}^m(\pi)$, the following inequality holds
\[\OO \leq d_{BV}( \OO_{\phi_{\widehat{\pi}} }).\]
\end{conj}

In this paper, we shall focus on the groups described in \S \ref{sec groups}, and work with the following partition version.

\begin{conj}\label{conj bound of WF}
Let $G=\GL_{m}(A)$ or $G(V)$, a pure inner form of classical group. For any $\pi \in \Pi(G)$ and any partition $\underline{p} \in \mathfrak{p}^m(\pi)$, the following inequality holds
\[\underline{p} \leq d_{BV}( \underline{p}(\phi_{\widehat{\pi}} )).\]
\end{conj}

\begin{remark}\label{rmk even orthogonal} 
 For even special orthogonal groups $G=G(V)$, we only have the weak Local Langlands Correspondence 
\[ \textrm{WLLC}: (\Pi(G)/\sim) \to (\phi(G)/\sim).\]
However, this is enough for our purpose since $\mathfrak{p}^{m}(\pi)=\mathfrak{p}^{m}(\pi^{\epsilon})$ and $ \underline{p}(\phi_1)=\underline{p}(\phi_2)$ if $\phi_1 \sim \phi_2$.
\end{remark}

Recall that we say a representation $\pi$ of $G$ is anti-discrete series (resp. essentially anti-discrete series, anti-tempered) if $\widehat{\pi}$ is a discrete series (resp. essentially discrete series, tempered) representation of $G$. Let $\Pi_{\widehat{2}}(G)$ (resp. $\Pi_A(G)$, $\Pi_{\widehat{temp}}(G)$) denote the set of equivalence classes of anti-discrete series representation (resp. Arthur type, anti-tempered) of $G$. We have the following inclusion
\begin{align}\label{eq chain of containment}
    \Pi(G) \supseteq \Pi_A(G) \supseteq \Pi_{\widehat{temp}}(G) \supseteq \Pi_{\widehat{2}}(G).
\end{align}

\subsection{\texorpdfstring{Main results for $\GL_m(A)$}{}}

Our first main result establishes Conjecture \ref{conj bound of WF} for $\GL_m(A)$. 
\begin{thm}\label{thm main GL}
Let $A$ be a central division algebra over $F$ and let $G=\GL_m(A)$. Conjecture \ref{conj bound of WF} holds for $G$. Moreover, when $A=F$, for all $\pi \in \Pi(G)$ we have 
\[\WFN(\pi)=\{d_{BV}(\OO_{\phi_{\widehat{\pi}}})\}.\]

\end{thm}

The case of $\GL_n(F)$ was established by M{\oe}glin and Waldspurger in \cite[\S II.2]{MW87} (whose method naturally extends to the case of even residue character based on the work in \cite{Var14}), where they explicitly computed the wavefront set, which agrees precisely with the upper bound given in Conjecture \ref{conj bound of WF}. Their argument consists of three main steps:
\begin{enumerate}
    \item[Step 1:] For any essentially anti-discrete series representation of a Levi subgroup of $\GL_m(F)$, the wavefront set is computed via Bernstein-Zelevinsky derivatives (see \cite[Theorem 2.2]{Zel81} and \cite[Theorem 7.1]{Zel80}). In this case, the wavefront set achieves the upper bound.
    \item[Step 2:] This upper bound is extended to all irreducible representations by applying Proposition \ref{prop induced orbit} and using the compatibility of parabolic induction with Barbasch-Vogan duality.
    \item[Step 3:] Noting that the Barbasch-Vogan duality map is a bijection in this context and that the upper bound in Step 1 is actually an equality, one deduces the equality of wavefront sets for all irreducible representations.
\end{enumerate}
Further details are provided in \S \ref{sec GLn(A) main}, where we also prove the upper bound for $\GL_m(A)$ for $A \neq F$. It is important to note that when $A \neq F$, the equality of wavefront sets does not hold, as the explicit formula for anti-discrete series representations of $\GL_m(A)$ is not available (see Remark \ref{rmk division algebra}). Nevertheless, the Jacquet-Langlands correspondence yields the desired upper bound for anti-discrete series representations.

As a corollary of Theorem \ref{thm main GL}, we have the following for $\GL_n(F)$ since local Arthur packets are singletons and Working Hypothesis \ref{assu ABV AZ dual} is verified in \cite[Proposition 3.2.1]{CFK22} in this case. A detailed proof will be given in \S \ref{sec GLn(A) main}.

\begin{cor}\label{cor GLn}
    Conjectures \ref{conj Jiang 2} and \ref{conj Jiang ABV} hold for $\GL_n(F)$.
\end{cor}

\subsection{\texorpdfstring{Main results for classical groups}{Main results for classical groups}{}}

 For pure inner forms of classical groups $G(V)$, we could still try the same strategy as for $\GL_m(A)$. However, the wavefront sets of (essentially) anti-discrete series representations in Step 1 are not known, and it may not even be a singleton (\cite{Tsa24}). Nevertheless, the argument in Step 2 still allows us to reduce Conjecture \ref{conj bound of WF} for $G(V)$ to anti-discrete series representations of groups of the same type.


Let us introduce some notation to shorten the statement. Write $V=V_{an,\mathfrak{r}}$. Let 
\[(\Xi, -) \in \{(\Pi, \ref{conj bound of WF}), (\Phi, \ref{conj Jiang ABV 2}), (\Psi, \ref{conj Jiang 2})\}.\]
We consider the collection of statements
\begin{align}\label{eq Xi ast}
    \tag{$\Xi_{\ast}$} \text{Conjecture }- \text{ holds for any }\pi \in \Pi_{\ast}(G(V_{an,r})) \text{ for any }r \leq \mathfrak{r},
\end{align}
where $\ast \in \{\emptyset, A, \widehat{temp}, \widehat{2}\}$. Here $\Xi_{\emptyset}$ is understood as $\Xi$.

Our second main result reduces Conjecture \ref{conj bound of WF} to anti-discrete series representations.

\begin{thm}\label{thm main classical groups}
    Let $V=V_{an, \mathfrak{r}}$. The statements $(\Pi), (\Pi_{A}), (\Pi_{\widehat{temp}}), (\Pi_{\widehat{2}})$ are equivalent.
\end{thm}

Now we compare the conjectural upper bound in the new conjecture and the Jiang conjecture for local Arthur packets (Conjecture \ref{conj Jiang 2}). Observe that assuming Conjecture \ref{conj bound of WF} and Working Hypothesis \ref{assu closure ordering A-packet}, for any $\pi \in \Pi_{\psi}$, we have
\begin{align}\label{eq sharper upper bound}
    \WFP(\pi) \leq d_{BV}(\underline{p}(\phi_{\widehat{\pi}})) \leq d_{BV}(\underline{p}(\phi_{\widehat{\psi}}))=d_{BV}(\underline{p}(\psi))
\end{align}
by Propositions \ref{prop C implies P}, \ref{prop dBV}(1). The second inequality of \eqref{eq sharper upper bound} can be strict as demonstrated in the following example. Thus, Conjecture \ref{conj bound of WF} improves the conjectural upper bounds given in Conjecture \ref{conj Jiang}.

\begin{exmp}\label{exmp strict inequality}
    Let $\pi$ be the representation of $\Sp_{10}(F)$ such that $\widehat{\pi}$ is equal to $\pi_1$ in \cite[Example 10.14]{HLL22}. Let $\rho$ denote the trivial representation of $W_F$.
    The computation there shows that 
    \[ \phi_{\widehat{\pi}}= \rho|\cdot|^{-3} \otimes S_1+ \rho \otimes S_1+ \rho \otimes S_3 + \rho \otimes S_5+ \rho|\cdot|^{3} \otimes S_1,\]
    and hence 
    \[ d_{BV}(\underline{p}(\phi_{\widehat{\pi}}))= d_{BV}([5,3,1,1,1])=[4,2,2,2].\]
    On the other hand, the computation there also shows that $\Psi(\pi)= \{\widehat{\psi}_1,\widehat{\psi}_2,\widehat{\psi}_3\}$, where
\begin{align*}
    \widehat{\psi}_1&= \rho \otimes S_7 \otimes S_1 + \rho \otimes S_2 \otimes S_2,\\
    \widehat{\psi}_2&= \rho \otimes S_7 \otimes S_1 + \rho \otimes S_1 \otimes S_1+ \rho \otimes S_1\otimes S_3,\\
    \widehat{\psi}_3&= \rho \otimes S_7 \otimes S_1 + \rho \otimes S_3 \otimes S_1 +\rho\otimes S_1\otimes S_1.
\end{align*}   
Thus,
\begin{align*}
    d_{BV}(\underline{p}(\widehat{\psi}_1))=[8,2],\ d_{BV}(\underline{p}(\widehat{\psi}_2))=[8,2],\ d_{BV}(\underline{p}(\widehat{\psi}_3))=[10],
\end{align*}
which are all strictly larger than $d_{BV}(\underline{p}(\phi_{\widehat{\pi}})).$ Note that $\pi$ is unipotent with real infinitesimal parameter, so $\WFP(\pi)=\{d_{BV}(\underline{p}(\phi_{\widehat{\pi}}))\}$ if the residue field of $F$ has sufficiently large characteristic by a result of \cite{CMBO23} (see Theorem \ref{thm CMBO} below).
\end{exmp}

  However, in our third main results, we show that Conjectures \ref{conj bound of WF} and \ref{conj Jiang 2} are indeed equivalent assuming Working Hypothesis \ref{assu closure ordering A-packet}.

\begin{thm}\label{thm main Jiang}
    Let $V=V_{an, \mathfrak{r}}$. Suppose that Working Hypothesis \ref{assu closure ordering A-packet} holds for $G(V_{an,r})$ for any $r \leq \mathfrak{r}$. Then the statements $(\Psi_A), (\Psi_{\widehat{temp}}), (\Psi_{\widehat{2}}), (\Pi)$ are equivalent.
\end{thm}

In particular, by Theorem \ref{thm closure ordering A-packet SpSO}, the above theorem holds for $\Sp_{2n}(F)$ and split $\SO_{2n+1}(F)$ without assumptions.

As our fourth main results, we compare the new conjecture and the Jiang conjecture on ABV-packets. In this case, the analogue of Working Hypothesis \ref{assu closure ordering A-packet} is known (Proposition \ref{prop ABV}(b)), while the analogue of Proposition \ref{prop A-packet}(c) is still open. Thus, to show the equivalence between Conjectures \ref{conj bound of WF}, \ref{conj Jiang ABV}, we need to assume Working Hypothesis \ref{assu ABV AZ dual}. 

\begin{thm}\label{thm main Jiang ABV}
    Let $V=V_{an, \mathfrak{r}}$ and $G=G(V)$. Suppose that Working Hypothesis \ref{assu ABV AZ dual} holds for $G(V_{an,r})$ for any $r \leq \mathfrak{r}$. Then the statements $(\Phi), (\Phi_{\widehat{temp}}), (\Phi_{\widehat{2}}), (\Pi)$ are equivalent.
\end{thm}

We prove Theorems \ref{thm main classical groups}, \ref{thm main Jiang} and \ref{thm main Jiang ABV} in \S \ref{sec proof classical}. We recall from the introduction the following diagram describing the relations among the statements $(\Xi_{\ast})$.

\begin{equation}\label{eq diagram}
        \begin{tikzcd}[column sep=huge]
   (\Phi) \ar[dd, Rightarrow] &(\Pi)\ar[l, Leftrightarrow, "\text{(D)}"',"\text{Lemma }\ref{lem (D) (E) (G)}"]  \ar[d, Rightarrow]& \\
    &(\Pi_A) \ar[d, Rightarrow] \ar[r, Rightarrow, "\text{(B)}", "\text{Lemma }\ref{lem (B)}"'] & (\Psi_A) \ar[d, Rightarrow]\\
    (\Phi_{\widehat{temp}}) \ar[d, Rightarrow]&(\Pi_{\widehat{temp}}) \ar[l, Leftrightarrow, "\text{(E)}"',"\text{Lemma }\ref{lem (D) (E) (G)}"]  \ar[d, Leftrightarrow, "\text{(F)}", "\text{Theorem }\ref{thm (F)}"'] \ar[uu, bend left=60, Rightarrow,"\text{(A)}"', " 
\text{Theorem }\ref{thm (A)}"] & (\Psi_{\widehat{temp}})\ar[d, Rightarrow] \ar[l, Rightarrow, "\text{(C)}"', "\text{Lemma }\ref{lem (C) (H)}"]\\
        (\Phi_{\widehat{2}})&(\Pi_{\widehat{2}})\ar[l, Leftrightarrow, , "\text{(G)}"',"\text{Lemma }\ref{lem (D) (E) (G)}"] \ar[r, Leftarrow, , "\text{(H)}",  "\text{Lemma }\ref{lem (C) (H)}"'] & (\Psi_{\widehat{2}})
    \end{tikzcd}
\end{equation}
The vertical implications downward immediately follow from the chain of containment \eqref{eq chain of containment}. The Working Hypothesis \ref{assu closure ordering A-packet} is used in direction (B) and the Working Hypothesis \ref{assu ABV AZ dual} is used in directions (D), (E) and (G). 

\begin{remark}\label{remarks}
Let $\RG$ be a general connected reductive algebraic group defined over $F$ and let $G=\RG(F)$.
\begin{enumerate}

\item  Assume that local Langlands correspondence (Conjecture \ref{conj LLC}) is established for $G$. Then one can consider the analogue of \eqref{eq Xi ast} by replacing the collection of groups $G(V_{an,r})$ with all Levi factors of $G$ and replacing Conjecture \ref{conj bound of WF} by Conjecture \ref{conj bound of WF general}. Then the proof of Theorem \ref{thm main classical groups} can be naturally generalized by using \cite[Proposition A2(c)]{BV85} instead of Lemma \ref{lem goal}. If the theories of local Arthur packets and ABV-packets are established and Propositions \ref{prop A-packet}, \ref{prop ABV} still hold true, then the proof Theorems \ref{thm main Jiang} and \ref{thm main Jiang ABV} also work.

\item If we specify $\RG$ to be a non-pure inner form of classical groups, then maximal Levi subgroups of $G$ are of the form $\GL_{m}(A) \times G^{-}$, where $A$ is a central division algebra with dimension $d_A^2 \in \{1,4\}$ over $F$ and $G^-$ is a group of the same type as $G$ of smaller rank (see \cite[\S 2.2]{CG16}). Therefore, though Conjectures \ref{conj Jiang}, \ref{conj Jiang ABV} and \ref{conj bound of WF} still make sense, one may expect that there should be a sharper upper bound according to Remark \ref{rmk division algebra} below.
\item If we specify $\RG$ to be a similitude group, then the nilpotent orbits of $\mathfrak{g}(\overline{F})$ can still be described by partitions. We expect that Theorems \ref{thm main classical groups}, \ref{thm main Jiang} and \ref{thm main Jiang ABV} still hold in this case.
    \end{enumerate}
\end{remark}

Finally, we explain why we consider Conjectures \ref{conj bound of WF}, \ref{conj Jiang 2} and \ref{conj Jiang ABV 2} on a family of groups $G(V_{an,r})$, $ r \leq \mathfrak{r}$, instead of a single group $G(V_{an,\mathfrak{r}})$. Indeed, if we only consider these conjectures for $\Pi_{\ast}(G)$ for $\ast \in \{\emptyset, A, \widehat{temp}\}$, then we do not need the assumptions on groups of smaller rank.

\begin{thm}\label{thm reduction to a single gp}
Fix $G=G(V_{an,\mathfrak{r}})$ and $r < \mathfrak{r}$.
\begin{enumerate}
    \item [(a)] Let $\ast \in \{\emptyset, A, \widehat{temp}\}$. Conjecture \ref{conj bound of WF} holds for any $\pi^{-} \in \Pi_{\ast}(G(V_{an,r}))$ if it holds for any $\pi \in \Pi_{\ast}(G(V_{an,\mathfrak{r}}))$.
    \item [(b)] Let $\ast \in \{ A, \widehat{temp}\}$. Conjecture \ref{conj Jiang 2} holds for any $\pi^{-} \in \Pi_{\ast}(G(V_{an,r}))$ if it holds for any $\pi \in \Pi_{\ast}(G(V_{an,\mathfrak{r}}))$.
\end{enumerate}
\end{thm}

We expect that similar arguments hold for Conjecture \ref{conj Jiang ABV 2}. See Remark \ref{rmk reduction to a single group}(1). However, the same argument does not work if $\ast= \widehat{2}$. See Remark \ref{rmk reduction to a single group}(2).

\section{\texorpdfstring{Proofs of the main results for general linear groups}{}}\label{sec GLn(A) main}
In this section, we prove Theorem \ref{thm main GL} and Corollary \ref{cor GLn}. Here we follow the idea of \cite[\S II.2]{MW87}, which treats the case that $A=F$. {Note that the assumption that the residual characteristic of $F$ is not $2$ in \cite{MW87} is removed, based on the work of Varma (\cite{Var14}). See Remark \ref{rmk Varma} below the proof.}

\subsection{\texorpdfstring{Proof of Theorem \ref{thm main GL}}{}}

First, we establish the upper bound for anti-discrete series representations. Namely, let $\rho'$ be a cuspidal representation of $\GL_{m}(A)$ we claim that 
\begin{align}\label{eq upper bound anti-discrete series GL}
    \WFP(\Speh(\rho',n)|\cdot|_A^x) \leq \{ d_{BV}(\underline{p}(\phi_{\widehat{\Speh(\rho',n)}|\cdot|_A^{x}})) \}.
\end{align}
Write $\JL(\rho')=\St(\rho,s(\rho'))$. Then $\JL(\St(\rho', n))|\cdot|_A^{x}= \St(\rho, s(\rho')n )|\cdot|^{x}$. We compute the right hand side of \eqref{eq upper bound anti-discrete series GL} by definition as follows:
    \begin{align*}
        d_{BV}(\underline{p}(\phi_{\widehat{\Speh(\rho',n)}|\cdot|_A^{x}})) &= d_{BV}(\underline{p}(\phi_{\St(\rho',n)|\cdot|_A^{x}}))\\
        &= d_{BV}(\underline{p}(\phi_{\St(\rho, s(\rho')n)|\cdot|^{x}}))\\
        &=d_{BV}([(s(\rho')n)^{\dim(\rho)}])\\
        &= [\dim(\rho)^{s(\rho')n}].
    \end{align*}
 When $A=F$, we have
$\WFP(\Speh(\rho',n)|\cdot|^x)=\{[\dim(\rho)^{s(\rho')n}\}$ by the computation of the Bernstein-Zelevinsky derivatives(see \cite[Theorem 3.5]{Zel80}). This proves \eqref{eq upper bound anti-discrete series GL} in this case. In general, we let $\LJ:   K\Pi(\GL_{md_A}(F)) \to K\Pi(\GL_m(A))$ denote the map arising from the inverse of Jacquet-Langlands correspondence on discrete series representations, extended to Grothendieck groups by parabolic induction (see \cite[\S 3.1]{Cai23}). It is known that 
\begin{itemize}
    \item For any $\pi \in  K\Pi(\GL_{md_A}(F))$, $\underline{p}_1 \in \WFP(\pi)$ and $\underline{p}_2 \in \WFP(\LJ(\pi))$, we have $\underline{p}_1 \geq \underline{p}_2$ (see \cite[Theorem 3.1, Proposition 4.12]{Cai23}).
    \item  If  $\JL(\St(\rho', n))= \St(\rho, s(\rho')n )$, then $\LJ( \Speh(\rho, s(\rho')n )) = \epsilon \cdot  \Speh(\rho', n) $ for some $\epsilon \in \{\pm 1\}$ (see \cite[Proposition 3.4(2)]{Cai23}).
\end{itemize}
As a consequence, by the case of $\GL_{md_A}(F)$, we have 
\[ 
\WFP(\Speh(\rho',n)|\cdot|_A^x) \leq \WFP(\Speh(\rho,s(\rho')n)|\cdot|^x) = [\dim(\rho)^{s(\rho')n}].\]
This establishes \eqref{eq upper bound anti-discrete series GL} for any anti-discrete series representation $\pi \in \Pi(\GL_m(A))$. See Remark \ref{rmk division algebra} below that the above inequality can be strict.

Next, let $\pi \in \Pi(\GL_m(A))$ be an arbitrary representation and write the standard module of $\widehat{\pi}$ as
    \begin{align}\label{eq stdmod GL inner dual}
         M(\widehat{\pi})=\St(\rho'_1, n_1)|\cdot|_A^{x_1} \times \cdots \times \St(\rho'_f, n_f)|\cdot|_A^{x_f}, 
    \end{align}
    where $\rho_i'$ is a cuspidal representation of $\GL_{m_i}(A)$. For $1 \leq i \leq f$, write $\JL(\St(\rho_i', n_i))|\cdot|_A^{x_i}= \St(\rho_i, s(\rho_i')n_i )|\cdot|^{x_i}$. Then by definition, we have
    \begin{align*}
        d_{BV}(\underline{p}(\phi_{\widehat{\pi}}))&= \sum_{i=1}^f  \left[ \dim(\rho_i)^{s(\rho_i')n_i}\right].
    \end{align*}
Now we apply Aubert-Zelevinsky involution to \eqref{eq stdmod GL inner dual} and obtain
  \[ \pi \leq \widehat{M(\widehat{\pi})}= \Speh(\rho'_1, n_1)|\cdot|_A^{x_1} \times \cdots \times \Speh(\rho'_f, n_f)|\cdot|_A^{x_f}.\]
  Applying Proposition \ref{prop induced orbit} and the special case \eqref{eq upper bound anti-discrete series GL}, we obtain
  \begin{align*}
      \WFP(\pi) \leq \WFP(\widehat{M(\widehat{\pi})})&= \sum_{i=1}^f \WFP(\Speh(\rho_i', n_i)|\cdot|^{x_i}_A)\\
      &\leq  \sum_{i=1}^f [\dim(\rho_i)^{s(\rho_i')n_i}]\\
      &=d_{BV}(\underline{p}_A(\phi_{\widehat{\pi}})),
  \end{align*}
This proves the upper bound for any representation $\pi \in \Pi(\GL_m(A))$ in Conjecture \ref{conj bound of WF} for $\GL_m(A)$.

Finally, we show that $\WFP(\pi)= \{ d_{BV}(\underline{p}(\phi_{\widehat{\pi}})) \}$ for any $\pi \in \Pi(\GL_n(F))$.
Observe that for any $\pi \in \Pi(\GL_m(A))$, if for any irreducible subquotient $\pi'$ of $\widehat{M(\widehat{\pi})}$ that is not isomorphic to $\pi$, the following strict inequality holds
\begin{align}\label{eq goal GL}
    \WFP(\pi')<\WFP(\widehat{M(\widehat{\pi})}), 
\end{align}
 then $\pi$ will be the only irreducible subquotient of $\widehat{M(\widehat{\pi})} $ that can achieve the wavefront set $ \WFP(\widehat{M(\widehat{\pi})})=d_{BV}(\underline{p}_A(\phi_{\widehat{\pi}}))$. Therefore, \eqref{eq goal GL} implies the desired conclusion for $\pi$. We are going to verify \eqref{eq goal GL} for any $\pi \in \Pi(\GL_n(F))_{\lambda}$ for any fixed infinitesimal parameter $\lambda$. Consider the partial ordering $\succeq$ on $ \Pi(\GL_n(F))_{\lambda}$ defined by $\pi \succeq \pi'$ if $M(\pi)\leq M(\pi')$, (which is also equivalent to $\pi \leq M(\pi')$). We recall several facts on the partially ordered set $(\Pi(\GL_n(F))_{\lambda},\succeq)$ now.

 First, the two partial ordering sets $(\Phi(\GL_n(F))_{\lambda}, \geq_C)$ and $(\Pi(\GL_n(F))_{\lambda}, \succeq)$ are isomorphic. Namely,
\begin{align}\label{eq GL closure ordering}
    \phi_{\pi} \geq_C \phi_{\pi'} \Longleftrightarrow \pi \succeq \pi'.
\end{align}
See \cite[Proposition 4.3, Theorem 5.3]{Tad90} and \cite[Theorem 2.2]{Zel81}. Second, it follows from the definition of $\underline{p}$ that 
\begin{align}\label{eq p ordering GL}
    \pi \succ \pi' \Longrightarrow \underline{p}(\phi_\pi) > \underline{p}(\phi_{\pi'}).
\end{align}
Finally, the partially ordered set $(\Pi(\GL_n(F))_{\lambda},\succeq)$ has a unique maximal (resp. minimal) element, which we denote by $\pi^0$ (resp. $\pi_0$). The representation generic and $M(\pi^0)$ is irreducible, and the representation $\pi_0$ is the Aubert-Zelevinsky involution of $\pi^0$. 

Now we apply induction on the partial ordering $\succeq$ on $\Pi(\GL_m(A))_{\lambda}$ to verify \eqref{eq goal GL}.
If $\pi= \pi_0$, then $\widehat{M(\widehat{\pi})}= \widehat{ M(\pi^0) }= \widehat{\pi^0}= \pi_0 $ is irreducible. Thus \eqref{eq goal GL} trivially holds. In general, suppose that \eqref{eq goal GL} is verified for any $\pi''$ such that $\widehat{\pi''}\succ \widehat{\pi}$, so that
\[ \WFP(\pi'')= \{d_{BV}(\underline{p}(\phi_{\widehat{\pi''}}))\}\]
for these $\pi''$.  Let $\pi'$ be an irreducible subquotient of $\widehat{M(\widehat{\pi})}$ not isomorphic to $\pi$. Then $\widehat{\pi'} \leq M(\widehat{\pi})$. Hence, $\widehat{\pi'} \succ \widehat{\pi}$. Thus, we have
\[ \WFP(\pi')= \{ d_{BV}(\underline{p}(\phi_{\widehat{\pi'}})) \} <  d_{BV}(\underline{p}(\phi_{\widehat{\pi}}))= \WFP(\widehat{M(\widehat{\pi''})}), \]
where the strict inequality follows from \eqref{eq p ordering GL} and the fact that $d_{BV}$ is an order-reversing bijection in this case. This completes the proof of the theorem. \qed

\begin{remark}\label{rmk division algebra}

    The equality that $\WFP(\pi) = \{d_{BV}(\underline{p}(\phi_{\widehat{\pi}}))\}$ may fail for general representations of $\GL_m(A)$ when $A\neq F$. Here is an example. Consider the case that $d_A=2$. Take $\rho$ a supercuspidal representation of $\GL_2(F)$ and let $\rho'$ be the representation of $\GL_1(A)$ such that $\JL(\rho')=\rho$. By the computation in the proof above, we have 
    \[  d_{BV}(\underline{p}(\phi_{\widehat{\Speh(\rho',1)}}))= [2^1].\]
    However, every representation of $\GL_1(A)$ has wavefront set $ \{[1^2]\}$. 
\end{remark}
{
\begin{remark}\label{rmk Varma}
Recall that there are two different definitions of wavefront set. The first one is via nonzero coefficient in the local character expansion, and the second one is via the non-vanishing of the degenerate generalized Whittaker models. The two definitions are equivalent by \cite[Theorem I.16]{MW87} when the residual characteristic is not $2$, and the remaining case is established in \cite[Theorem 1]{Var14}. 
 In this proof, we indeed use both definitions of the wavefront set. The proof of Proposition \ref{prop induced orbit} in \cite{MW87} used the first definition (and it works for all residual characteristics). The computation of Bernstein-Zelevinsky derivatives for Speh representations gives us their wavefront set in the second definition.
\end{remark}
}

\subsection{Proof of Corollary \ref{cor GLn}}

     For any local Arthur parameter $\psi$ of $\GL_n(F)$, we have $\Pi_{\psi}= \Pi_{\phi_{\psi}}$, which is a singleton. Thus, Theorem \ref{thm main GL} implies Conjecture \ref{conj Jiang 2}. 
     
     For Conjecture \ref{conj Jiang ABV 2}, ABV-packets are not always singletons, see \cite{CFK22}. However, Working Hypothesis \ref{assu ABV AZ dual} for $\GL_n(F)$ is proved in \cite[Proposition 3.2.1]{CFK22}. Let $\pi \in \Pi_{\phi}^{\textrm{ABV}}$. We have $\underline{p}(\phi_{\widehat{\pi}}) \geq \underline{p}(\widehat{\phi})$. Hence,
     \[ \WFP(\pi)= d_{BV}(\underline{p}(\phi_{\widehat{\pi}})) \leq d_{BV}(\underline{p}(\widehat{\phi})). \]
     This completes the proof of the theorem.     \qed

\section{Proofs of the main results for classical groups}\label{sec proof classical}
We prove Theorems \ref{thm main classical groups}, \ref{thm main Jiang} and \ref{thm main Jiang ABV} in this section.

\subsection{Reduction to the anti-tempered case}\label{sec anti-tempered}
In this subsection, we verify the implications (A), (B), (C), (D) and (E) in the diagram \eqref{eq diagram}.

First, we prove (A).

\begin{thm}\label{thm (A)}
    Let $V=V_{an, \mathfrak{r}}$ and $G=G(V)$.  The statement  $(\Pi_{\widehat{temp}})$ implies $(\Pi)$.
\end{thm}
\begin{proof}
Our goal is to show that for any $\pi \in \Pi( G(V_{an, r}))$ with $r \leq \mathfrak{r}$, the following inequality holds
\begin{align}\label{eq goal (A)}
    \WFP(\pi) \leq d_{BV}(\underline{p}(\phi_{\widehat{\pi}})).
\end{align}
 If $\pi \in \Pi_{\widehat{temp}}(G(V_{an, r}))$, then \eqref{eq goal (A)} holds by  $(\Pi_{\widehat{temp}})$. Thus, we assume $\pi \in \Pi(G(V_{an,r})) \setminus \Pi_{\widehat{temp}}(G(V_{an,r}))$. Write the standard module of $\widehat{\pi}$ as
    \[ M(\widehat{\pi})=\St(\rho_1,a_1)|\cdot|^{x_1} \times \cdots \times \St(\rho_f, a_f) |\cdot|^{x_f} \rtimes (\widehat{\pi})_{temp}. \]
    We apply induction on $f(\pi):=f$. Note that $f(\pi) \geq 1$ since $\pi$ is not anti-tempered. 
    
    Let $n_1= \dim(\rho_1) a_1$,  $r^{-}:= r-n_1$ and $G^{-}:= G(V_{an, r^{-}})$. Let $(\widehat{\pi})^{-} \in \Pi(G^-)$ such that
    \[ M((\widehat{\pi})^-)= \St(\rho_2,a_2)|\cdot|^{x_2} \times \cdots \times \St(\rho_f, a_f) |\cdot|^{x_f} \rtimes (\widehat{\pi})_{temp}.\]
    We claim that 
    \begin{align}\label{eq inj (A)}
        \widehat{\pi} \hookrightarrow \St(\rho_1, a_1)|\cdot|^{x_1} \rtimes (\widehat{\pi})^{-}.
    \end{align}
    Indeed, if $\sigma$ is any irreducible subrepresentation of $\St(\rho_1, a_1)|\cdot|^{x_1} \times (\widehat{\pi})^{-}$, then the exactness of parabolic induction gives
    \begin{align*}
        \sigma &\hookrightarrow \St(\rho_1, a_1)|\cdot|^{x_1} \rtimes (\widehat{\pi})^{-}\\
        &\hookrightarrow \St(\rho_1, a_1)|\cdot|^{x_1} \rtimes M((\widehat{\pi})^{-})\\
        &= M(\widehat{\pi}).
    \end{align*}
    Since $\widehat{\pi}$ is the unique irreducible subrepresentation of $M(\widehat{\pi}) $, we conclude that $\sigma=\widehat{\pi}$, which verifies the claim \eqref{eq inj (A)}.
    
    Write $\pi_{-}:= \widehat{ (\widehat{\pi})^-}$ for short. By definition, $f(\pi_{-})= f(\pi)-1$. Also, the compatibility of local Langlands correspondence and Langlands classification implies
    \begin{align}\label{eq (A) L-par}
        (\phi_{\widehat{\pi}})_{\GL}= (\rho_1|\cdot|^{x_1} \otimes S_{a_1} + \rho_1^{\vee}|\cdot|^{-x_1} \otimes S_{a_1} ) + (\phi_{\widehat{\pi_{-}}})_{\GL}. 
    \end{align} 
    Applying Aubert-Zelevinsky involution on \eqref{eq inj (A)}, we obtain
    \begin{align*}
        \pi \leq \reallywidehat{\St(\rho_1,a_1)|\cdot|^{x_1}} \rtimes  \widehat{(\widehat{\pi})^-}= \Speh(\rho_1, a_1)|\cdot|^{x_1} \rtimes \pi_-.
    \end{align*}
Thus, let $X\in\{A,B,C,D\}$ corresponding to the type of the group $G(V)$. By Propositions \ref{prop induced orbit}, \ref{prop computation of induced orbit}, Theorem \ref{thm main GL}, the induction hypothesis for $\pi_{-}$, Lemma \ref{lem goal} and \eqref{eq (A) L-par}, we obtain
\begin{align*}
    \WFP(\pi) &\leq \left( \WFP(\Speh(\rho_1, a_1)|\cdot|^{x_1}) + \WFP(\Speh(\rho_1, a_1)|\cdot|^{x_1}) +\WFP(\pi_-)  \right)_X\\
    &= \left( [(2\dim(\rho_1))^{a_1}] +\WFP(\pi_-)  \right)_X\\
    & \leq \left( [(2\dim(\rho_1))^{a_1}] + d_{BV}(\underline{p}(\phi_{\widehat{\pi_-}}))  \right)_X\\
    &= d_{BV}( [a_1^{2 \dim(\rho_1)}] \sqcup \underline{p} (\phi_{\widehat{\pi_-}}))\\
    &= d_{BV}(\underline{p}(\phi_{\widehat{\pi}})).
\end{align*}
This verifies \eqref{eq goal (A)} and completes the proof of the theorem.
\end{proof}

Next, we verify the rest of the directions, which mostly follow from definitions and assumptions.

\begin{lemma}\label{lem (B)}
     Let $V=V_{an, \mathfrak{r}}$ and $G=G(V)$. If Working Hypothesis \ref{assu closure ordering A-packet} holds for $G(V_{an,r})$ for any $r \leq \mathfrak{r}$, then $(\Pi_A)$ implies $(\Psi_A)$.
\end{lemma}

\begin{proof}
Let $\pi \in G(V_{an,r})$ for some $r \leq \mathfrak{r}$ and $\pi \in \Pi_{\psi}$. Proposition \ref{prop A-packet}(c) implies $\widehat{\pi} \in \Pi_{\widehat{\psi}}$. Thus, the statement $(\Pi_A)$, Working Hypothesis \ref{assu closure ordering A-packet} and Propositions \ref{prop C implies P}, \ref{prop dBV}(1) imply that
\begin{align*}
    \WFP(\pi) \leq d_{BV}(\underline{p}(\phi_{\widehat{\pi}})) \leq d_{BV}( \underline{p}(\phi_{\widehat{\psi}}))= d_{BV}( \underline{p}(\psi)).
\end{align*}
This completes the proof of the lemma.
\end{proof}

\begin{lemma}\label{lem (C) (H)}
     Let $V=V_{an, \mathfrak{r}}$ and $G=G(V)$. If $\pi$ is an anti-tempered representation of $G(V_{an,r})$ for some $r \leq \mathfrak{r}$, then there exists an anti-tempered local Arthur parameter such that $\pi \in \Pi_{\psi}$ and
     \[ d_{BV}(\underline{p}(\phi_{\widehat{\pi}})) = d_{BV}(\underline{p}(\psi)).\]
     In particular, $(\Psi_{\widehat{temp}})$ implies $(\Pi_{\widehat{temp}})$ and $(\Psi_{\widehat{2}})$ implies $(\Pi_{\widehat{2}})$.
\end{lemma}

\begin{proof}
    Since $\widehat{\pi}$ is tempered, there exists a tempered local Arthur parameter $\widehat{\psi}$ such that $\widehat{\pi} \in \Pi_{\widehat{\psi}}$ and $\phi_{\widehat{\pi}}= \phi_{\widehat{\psi}}$ by Proposition \ref{prop A-packet}(a). Thus, 
    \[ d_{BV}(\underline{p}(\phi_{\widehat{\pi}})) = d_{BV}(\underline{p}(\psi)),\]
    where $\psi:= \widehat{\widehat{\psi}}$. Finally, Proposition \ref{prop A-packet}(c) implies $\pi \in \Pi_{\psi}$, which completes the proof of the lemma. 
\end{proof}

\begin{lemma}\label{lem (D) (E) (G)}
    Let $V=V_{an, \mathfrak{r}}$ and $G=G(V)$. Assume that Working Hypothesis \ref{assu ABV AZ dual} holds for $G(V_{an,r})$ for any $r \leq \mathfrak{r}$. For any $\pi \in \Pi(G(V_{an,r}))$ and $\phi \in \Phi(\pi)$, we have 
    \[d_{BV}( \underline{p}(\phi_{\widehat{\pi}})) \leq d_{BV}( \underline{p}(\widehat{\phi})).\]
    Moreover, there exist a $\phi \in \Phi(\pi)$ such that the above inequality is an equality. In particular, $(\Pi_{\ast})$ is equivalent to $(\Phi_{\ast})$ for $\ast \in \{\emptyset, \widehat{temp}, \widehat{2}\}$.
\end{lemma}
\begin{proof}
   Suppose $\pi \in \Pi_{\phi}^{\textrm{ABV}}$. Working Hypothesis \ref{assu ABV AZ dual} implies that $\widehat{\pi} \in \Pi_{\widehat{\phi}}^{\textrm{ABV}}$. Thus, Propositions \ref{prop ABV}(b), \ref{prop dBV}(1) give
  \[d_{BV}( \underline{p}(\phi_{\widehat{\pi}})) \leq d_{BV}( \underline{p}(\widehat{\phi})).\]
  On the other hand, Proposition \ref{prop ABV}(a) gives $\widehat{\pi} \in \Pi_{\phi_{\widehat{\pi}}} \subseteq \Pi_{\phi_{\widehat{\pi}}}^{\textrm{ABV}}$. Thus, Working Hypothesis \ref{assu ABV AZ dual} implies that $\phi:=\widehat{\phi_{\widehat{\pi}}} \in \Phi(\pi)$. Then 
  \[d_{BV}( \underline{p}(\phi_{\widehat{\pi}})) = d_{BV}( \underline{p}(\widehat{\phi})).\]
  This completes the proof of the lemma.
\end{proof}

\subsection{Reduction to the anti-discrete case}\label{sec anti-discrete}

In this subsection, we prove direction (F) in diagram \eqref{eq diagram}.

\begin{thm}\label{thm (F)}
    Let $V=V_{an, \mathfrak{r}}$ and $G=G(V)$.  The statement $(\Pi_{\widehat{2}})$ implies $(\Pi_{\widehat{temp}})$.
\end{thm}

\begin{proof}
Our goal is to show that for any $\pi \in \Pi_{\widehat{temp}}( G(V_{an, r}))$ with $r \leq \mathfrak{r}$, the following inequality holds
\begin{align}\label{eq goal (F)}
    \WFP(\pi) \leq d_{BV}(\underline{p}(\phi_{\widehat{\pi}})).
\end{align}
 If $\pi \in \Pi_{\widehat{2}}(G(V_{an, r}))$, then \eqref{eq goal (F)} holds by  $(\Pi_{\widehat{2}})$. Thus, we assume $\pi \in \Pi_{\widehat{temp}}(G(V_{an,r})) \setminus \Pi_{\widehat{2}}(G(V_{an,r}))$. Proposition \ref{prop A-packet} implies that $\widehat{\pi}$ is a subquotient of the (unitary) parabolic induction
    \[ \St(\rho_1,a_1) \times \cdots \times \St(\rho_f,a_f)  \rtimes (\widehat{\pi})_2,\]
    where $(\widehat{\pi})_2$ is discrete series. We apply induction on $f(\pi):= f$. Note that $f(\pi) \geq 1$ since $\pi$ is not anti-discrete series.

    Let $r^-:= r- a_1 \dim(\rho_1)$. There exists an irreducible subquotient $(\widehat{\pi})^{-} \in \Pi_{temp}(G(V_{an,r^{-}}))$ of 
     \[ \St(\rho_2,a_2) \times \cdots \times \St(\rho_f,a_f)  \rtimes (\widehat{\pi})_2\]
    such that 
    \begin{align}\label{eq inj (F)}
         \widehat{\pi} \leq \St(\rho_1, a_1) \rtimes (\widehat{\pi})^-.
    \end{align}
    By Parts (a) and (b) of Proposition \ref{prop A-packet}, we have
    \begin{align*}
        (\phi_{\widehat{\pi}})_{\GL}= (\rho_1 \otimes S_{a_1}+\rho_1^{\vee} \otimes S_{a_1}) +(\phi_{(\widehat{\pi})^{-}})_{\GL}.
    \end{align*}
    Let $\pi_-:= \widehat{(\widehat{\pi})^-}$. Note that $f( \pi_-)=f(\pi)-1$. Applying Aubert-Zelevinsky involution on \eqref{eq inj (F)}, we obtain
    \[ \pi \leq \Speh(\rho_1,a_1) \rtimes \pi_{-}.\]
    Then \eqref{eq goal (F)} can be verified by the same computation at the end of the proof of Theorem \ref{thm (A)}. This completes the proof of the theorem.
\end{proof}

\subsection{Reduction to a single group}
In this subsection, we prove Theorem \ref{thm reduction to a single gp}. A key observation is the following lemma, whose proof will be given in \S \ref{sec collapse}.

\begin{lemma}\label{lem inj of induced partition}
    Let $X \in \{B,C,D\}$, $d \in \Z_{> 0}$, and  $\underline{p}, \underline{q} \in \mathcal{P}_X(n)$. If $ ([2d]+\underline{p})_{X} \geq ([2d]+\underline{q})_{X}, $ 
      then $\underline{p} \geq \underline{q}$.
\end{lemma}

Now we prove Theorem \ref{thm reduction to a single gp}.

\begin{proof}[Proof of Theorem \ref{thm reduction to a single gp}]
Let $d:= \mathfrak{r}-r$. Take any $\pi^{-} \in \Pi_{\ast}(G(V_{an,r}))$. We may always find a non-selfdual unitary supercuspidal representation $\rho$ of $\GL_{d}(E)$ such that the parabolic induction 
    \[\pi:= \rho \rtimes \pi^- \]
    is irreducible. If $\ast \in \{\emptyset, \widehat{temp}\}$, it follows from the construction that $\pi \in \Pi_{\ast}(G(V_{an,\mathfrak{r}}))$. If $\ast=A$, then again $\pi \in \Pi_{\ast}(G(V_{an,\mathfrak{r}}))$ by \cite[Proposition 5.1]{Moe11b}. Moreover, we have
  \[ \Psi(\pi)=\{ (\rho \otimes S_{1}\otimes S_1+\rho^{\vee} \otimes S_{1}\otimes S_1 )+\psi \ | \ \psi \in \Psi(\pi^-)  \}.\]

    Let $X \in \{A,B,C,D\}$ based on the type of the group.  Proposition \ref{prop induced orbit}, Theorem \ref{thm main GL} and Lemma \ref{lem inj of induced partition} imply that for any $\underline{p}^- \in \WFP(\pi^-)$, there exists a $\underline{p} \in \WFP(\pi)$ such that
    \[ \underline{p}= ( [2d] + \underline{p}^- )_X.
    \]

    For Part (a), we have 
 \[ \phi_{\widehat{\pi}}= (\rho\otimes S_1 + \rho^{\vee} \otimes S_{1} )+\phi_{\widehat{\pi^-}}.  \]
Thus, Lemma \ref{lem goal} implies that 
 \[ d_{BV}( \underline{p}( \phi_{\widehat{\pi}}))= ( [2d]+ d_{BV}( \underline{p}( \phi_{\widehat{\pi^-}})))_X.\]
 Then Lemma \ref{lem inj of induced partition} implies that if $\underline{p} \leq  d_{BV}( \underline{p}( \phi_{\widehat{\pi}}))$, then $\underline{p}^- \leq  d_{BV}( \underline{p}( \phi_{\widehat{\pi^-}}))$. This proves Part (a). 

 Part (b) can be proved by the same computation above, which we omit. This completes the proof of the theorem.    
\end{proof}

\begin{remark}\label{rmk reduction to a single group}
    We give two remarks on the proof above.
    \begin{enumerate}
        \item We expect that the following also holds.
    \begin{enumerate}
        \item [(c)] Let $\ast \in \{ \emptyset, \widehat{temp}\}$. Conjecture \ref{conj Jiang ABV 2} holds for any $\pi^{-} \in \Pi_{\ast}(G(V_{an,r}))$ if it holds for any $\pi \in \Pi_{\ast}(G(V_{an,\mathfrak{r}}))$.
    \end{enumerate}
    Indeed, let $\pi^-$ and $\pi$ be the same representations in the proof above. It is expected that
    \[ \Phi(\pi)= \{ (\rho\otimes S_1 + \rho^{\vee} \otimes S_{1} )+ \phi \ | \ \phi \in \Phi(\pi^-) \}.\]
    \item Let $\pi^- \in \Pi_{\widehat{2}}(G(V_{an,r}))$. There does not exist an irreducible representation $\tau$ of $\GL_{\mathfrak{r}-r}(E)$ such that $\pi=\tau \rtimes \pi^-$ is irreducible and $\pi$ is anti-discrete at the same time. Thus Conjectures \ref{conj bound of WF}, \ref{conj Jiang 2} and \ref{conj Jiang ABV 2} for $\Pi_{\widehat{2}}(G(V_{an,\mathfrak{r}}))$ do not imply the same conjecture for $\Pi_{\widehat{2}}(G(V_{an,r}))$ if $r<\mathfrak{r}$ in an obvious way. 
    \end{enumerate}
    
\end{remark}

\section{Limitation of the reduction method}\label{limitation}

In the proof of Theorems \ref{thm (A)}, \ref{thm (F)}, from a representation $\pi$ of $G=G(V_{an,r})$, we constructed a pair of representations $(\tau, \sigma)\in \Pi(\GL_{d}(E)) \times \Pi(G(V_{an,r^-}))$ with $d>0$ such that
\begin{enumerate}
    \item [(a)] $\pi$ is a subquotient of $\tau \rtimes \sigma$, and 
    \item [(b)] $d_{BV}(\underline{p}(\phi_{\widehat{\pi}}) )= \Ind_{\GL_d(E) \times G(V_{an,r^-})}^{G}( (\WFP(\tau), d_{BV}(\underline{p}(\phi_{\widehat{\sigma}}))) )$.
\end{enumerate}
Then applying Proposition \ref{prop induced orbit}, Conjecture \ref{conj bound of WF} for $\pi$ is reduced to that for $\sigma$. In this section, we explain that we can not reduce further from the anti-discrete series case using above method. Namely, for an arbitrary anti-discrete series $\pi$, it is possible that there does not exist a pair $(\tau, \sigma)$ satisfying properties (a) and (b) above. For simplicity, we let $G_n=\Sp_{2n}(F)$ or split $\SO_{2n+1}(F)$ throughout this section. The goal is to prove the following.

\begin{prop}\label{prop limit of reduction}
    If $\pi$ is an anti-discrete series representation of $G_n$ and $\pi$ is a subquotient of $\tau \rtimes \sigma$ for some $\tau \in \GL_d(F)$ with $d>0$ and $\sigma \in \Pi(G_{n-d})$. Then there exists an irreducible subquotient $\pi'$ of $\tau \rtimes \sigma$ such that
    \begin{enumerate}
        \item [(i)] $ \underline{p}(\phi_{\widehat{\pi'}})< \underline{p}(\phi_{\widehat{\pi}})$ , and
        \item [(ii)] $d_{BV}(\underline{p}(\phi_{\widehat{\pi'}}) )= \Ind_{\GL_d(F) \times G_{n-d}}^{G_n}( (\WFP(\tau), d_{BV}(\underline{p}(\phi_{\widehat{\sigma}}))) )$.
    \end{enumerate}
\end{prop}

This proposition implies that
\[ d_{BV}( \underline{p}(\phi_{\widehat{\pi}})) \leq \Ind_{\GL_d(F) \times G_{n-d}}^{G_n}( (\WFP(\tau), d_{BV}(\underline{p}(\phi_{\widehat{\sigma}}))) ). \]
The inequality above is strict unless $\phi_{\widehat{\pi'}}$ lies in the fiber 
\[ \{ \phi' \in \Phi(G_n)_{\lambda} \ | \ d_{BV}( \underline{p}(\phi'))=d_{BV}( \underline{p}(\phi_{\widehat{\pi}})) \}, \]
which is considered in \cite[Proposition 1.6]{LL23}. In many cases, this fiber is a singleton consisting of $\phi_{\widehat{\pi}}$ itself. Thus in these cases, there does not exist $(\tau, \sigma)$ satisfying the desired properties (a), (b) for the reduction process. We give an example demonstrating the above argument.

\begin{exmp}
Let $(G_n,X) \in \{ (\SO_{2n+1}(F), B ), (\Sp_{2n}(F),C) \}$ and let $\widehat{\pi}$ be a generalized Steinberg representation of $G_n$ considered in \cite{Tad18} of corank $1$. To be explicit, $\widehat{\pi}$ is the unique discrete series subquotient of a parabolic induction
    \[ \rho|\cdot|^{a_{\rho}+1-\varepsilon_{\rho}/2} \rtimes \pi_{sc}, \]
    where $\rho$ is a selfdual supercuspidal representation of $\GL_d(F)$ ($d>0$) and $\pi_{sc}$ is a supercuspidal representation of $G_{n-d}$ whose $L$-parameter is of the form
    \[ \phi_{sc}= \bigoplus_{\rho'}  \bigoplus_{k=1}^{a_{\rho'}} \rho \otimes S_{2k -\varepsilon_{\rho'}},   \]
    where $\rho'$ ranges over all self dual supercuspidal representations and $\varepsilon_{\rho'} \in \{ 0,1\}$ depending on the types of $\rho'$ and $G_n$. Note that $a_{\rho'}=0$ for almost all $\rho'$.

     The representation $\pi:= \widehat{(\widehat{\pi})}$ is anti-discrete series.  For simplicity, write $\alpha:= a_{\rho}+1-\varepsilon_{\rho}/2$. Then we have
    \[ 0 \to \widehat{\pi} \to \rho|\cdot|^{\alpha} \rtimes \pi_{sc} \to \pi \to 0, \]
    and the $L$-parameters are
    \begin{align*}
        \phi_{\pi}&= \phi_{sc}+(\rho|\cdot|^{\alpha}\otimes S_{1}+ \rho|\cdot|^{-\alpha}\otimes S_{1})  ,\\
    \phi_{\widehat{\pi}}&= \phi_{sc}- \rho \otimes S_{2 \alpha -1}+ \rho \otimes S_{2 \alpha +1}.
    \end{align*}
    In particular, let (recall $\dim(\rho)=d$) 
    \[\underline{p}' := \left(\bigsqcup_{\rho' \not\cong \rho}  \bigsqcup_{k=1}^{a_{\rho'}} [ (2k-\varepsilon_{\rho'})^{\dim(\rho')}  ] \right) \sqcup \bigsqcup_{k=1}^{a_{\rho}-1 } [ (2k-\varepsilon_{\rho})^{d}  ].  \]
    Then
    \begin{align*}
        \underline{p}(\phi_{sc})&= \underline{p'} \sqcup [(2 \alpha-1)^{d}],\\
        \underline{p}(\phi_{\pi})&= \underline{p'} \sqcup [(2 \alpha-1)^{d}] \sqcup[ 1^{2 d}],\\
        \underline{p}(\phi_{\widehat{\pi}})&= \underline{p'} \sqcup [(2 \alpha+1)^{d}].  
    \end{align*}
    By observing the supercuspidal support of $\pi$, we see that the only possible choices of the pairs $(\tau, \sigma)$ such that $\pi \leq \tau \rtimes \sigma$ are  
    \[\{ ( \rho|\cdot|^{\alpha}, \pi_{sc}), (\rho|\cdot|^{-\alpha}, \pi_{sc}) \}.   \]
    Assuming $\WFP(\pi_{sc}) \leq d_{BV}( \underline{p}(\phi_{\pi_{sc}}))$, then any choice of $(\tau,\sigma)$ gives
    \[ \WFP(\pi) \leq ( [(2d)^1] + d_{BV}(\underline{p}(\phi_{\pi_{sc}}) ) )_{X}= d_{BV}(\underline{p}(\phi_{\pi})).\]
    On the other hand, according to \cite[Theorem 3.4]{LLS23}, we see that 
    \[ d_{BV}(\underline{p}(\phi_{\widehat{\pi}}))< d_{BV}(\underline{p}(\phi_{\pi}))\]
    unless $G_n= \SO_{2n+1}(F)$, $d=1$, $\alpha=1/2$ and $\underline{p}'$ satisfies certain conditions. Therefore, the reduction process does not work for most of the Aubert-Zelevinsky involution of generalized Steinberg representations of corank 1. 
\end{exmp}

To prove Proposition \ref{prop limit of reduction}, we first show a consequence of Theorem \ref{thm closure ordering std module}.

\begin{lemma}\label{lem parbolic induction subquotient}
Let $\tau \in \Pi(\GL_d(F))$ and $\sigma \in \Pi(G_{n-d})$. The parabolic induction $\tau \rtimes \sigma$ must contain an irreducible subquotient $\pi$ such that
\[ (\phi_{\pi})_{\GL}= (\phi_{\tau}+\phi_{\tau}^{\vee}) + (\phi_{\sigma})_{\GL} .\]
\end{lemma}
\begin{proof}
 Proposition \ref{prop A-packet}(b) implies that in the Grothendieck group,
\[ M(\tau) \rtimes M(\sigma)= \sum_{\pi \in \Pi_{\phi}}  m_{\pi}M(\pi),   \]
where $\phi_{\GL}=(\phi_{\tau} +  \phi_{\tau}^{\vee})+(\phi_{\sigma})_{\GL} $ and the coefficient $m_{\pi}\in \{0,1\}$. Take any $\pi$ such that $m_{\pi}\neq 0$. In particular, $\pi\leq M(\tau) \rtimes M(\sigma)$. We claim that $\pi \leq \tau \rtimes \sigma$ also holds.

 Indeed, there exist irreducible representations $\tau' \leq M(\tau)$ and $\sigma' \leq M(\sigma)$ such that $\pi \leq \tau' \rtimes \sigma'$. Thus, $\pi \leq M(\pi')$ for some $\pi'$ satisfying that
 \[ (\phi_{\pi'})_{\GL}=  (\phi_{\tau'} + \phi_{\tau'}^{\vee})+ (\phi_{\sigma'})_{\GL} .\]
Now applying Theorem \ref{thm closure ordering std module} to $\tau' \leq M(\tau)$ and $\sigma' \leq M(\sigma)$, we obtain
\[ \phi_{\tau'} \geq_C \phi_{\tau}, \phi_{\sigma'} \geq_C \phi_{\sigma}.\]
Hence, applying Theorem \ref{thm closure ordering std module} on $\pi \leq M(\pi')$, we obtain
 \[\phi_{\pi} \geq_C \phi_{\pi'} \geq_C \phi.\]
 Thus the inequalities above are all equalities. In particular, $\phi_{\tau'}=\phi_{\tau}$ and $\phi_{\sigma'}=\phi_{\sigma}$. Then the last part of Theorem \ref{thm closure ordering std module} implies that $\tau'= \tau$ and $\sigma' = \sigma$. This completes the verification of the claim and the proof of the lemma. 
\end{proof}

Now we prove Proposition \ref{prop limit of reduction}.

\begin{proof}[Proof of Proposition \ref{prop limit of reduction}]
    By Lemma \ref{lem parbolic induction subquotient}, there exists a subquotient $\widehat{\pi'}$ of $\widehat{\tau} \rtimes \widehat{\sigma}$ such that
    \[ (\phi_{\widehat{\pi'}})_{\GL}= (\phi_{\widehat{\tau}}+ \phi_{\widehat{\tau}}^{\vee})+ (\phi_{\widehat{\sigma}})_{\GL}.\]
    We check that $\pi'$ satisfies Conditions (i), (ii).

    For Condition (i), note that $\phi_{\widehat{\pi'}}$ is not multiplicity free, and hence $\phi_{\widehat{\pi}} >_{C} \phi_{\widehat{\pi}'}$ since $\phi_{\widehat{\pi}}$ corresponds to the unique open orbit in the associated Vogan variety (\cite[Proposition 7.2]{HLLZ22}). This implies $\underline{p}(\phi_{\widehat{\pi'}})< \underline{p}(\phi_{\widehat{\pi}})$ by Proposition \ref{prop C implies P}.

    For Condition (ii), write
    \[ \phi_{\widehat{\tau}}= \bigoplus_{i=1}^f \rho_i|\cdot|^{x_i} \otimes S_{a_i},\]
    where $\dim(\rho_i)=d_i$. Then $\WFP(\tau)=\{ \sum_{i=1}^f [d_i^{a_i}] \}$. Let $(X,X')\in \{(B,C), (C,B)\}$. It suffices to show that for any partition $\underline{p}$ of type $X'$, the following equality holds.
    \[ d_{BV}\left( \bigsqcup_{i=1}^f [a_i^{2d_i}] \sqcup \underline{p} \right)= \left(  \sum_{i=1}^f [(2d_i)^{a_i}] + d_{BV}(\underline{p})   \right)_X.\]
    However, it follows directly from Proposition \ref{prop computation of induced orbit} and Lemma \ref{lem goal}. This completes the proof of the proposition.
\end{proof}

\section{Representations with unipotent cuspidal support}\label{sec unipotent}

In this section, we survey the recent progress of Conjecture \ref{conj bound of WF general} for representations with unipotent cuspidal support. We prove certain cases of the Jiang conjecture (Conjecture \ref{conj Jiang}) (Theorems \ref{thm Jiang upper bound unipotent}, \ref{thm upper bound SO(2n+1) unip}) as an application of Theorem \ref{thm main Jiang}. We assume the residue field of $F$ has sufficiently large characteristic throughout this section.

Representations with unipotent cuspidal support, or simply unipotent representations, are classified by Deligne-Langlands-Lusztig parameters. This is proved in \cite{KL87, Lus95, Lus02}, and we refer the reader to \cite[Theorem 4.1.1]{CMBO23} for the properties of the correspondence. For pure inner forms of classical groups, it is discussed in \cite[\S 2.3]{AMS21} and \cite[\S 4]{AMS22} that the above correspondence for unipotent representations is compatible with the local Langlands correspondence (Conjecture \ref{conj LLC}) defined from endoscopy theory. Therefore, we have the following equivalent definition of unipotent representations.

\begin{defn}\label{def unipotent}
We say $\pi \in \Pi(G(V))$ is with unipotent cuspidal support, or unipotent for short, if $\phi_{\pi}$ is trivial on $I_{F}$, the inertia subgroup of $W_F$. In this case, we say $\pi$ is with real infinitesimal parameter if the eigenvalues of 
 \[ \phi_{\pi}\left( \Fr, \begin{pmatrix}
     q^{1/2} & \\& q^{-1/2}
 \end{pmatrix} \right) \]
are real powers of $q$, where $\Fr$ is any choice of Frobenius element and $q$ is the cardinality of the residue field of $F$.
\end{defn}

Let $\Pi^{unip}(G)$ (resp. $\Pi^{\R unip}(G)$) denote the subset of $\Pi(G)$ consists of unipotent representations (resp. unipotent representations with real infinitesimal parameter). For $\ast \in \{\emptyset, A,\widehat{temp}, \widehat{2}\}$ and $ \oldbullet \in \{unip, \R unip\}$, let $\Pi^{\oldbullet}_{\ast}(G):= \Pi^{\oldbullet}(G)\cap \Pi_{\ast}(G)$. 

Write $G=G(V)$ and $V=V_{an,\mathfrak{r}}$. Let $(\Xi, -) \in \{(\Pi, \ref{conj bound of WF}), (\Phi, \ref{conj Jiang ABV 2}), (\Psi, \ref{conj Jiang 2})\}$. We consider the restriction of the statements \eqref{eq Xi ast} 
\begin{align}\label{eq Xi ast heart}
    \tag{$\Xi_{\ast}^{\oldbullet}$} \text{Conjecture }- \text{ holds for any }\pi \in \Pi_{\ast}^{\oldbullet}(G(V_{an},r)) \text{ for any }r \leq \mathfrak{r},
\end{align}
for $\oldbullet \in \{unip, \R unip\}$. One can check that the proof Theorems \ref{thm main classical groups}, \ref{thm main Jiang} and \ref{thm main Jiang ABV} work for the equivalence among $(\Xi_{\ast}^{\oldbullet})$ without change. We state it in the following theorem.

\begin{thm}\label{thm main unip}
    Let $V=V_{an, \mathfrak{r}}$ and $G=G(V)$. Let $\oldbullet \in \{unip, \R unip \}$
    \begin{enumerate}
        \item [(a)] The statements  $(\Pi^{\oldbullet}), (\Pi_{A}^{\oldbullet}), (\Pi_{\widehat{temp}}^{\oldbullet}), (\Pi_{\widehat{2}}^{\oldbullet})$ are equivalent.
        \item [(b)] Suppose that Working Hypothesis \ref{assu closure ordering A-packet} holds for any $\pi \in \Pi^{\oldbullet}(G(V_{an,r}))$ for any $r \leq \mathfrak{r}$. Then the statements $(\Psi_A^{\oldbullet}), (\Psi_{\widehat{temp}}^{\oldbullet}), (\Psi_{\widehat{2}}^{\oldbullet}), (\Pi^{\oldbullet})$ are equivalent.
        \item [(c)] Suppose that Working Hypothesis \ref{assu ABV AZ dual} holds for any $\pi \in \Pi^{\oldbullet}(G(V_{an,r}))$ for any $r \leq \mathfrak{r}$. Then the statements $(\Phi^{\oldbullet}), (\Phi_{\widehat{temp}}^{\oldbullet}), (\Phi_{\widehat{2}}^{\oldbullet}), (\Pi^{\oldbullet})$ are equivalent.
    \end{enumerate}
\end{thm}

\subsection{\texorpdfstring{Real infinitesimal parameter}{}}

We say a local Arthur parameter $\psi$ is \emph{basic} if $\psi$ is trivial on $W_F \times \SL_2^{D}(\BC)$. We recall one of the main results in \cite{CMBO22,CMBO23}.

\begin{thm}[{\cite[Corollary 3.0.5]{CMBO22} \cite[Theorem 1.4.1(1)]{CMBO23}}]\label{thm CMBO} 
Let $\mathrm{G}$ be a connected reductive group defined and inner to split over $F$, and let $G=\mathrm{G}(F)$. Assume that the residue field of $F$ has sufficiently large characteristic. Suppose $\pi \in \Pi(G)$ is unipotent with real infinitesimal parameter. Then the wavefront set of $\pi$ is a singleton, and 
    \begin{align}\label{eq computation of WF}
        \WFN(\pi)= \{ d_{BV}(\OO_{\phi_{\widehat{\pi}}}) \}.
    \end{align}
    In particular, this implies the followings.
    \begin{enumerate}
        \item [(i)] The statement $(\Pi^{\R unip})$ is true.
        \item [(ii)]If $\psi_0$ is a basic local Arthur parameter, then Part (i) of Conjecture \ref{conj Jiang} holds, and the wavefront set of any $\pi \in \Pi_{\psi_0}$ achieves the upper bound in Part (ii) of Conjecture \ref{conj Jiang}.
    \end{enumerate}       
\end{thm}

By Theorem \ref{thm main unip}, Part (i) of the above theorem implies the Jiang conjecture (Conjecture \ref{conj Jiang}) on local Arthur packets and ABV-packets assuming Working Hypothesis \ref{assu closure ordering A-packet}, \ref{assu ABV AZ dual}. We can also prove the achievement part by the explicit formula of wavefront sets. Let us focus on Conjecture \ref{conj Jiang} and introduce some notation. We say a local Arthur parameter $\psi$ is \emph{$W_F$-trivial} (resp. \emph{$I_F$-trivial}) if $\psi|_{W_F}$ (resp. $\psi|_{I_F}$) is trivial. Suppose $\pi \in \Pi_{\psi}$. Then $\pi$ is unipotent (resp. unipotent with real infinitesimal parameter) if and only if $\psi$ is $I_F$-trivial (resp. $W_F$-trivial).

\begin{thm}\label{thm Jiang upper bound unipotent}
Let $G=G(V)$ be a pure inner form of a split classical group defined over $F$. Assume that the residue field of $F$ has sufficiently large characteristic. We have the followings.
\begin{enumerate}
    \item [(i)]  If Working Hypothesis \ref{assu closure ordering A-packet} holds for any $\pi \in \Pi^{\R unip}(G)$, then Conjecture \ref{conj Jiang}(i) holds for any $W_F$-trivial local Arthur parameters $\psi$ of $G$.
    \item [(ii)] If $G$ is quasi-split, then
    Conjecture \ref{conj Jiang}(ii) holds for any $W_F$-trivial local Arthur parameters $\psi$ of $G$.
\end{enumerate}
\end{thm}

\begin{proof}
    Part (i) follows from Theorems \ref{thm main unip} and \ref{thm CMBO}. For Part (ii), recall that we require the group $G$ to be quasi-split in the statement. Therefore, the $L$-packet $\Pi_{\phi_{\widehat{\psi}}}$ is non-empty and we have an inclusion $\Pi_{\phi_{\widehat{\psi}}} \subseteq \Pi_{\widehat{\psi}}$ by \cite[Proposition 7.4.1]{Art13} and \cite[Proposition 8.4.1]{Mok15}. Thus, we can take $\widehat{\pi} \in \Pi_{\phi_{\widehat{\psi}}} \subseteq \Pi_{\widehat{\psi}}$ and let $\pi:= \widehat{(\widehat{\pi})} \in \Pi_{\psi}$ by Proposition \ref{prop A-packet}(c). This gives the desired representation for Conjecture \ref{conj Jiang}(ii) for $\psi$ since Theorem \ref{thm CMBO} gives
    \[ \WFP(\pi)= d_{BV}\{(\underline{p}(\phi_{\widehat{\pi}}))\}=\{d_{BV}(\underline{p}(\phi_{\widehat{\psi}}))\}=\{d_{BV}(\underline{p}(\psi))\}.\]
    This completes the proof of the theorem.
\end{proof}

\begin{remark}
For the groups $G=\Sp_{2n}(F)$ or split $\SO_{2n+1}(F)$, Working Hypothesis \ref{assu closure ordering A-packet} is verified (Theorem \ref{thm closure ordering A-packet SpSO}). Thus Conjecture \ref{conj Jiang} holds for all $W_F$-trivial local Arthur parameters $\psi$ of $G$ as long as the residue field of $F$ has sufficiently large characteristic.
\end{remark}

\subsection{\texorpdfstring{Non-real infinitesimal parameter for $\SO_{2n+1}$}{}}

For unipotent $\pi$ with \emph{non-real} infinitesimal parameter, the equality \eqref{eq computation of WF} may fail (see the remark after \cite[Theorem 1.4.1]{CMBO23}). However, for $\SO_{2n+1}(F)$, Conjecture \ref{conj bound of WF} is established for all unipotent anti-tempered representations in \cite{Wal18}.

\begin{thm}[{\cite[Theorem 3.3]{Wal18}}]\label{thm Wal unipotent anti-tempered}
 Suppose $\pi$ is a unipotent anti-tempered representation of $\SO_{2n+1}(F)$ (not necessarily quasi-split) and the cardinality of the residue field of $F$ is greater than or equal to $6n+4$. Then 
\[ \WFP(\pi)= \{ d_{BV}( \underline{p}(\phi_{\widehat{\pi}})) \}.\]
\end{thm}

Therefore, Theorem \ref{thm main unip} implies the following. 

\begin{thm}\label{thm upper bound SO(2n+1) unip}
Let $G=G(V)$ be an odd special orthogonal group.
\begin{enumerate}
    \item [(a)]Conjecture \ref{conj bound of WF} holds for any unipotent representations of $G$.
    \item [(b)]Conjecture \ref{conj Jiang} holds for any $I_F$-trivial local Arthur parameters of $G$ if $G$ is split over $F$. If $G$ is not split over $F$, the same holds if Working Hypothesis \ref{assu closure ordering A-packet} holds for any $I_F$-trivial local Arthur parameters of $G$.
\end{enumerate}
\end{thm}

\begin{proof}
    Part (a) follows from Theorems \ref{thm main unip} and \ref{thm Wal unipotent anti-tempered}. For Part (b), Part (i) of Conjecture \ref{conj Jiang} is verified by Theorems \ref{thm main unip} and \ref{thm Wal unipotent anti-tempered} again. Part (ii) of Conjecture \ref{conj Jiang} can be verified as in the proof Part (ii) of Theorem \ref{thm Jiang upper bound unipotent}. This completes the proof of the theorem.
\end{proof}

 Waldspurger (\cite{Wal20}) also computed the wavefront sets of unipotent tempered representations of $\SO_{2n+1}(F)$, which are all singletons. For these representations $\pi$, we already have $\WFP(\pi) \leq d_{BV}(\underline{p}(\phi_{\widehat{\pi}}))$ by Theorem \ref{thm upper bound SO(2n+1) unip}. We expect that $\WFP(\pi)=\{  d_{BV}(\underline{p}(\phi_{\widehat{\pi}})) \}$ and that it also matches the arithmetic wavefront set considered in \cite{CJLZ23, JLZ25}. In the rest of this subsection, we examine this for the two families considered in \cite[Appendix B.2.3]{JLZ22}.

  Let $G_n^{\ast}$ be the split group $\SO_{2n+1}(F)$ and $G_n$ its non-quasi-split pure inner form in the following discussion. Let $\rho$ denote the trivial character of $W_F$ and let $\chi$ denote the non-trivial unramified quadratic character of $W_F$. We consider unramified tempered $L$-parameters of $G_n$ of \emph{good parity}, which are of the form
  \[ \phi= \bigoplus_{i \in I_{\rho}} \rho \otimes S_{2a_i} + \bigoplus_{i \in I_{\chi}} \chi \otimes S_{2a_i},\]
    where $a_i$'s are integers for any $ i\in I_{\rho} \sqcup I_{\chi}$ and
    \[ \sum_{i \in I_{\rho}\sqcup I_{\chi}} 2a_i= 2n. \]
    
    The representations $\pi$ in the $L$-packet $\Pi_{\phi}:=\Pi_{\phi}(G_n^{\ast}) \sqcup \Pi_{\phi}(G_n)$ can be identified with functions $\varepsilon: I_{\rho} \sqcup I_{\chi} \to \{\pm 1\}$ such that for any $i,j \in I_{\rho}$ (resp. $i,j \in I_{\chi}$),  $\varepsilon(i)=\varepsilon(j)$ if $a_i=a_j$. Write $\pi=\pi(\phi, \varepsilon)$ in this case. The representation $\pi(\phi,\varepsilon)$ is in $\Pi_{\phi}(G_n^{\ast})$ if and only if $\prod_{i \in I_{\rho}\sqcup I_{\chi}}\varepsilon(i)=1$.

\begin{exmp}\label{exmp SO odd} With the notation above, we compute $d_{BV}(\underline{p}(\phi_{\widehat{\pi}}))$ for two families of unipotent representations.
    \begin{enumerate}
        \item  Consider the case that $\pi=\pi(\phi,\varepsilon)$ is supercuspidal. In this case, $\phi$ is discrete, and hence we may write the multi-sets as
        \[ \{ a_i \}_{i \in I_{\rho}}= \{1, 2,\ldots, a_{\rho}\},\ \{ a_i \}_{i \in I_{\chi}}= \{1,2, \ldots, a_{\chi}\}.\]
        Moreover, by \cite[Corollary 3.5]{MR18}, the function $\varepsilon$ is determined by
        \[\varepsilon(i)= (-1)^{a_i}.\]
        Now we compute $d_{BV}(\underline{p}(\phi_{\widehat{\pi}}))$. Since $\pi$ is supercuspidal, we have $\widehat{\pi}= \pi$, and hence 
        \begin{align*}
            d_{BV}(\underline{p}(\phi_{\widehat{\pi}}))=d_{BV}( [2a_{\rho},2a_{\rho}-2,\ldots,2]\sqcup [2a_{\chi},2a_{\chi}-2,\ldots,2]  ),
        \end{align*}
        which matches the computation in \cite[(B.2.3.a)]{JLZ22} and the wavefront set in \cite{Wal20}.
        
        \item Consider the case that $I_{\chi}$ is empty and $|I_{\rho}|=3$. Indeed, in this case, $\pi=\pi(\phi,\varepsilon)$ is unipotent with real infinitesimal parameter, and hence Theorems \ref{thm CMBO}, \ref{thm Wal unipotent anti-tempered} already imply that $d_{BV}(\underline{p}(\phi_{\widehat{\pi}}))$ matches $\WFP(\pi)$ computed by Waldspurger. Thus, we only focus on the case that $I_{\rho}=\{1,2,3\}$, $a_1<a_2<a_3$ and $\varepsilon(i)=(-1)^i$ to show how to compute $d_{BV}(\underline{p}(\phi_{\widehat{\pi}}))$ using (the proof of) \cite[Theorem 3.4]{MR18}, following the idea of \cite{Jan18} and \cite[Algorithm 4.1]{AM20}.

        For $3 \leq j \leq a_3$, consider the $L$-parameters
        \[ \phi_j:=\begin{cases}
            \rho\otimes S_{2a_1}+ \rho\otimes S_{2a_2}+ \rho\otimes S_{2j} & \text{ if }a_2<j\leq a_3,\\
            \rho\otimes S_{2a_1}+ \rho\otimes S_{2j-2}+ \rho\otimes S_{2j} & \text{ if } a_1+1< j \leq a_2,\\
            \rho\otimes S_{2j-4}+ \rho\otimes S_{2j-2}+ \rho\otimes S_{2j} & \text{ if }3\leq j\leq a_1+1,
        \end{cases} \]
        and the corresponding representations $\pi_j:= \pi(\phi_j, \varepsilon)$. Note that $\pi_{a_3}=\pi$ and $\pi_3$ is supercuspidal. Applying \cite[Theorem 3.4]{MR18} and the compatibility between Aubert-Zelevinsky involution and partial Jacquet module (see \cite[Proposition 3.9]{AM20}), we have the followings.
        \begin{enumerate}
            \item [(a)] If $a_2 <j \leq a_3-1$, then $ \pi_{j+1} \hookrightarrow \rho|\cdot|^{\half{2j+1}} \rtimes \pi_{j},$ and hence 
            \begin{align*}
                \widehat{\pi}_{j+1} \hookrightarrow \rho|\cdot|^{\half{-2j-1}} \rtimes \widehat{\pi}_{j}. 
            \end{align*}
            \item [(b)] If $a_1+1 < j \leq a_2$, then $ \pi_{j+1} \hookrightarrow \Speh(\rho,2)|\cdot|^{j} \rtimes \pi_{j},$ and hence
            \begin{align*}
                 \widehat{\pi}_{j+1} \hookrightarrow \St(\rho,2)|\cdot|^{-j} \rtimes \widehat{\pi}_{j}.
            \end{align*}
            \item [(c)] If $3\leq j\leq a_1+1$, then $ \pi_{j+1} \hookrightarrow \Speh(\rho,3)|\cdot|^{\half{2j-1}} \rtimes \pi_{j},$ and hence
            \begin{align*}
                \widehat{\pi}_{j+1} \hookrightarrow \St(\rho,3)|\cdot|^{\half{-2j+1}} \rtimes \widehat{\pi}_{j}.
            \end{align*}
        \end{enumerate}
        Composing the injections above, we obtain the standard module of $\widehat{\pi}$:
        \begin{align*}
            \widehat{\pi} \hookrightarrow \bigtimes_{a_2 < j \leq a_3-1} \St(\rho,1)|\cdot|^{\half{-2j-1}} \times \bigtimes_{a_1+1<j\leq a_2 } \St(\rho,2)|\cdot|^{-j} \times \bigtimes_{3 \leq j\leq a_1+1} \St(\rho,3)|\cdot|^{\half{-2j+1}} \rtimes \widehat{\pi}_3. 
        \end{align*}
        As a consequence, we have $(\phi_{\widehat{\pi}})_{\GL}= \phi_1+ (\phi_{\widehat{\pi}_3})_{\GL} +\phi_1^{\vee}$, where
        \begin{align*}
            \phi_1&= \bigoplus_{a_2 < j \leq a_3-1}\rho|\cdot|^{\half{-2j-1}}\otimes S_1+ \bigoplus_{a_1+1 < j \leq a_2}\rho|\cdot|^{-j}\otimes S_2+ \bigoplus_{3 < j \leq a_1+1}\rho|\cdot|^{\half{-2j+1}}\otimes S_3,\\
            (\phi_{\widehat{\pi}_3})_{\GL}&=\rho\otimes S_{2}+\rho\otimes S_{4}+\rho\otimes S_{6}.
        \end{align*}
        Therefore,        
        \begin{align*}
           d_{BV}( \underline{p}(\phi_{\widehat{\pi}}))&= d_{BV}([3^{a_1-1}, 2^{a_2-a_1-1}, 1^{a_3-a_2-1}] \sqcup [3^{a_1-1}, 2^{a_2-a_1-1}, 1^{a_3-a_2-1}] \sqcup [6,4,2])\\
           &= d_{BV}([6,4,3^{2a_1-2}, 2^{2a_2-2a_1-1}, 1^{2a_3-2a_2-2}])\\
           &= ([7,4,3^{2a_1-2}, 2^{2a_2-2a_1-1}, 1^{2a_3-2a_2-2}]_B)^{\ast}\\
           &=[7,3^{2a_1},2^{2a_2-2a_1-2},1^{2a_3-2a_2-2}]^{\ast}\\
           &=[2a_3-3,2a_2-1,2a_1+1,1^4],
        \end{align*}
       which matches the computation in \cite[\S 5.3]{Wal20}.
    \end{enumerate}
\end{exmp}

Using the algorithm for computing Aubert-Zelevinsky involution in \cite{AM20}, we have checked that for any tempered unipotent representation $\pi$ of split $\SO_{2n+1}(F)$ with $n\leq 25$, the upper bound $d_{BV}(\underline{p}(\phi_{\widehat{\pi}}))$ matches $\WFP(\pi)$ computed by Waldspurger.

\begin{remark}\label{rmk La}
    Let $\pi$ be a unipotent representation of $\SO_{2n+1}(F)$. One may define another unipotent representation $\textbf{IM}(\pi)$ of $\SO_{2n+1}(F)$, the Iwahori-Matsumoto involution of $\pi$, via the affine Hecke algebra. When $\pi$ is tempered, with the computation in \cite[\S 7]{La24}, one can see that the wavefront set of $\pi$ computed in \cite{Wal20} is exactly $\{d_{BV}(\underline{p}(\phi_{\textbf{IM}(\pi)}))\}$. Therefore, our expectation that $\WFP(\pi)=\{d_{BV}(\underline{p}(\phi_{\widehat{\pi}}))\}$ is a consequence of the equality
    \[ \widehat{\pi}=\textbf{IM}(\pi),\]
    which is also expected (see the last paragraph of \cite[\S 4]{La24}).
\end{remark}

\subsection{Non-real infinitesimal parameter for general groups}

For general groups, currently there are no definite results to compute the wavefront set of unipotent representations with non-real infinitesimal parameter. On the other hand, Okada (\cite[Conjecture 1.4.3]{CMBO23}) conjectured that the wavefront sets of these representations are all singletons and described their image under the Sommer duality. In this subsection, we briefly recall necessary notations to state Okada's conjecture explicitly, and show that Okada's conjecture implies Conjecture \ref{conj bound of WF general} for general unipotent representations. See Theorem \ref{thm equiv CMBO conj}. 

For the following notations, see \cite[\S 2.3]{CMBO23} or the introduction of \cite{Ach03} for more details. Let $\mathcal{N}_{o}$ denote the set of nilpotent orbit of $\mathfrak{g}(\BC)$ and let $\mathcal{N}_{o,c}$ (resp. $\mathcal{N}_{0,\overline{c}}$) denote the set of pairs $(\OO, C)$ (resp. $(\OO, \overline{C})$) where $\OO \in \mathcal{N}_o$ and $C$ (resp. $\overline{C}$) is a conjugacy class in the fundamental group $A(\OO)$ of $\OO$ (resp. Lusztig’s canonical quotient $\overline{A} (\OO)$ of $A(\OO)$). Denote $\mathcal{N}_{o}^{\vee}$, $\mathcal{N}_{o,c}^{\vee}$, and $\mathcal{N}_{o,\overline{c}}^{\vee}$ similarly for the dual group $\widehat{G}$. Let 
\begin{align*}
    d_{S}:\mathcal{N}_{o,c} \to \mathcal{N}_{o}^{\vee}
\end{align*}
denote the duality map defined in \cite{Som01}. By \cite[Proposition 15]{Som01}, the map $d_S$ factors through the quotient map $ \mathcal{N}_{o,c} \twoheadrightarrow \mathcal{N}_{o,\overline{c}}$, and hence we may regard
\[d_{S}:\mathcal{N}_{o,\overline{c}} \to \mathcal{N}_{o}^{\vee}.\]
We need the following properties of $d_{S}$.

\begin{prop}[{\cite[Proposition 2.3]{Ach03}}]\label{prop Sommer}
Let $(\OO,\overline{C}) \in \mathcal{N}_{o,\overline{c}}$. We have
    \[ d_{BV}(\OO)= d_S(\OO,1)\geq d_S(\OO,\overline{C}). \]
\end{prop}

Suppose $\RG$ is a connected reductive algebraic group defined and inner to split over $F$. In \cite{Oka22}, he defined the \emph{unramified canonical wavefront set} ${}^K \mathrm{WF}(\pi)$ for each depth-0 representation $\pi$ of $G:=\RG(F)$, which can be viewed as a subset of $\mathcal{N}_{o,\overline{c}}$. When ${}^K \mathrm{WF}(\pi)=\{ (\OO, \overline{C})\}$ is a singleton, he also showed that the wavefront set of $\pi$ is also a singleton and  $\overline{\mathfrak{n}}^m(\pi)=\{\OO\}$ (\cite[Theorem 2.37]{Oka22}). Now we state Okada's conjecture.
\begin{conj}[{\cite[Conjecture 1.4.3]{CMBO23}}]\label{conj CMBO}
    Let $\pi$ be a unipotent representation of $G$. The unramified canonical wavefront set ${}^K\mathrm{WF}(\pi)$ is a singleton, and
    \begin{align*}
         d_S( {}^K\mathrm{WF}(\pi))= \OO_{\phi_{\widehat{\pi}}}.
    \end{align*}
\end{conj}
Note that the above conjecture is already proved if $\pi$ has real infinitesimal parameter (Theorem \ref{thm CMBO}). Thus, the content of the conjecture is for representations with non-real infinitesimal parameter.

The following is the main result of this subsection.

\begin{thm}\label{thm equiv CMBO conj} If Conjecture \ref{conj CMBO} holds for $G$, then Conjecture \ref{conj bound of WF general} holds for all unipotent representations $\pi$ of $G$.
\end{thm}
\begin{proof}
    Let $\pi$ be a unipotent representation of $G$. By Conjecture \ref{conj CMBO}, write ${}^K\mathrm{WF}(\pi)=\{(\OO,\overline{C})\}$, and Proposition \ref{prop Sommer} gives
\[ d_{BV}(\OO)= d_S(\OO,1) \geq d_S(\OO,\overline{C})= \OO_{\phi_{\widehat{\pi}}}. \]
Taking the Barbasch-Vogan dual on both sides, we obtain
\[ d_{BV}( \OO_{\phi_{\widehat{\pi}}}) \geq d_{BV}^2( \OO) \geq \OO.\]
This verifies Conjecture \ref{conj bound of WF} for $\pi$ and completes the proof of the theorem.
\end{proof}

\section{\texorpdfstring{Proof of Lemma \ref{lem goal}}{}}\label{sec proof of key lemma}

\subsection{Computation of collapse}\label{sec collapse}
In this subsection, we introduce the notation and lemmas needed for the computation of the collapse.

\begin{defn}
For $\underline{p}= [p_1,\dots, p_N]\in \mathcal{P}(n)$ and $b\in \Z$, we define
\[ \underline{p}_{> b}= [p_1,\dots, p_i] \]
where $i= \max\{ 1 \leq j\leq N \ | \ p_j>b \}$. We define $\underline{p}_{\oldbullet b}$ similarly for $\oldbullet \in \{ =,<,\leq,\geq \}$ so that $\underline{p}= \underline{p}_{>b} \sqcup \underline{p}_{\leq b}= \underline{p}_{>b} \sqcup \underline{p}_{=b} \sqcup \underline{p}_{<b}$, etc.
\end{defn}

Following the notation in \cite{Ach03}, we shall omit the parentheses between the superscript and subscript. For example, we shall write $\underline{p}_{>x,D} \!^{+} \underline{\vphantom{p}}_{B}  {}^{-\ast}$ instead of $(((((\underline{p}_{>x})_{D})^{+})_{B})^{-})^{\ast}$.

We apply \cite[Lemma 3.1]{Ach03} in the following form. Note that $\underline{p}_{>x}$ is always superior to $\underline{p}_{\leq x}{}^{+}$ in the notation there.

\begin{lemma}\label{lem Ach03}
Let $x$ be a positive integer and $\underline{p}$ be a partition. Then for $X \in \{B,C,D\}$, the $X$-collapse (if defined) of $\underline{p}$ is given by the following table.
\begin{center}
\begin{tabular}[b]{|l|c|c|c|c|}
\hline
& \multicolumn{2}{|c|}{$l(\underline{p}_{>x})$ even}&\multicolumn{2}{|c|}{$l(\underline{p}_{>x})$ odd} \\ \hline   
& $|\underline{p}_{>x}|$ even     & $|\underline{p}_{>x}|$ odd     & $|\underline{p}_{>x}|$ even     & $|\underline{p}_{>x}|$ odd     \\  \\ [-2em]&&&& \\ \hline  \\ [-2em]&&&& \\
$\underline{p}_B$: &
  $\underline{p}_{>x,D}        \sqcup \underline{p}_{\leq x,B}$ &
  $\underline{p}_{>x}{}^-\underline{\vphantom{p}}_D    \sqcup \underline{p}_{\leq x}{}^+\underline{\vphantom{p}}_B$ &
  $\underline{p}_{>x}{}^-\underline{\vphantom{p}}_B    \sqcup \underline{p}_{\leq x}{}^+\underline{\vphantom{p}}_D$ &
  $\underline{p}_{>x,B}        \sqcup \underline{p}_{\leq x,D}$  \\  \\ [-2em]&&&& \\ \hline  \\ [-2em]&&&& \\
$\underline{p}_C$: &
  $\underline{p}_{>x,C}        \sqcup \underline{p}_{\leq x,C}$ &
  $\underline{p}_{>x}{}^-\underline{\vphantom{p}}_C    \sqcup \underline{p}_{\leq x}{}^+\underline{\vphantom{p}}_C$ &
  $\underline{p}_{>x,C}        \sqcup \underline{p}_{\leq x,C}$ &
  $\underline{p}_{>x}{}^-\underline{\vphantom{p}}_C    \sqcup \underline{p}_{\leq x}{}^+\underline{\vphantom{p}}_C$  \\  \\ [-2em]&&&& \\ \hline  \\ [-2em]&&&& \\
$\underline{p}_D$: &
  $\underline{p}_{>x,D}        \sqcup \underline{p}_{\leq x,D}$ &
  $\underline{p}_{>x}{}^-\underline{\vphantom{p}}_D    \sqcup \underline{p}_{\leq x}{}^+\underline{\vphantom{p}}_D$ &
  $\underline{p}_{>x}{}^-\underline{\vphantom{p}}_B    \sqcup \underline{p}_{\leq x}{}^+\underline{\vphantom{p}}_B$ &
  $\underline{p}_{>x,B}        \sqcup \underline{p}_{\leq x,B}$ \\  [-1em]&&&& \\\hline
\end{tabular}
\end{center}
\end{lemma}

We need the following corollary of above lemma.
\begin{cor}\label{cor Ach03}
Let $\underline{p}$ be a partition and $x$ a positive integer.
\begin{enumerate}
    \item [(i)] If $\underline{p}\in \mathcal{P}(2n+1)$ and $\underline{p}_{=2x+1}$ is non-empty, then
    \[ \underline{p}_{B}= (\underline{p}_{>2x+1}\sqcup \underline{p}_{<2x+1})_X \sqcup \underline{p}_{=2x+1},\]
    where $X=B$ or $D$.
    \item [(ii)]If $\underline{p}\in \mathcal{P}(2n)$ and $\underline{p}_{=2x}$ is non-empty, then
    \[ \underline{p}_{C}= (\underline{p}_{>2x}\sqcup \underline{p}_{<2x})_C \sqcup \underline{p}_{=2x} .\]
    \item [(iii)] If $\underline{p}\in \mathcal{P}(2n)$ and $\underline{p}_{=2x+1}$ is non-empty, then
    \[ \underline{p}_{D}= (\underline{p}_{>2x+1}\sqcup \underline{p}_{<2x+1})_X \sqcup \underline{p}_{=2x+1},\]
    where $X=B$ or $D$.
\end{enumerate}
\end{cor}

Now we prove Lemma \ref{lem inj of induced partition} which is used in the proof of Theorem \ref{thm reduction to a single gp}.

\begin{proof}[Proof of Lemma \ref{lem inj of induced partition}]
    Write $\underline{p}=[p_1,\ldots, p_r],$ $\underline{q}=[q_1,\ldots, q_s],$ $([2d]+\underline{p})_X=[p_1',\ldots, p_r'],$ and $([2d]+\underline{q})_X=[q_1',\ldots, q_s'].$  By Lemma \ref{lem Ach03}, it is not hard to see that $p_{i}'=p_i$ for $i\geq 3$, and 
    \[ [p_1',p_2']= \begin{cases}
        [p_1+2d, p_2] & \text{ if }(-1)^{p_1}=\varepsilon,\\
        [p_1+2d-1, p_2+1] & \text{ if }(-1)^{p_1}= -\varepsilon,
    \end{cases} \]
   where $\varepsilon= 1$ if $X=C$ and $\varepsilon= -1$ if $X\in \{ B,D\}$. The same holds for $\underline{q}$.

    Since $ ([2d]+\underline{p})_{X} \geq ([2d]+\underline{q})_{X}, $ we obtain that for $ t \geq 2$,
    \[ \sum_{i=1}^t p_i =\sum_{i=1}^t p_i' -2d \geq \sum_{i=1}^t q_i' -2d= \sum_{i=1}^t q_i. \]
    Thus, it remains to check $p_1 \geq q_1$. Suppose the contrary that $p_1< q_1$. Since $p_1'\geq q_1'$, we must have 
    \[p_1+2d=p_1'=q_1'=q_1+2d-1\]
    and $ -\varepsilon= (-1)^{q_1}$. However, since $p_1+p_2\geq q_1+q_2$, we must have $p_2 > q_2$ and hence
    \[ q_1 >p_1 \geq p_2 >q_2.\]
    This implies that the multiplicity of $q_1$ in $\underline{q}$ is 1, which is odd. Thus, we derive a contradiction that $\underline{q}$ is not of type $X$. This completes the proof of the lemma.
\end{proof}

\subsection{\texorpdfstring{Type $B$}{}}\label{sec B}
In this subsection, we prove \eqref{eq goal partition} when $\underline{p}$ is of type $B$. We may rewrite the equality as 
\begin{align}\label{eq goal type B}
    (\underline{p}\sqcup [b^{2d}])^{-}\underline{\vphantom{p}}_{C}{}^{\ast} =   ([(2d)^b] +  \underline{p}^{-}\underline{\vphantom{p}}_{C}{}^{\ast})_C.
\end{align} 
First, we compute $(\underline{p}\sqcup [b^{2d}])^{-}\underline{\vphantom{p}}_C$ explicitly.
\begin{lemma}\label{lem LHS type B}
For any $\underline{p}\in \mathcal{P}(2n+1)$, we have 
\begin{align}\label{eq LHS type B}
    ((\underline{p}\sqcup [b^{2d}])^{-})_C= (\underline{p}^{-})_C \sqcup [b^{2d}]
\end{align}
unless all of the following conditions hold.
\begin{enumerate}
    \item [(a)] $b$ is odd.
    \item [(b)] $|\underline{p}_{>b}|$ is odd.
    \item [(c)]$\underline{p}\ul{=b}$ is empty.
   
\end{enumerate}
 If \eqref{eq LHS type B} fails, then
\begin{align}\label{eq LHS type B bad}
    ((\underline{p}\sqcup [b^{2d}])^{-})_C= (\underline{p}^{-})_C \sqcup [b+1,b^{2d-2},b-1].
\end{align}
\end{lemma}
\begin{proof}
First, we deal with the case that $ \underline{p}_{\leq b}$ is empty. Note that Conditions (b) and (c) automatically hold in this case. By Lemma \ref{lem Ach03}, we have
\begin{align*}
    ((\underline{p}\sqcup [b^{2d}])^{-})_C&= (\underline{p} \sqcup [b^{2d-1}, b-1])_C\\
    &= (\underline{p}^{-})_C\sqcup ([b+1,b^{2d-2},b-1])_C.
\end{align*}
If $b$ is even, then $([b+1,b^{2d-2},b-1])_C=[b^{2d}]$, which gives \eqref{eq LHS type B}. Otherwise, we get \eqref{eq LHS type B bad}. This completes the verification of this case.

Next, we consider the case that $\underline{p}_{\leq b}$ is non-empty, which gives that $(\underline{p}\sqcup [b^{2d}])^{-}=\underline{p}^{-}\sqcup [b^{2d}]$. Note that in this case, we have
\[ \underline{p}_{\oldbullet x}= (\underline{p}^{-})_{\oldbullet x}\]
for $x \geq b$ and $\oldbullet \in \{>, \geq \}$.

 If $b$ is even, then \eqref{eq LHS type B} follows from Corollary \ref{cor Ach03}(ii). If $|\underline{p}_{>b}|$ is even, then 
 \begin{align*}
      ((\underline{p}\sqcup [b^{2d}])^{-})_C&= (\underline{p}_{>b})_C \sqcup ([b^{2d}] \sqcup (\underline{p}^{-})_{\leq b})_C \\
      &= (\underline{p}_{>b})_C \sqcup [b^{2d}] \sqcup ((\underline{p}^{-})_{\leq b})_C\\
      &= (\underline{p}^{-})_{C} \sqcup [b^{2d}].
 \end{align*}
 
 If $\underline{p}\ul{=b}$ is non-empty, $b$ is odd and $|\underline{p}_{>b}|$ is odd, then $\underline{p}^{-}\ul{=b}$ is also non-empty, for otherwise, $\underline{p}=\underline{p}_{>b}\sqcup [b]$ and hence $|\underline{p}|$ is even, a contradiction. Therefore,
 \begin{align*}
     (\underline{p}\sqcup [b^{2d}])^{-}\ul{C}&= \underline{p}_{>b}{}^{-}\ul{C} \sqcup ([b^{2d}] \sqcup \underline{p}^-\ul{=b} \sqcup \underline{p}^{-}\ul{< b})^{+}\ul{C} \\
     &=\underline{p}_{>b}{}^{-}\ul{C} \sqcup [b^{2d}] \sqcup (\underline{p}^-\ul{=b}{}^{+}\sqcup \underline{p}^{-}\ul{< b})\ul{C}\\
     &=\underline{p}_{>b}{}^{-}\ul{C} \sqcup [b^{2d}] \sqcup \underline{p}^{-}\ul{\leq b}{}^{+}\ul{C}=\underline{p}^{-}\ul{C} \sqcup [b^{2d}].
 \end{align*}
 
 Finally, assuming Conditions (a), (b) and (c) hold, we have
 \begin{align*}
     (\underline{p}\sqcup [b^{2d}])^{-}\ul{C}&= (\underline{p}_{>b} \sqcup ( [b^{2d}] \sqcup \underline{p}_{<b})^{-})\ul{C}\\
     &= \underline{p}_{>b}{}^{-}\ul{C} \sqcup  ( [b+1, b^{2d-1}] \sqcup \underline{p}_{<b})^{-}\ul{C}\\
     &=\underline{p}_{>b}{}^{-}\ul{C}  \sqcup [b+1, b^{2d-2},b-1]\sqcup \underline{p}_{<b}{}^{+-}\ul{C}\\
     &= \underline{p}^{-}\ul{C} \sqcup [b+1,b^{2d-2},b-1].
 \end{align*}
 This completes the proof of the lemma.
\end{proof}
We give a corollary of the computation in the proof above, which will be used in future work.
\begin{cor}\label{cor type B ineq}
Suppose $\underline{p}, \underline{q}$ are partitions in $\mathcal{P}_B(2n+1)$ such that $\underline{p} \geq \underline{q}$ and $d_{BV}(\underline{p})<d_{BV}(\underline{q})$. Then for any positive integers $b,d$, 
\[ d_{BV}(\underline{p}\sqcup [b^{2d}])< d_{BV}(\underline{q}\sqcup [b^{2d}]). \]
\end{cor}
\begin{proof}
Since $d_{BV}$ is order-reversing and $\underline{p}\geq \underline{q}$ implies $\underline{p} \sqcup [b^{2d}] \geq  \underline{q} \sqcup [b^{2d}] $, we already have 
\[ d_{BV}(\underline{p}\sqcup [b^{2d}])\leq  d_{BV}(\underline{q}\sqcup [b^{2d}]). \]
Therefore, it suffices to show this inequality is strict.

Since transpose is an order-reversing bijection on $\mathcal{P}(2n)$, we have 
\[\underline{p}^{-}\ul{C}= d_{BV}(\underline{p})^{\ast} >d_{BV}(\underline{q})^{\ast}= \underline{q}^{-}\ul{C}. \]
If \eqref{eq LHS type B} holds for $\underline{q}$, then since either \eqref{eq LHS type B} or \eqref{eq LHS type B bad} holds for $\underline{p}$, we have
\[ (\underline{p}\sqcup [b^{2d}])^{-}\ul{C} \geq \underline{p}^{-}\ul{C} \sqcup [b^{2d}]>\underline{q}^{-}\ul{C} \sqcup [b^{2d}] = (\underline{q}\sqcup [b^{2d}])^{-}\ul{C},  \]
which gives the desired inequality by taking transpose again.

Suppose \eqref{eq LHS type B bad} holds for $\underline{q}$. In this case, the computation in the proof of Lemma \ref{lem LHS type B} shows that the multiplicity of $b$ in $ (\underline{q} \sqcup [b^{2d}])^{-}\ul{C}$ is exactly $2d-2$. Therefore, either \eqref{eq LHS type B bad} also holds for $\underline{p}$, and hence the desired strict inequality holds by the same argument above, or 
\[(\underline{p} \sqcup [b^{2d}])^{-}\ul{C}=\underline{p}^{-}\ul{C} \sqcup [b^{2d}] \neq (\underline{q} \sqcup [b^{2d}])^{-}\ul{C} \]
by comparing the multiplicity of $b$. As the equality does not hold in this case, we obtain the desired strict inequality. This completes the proof of the corollary.
\end{proof}

Using the fact that $d_{BV}(\underline{p} \sqcup [b^{2d}])$ is of type $C$ if $\underline{p}$ is of type $B$, we prove Lemma \ref{lem goal} when $(X,X')=(B,C)$. 
\begin{prop} 
Lemma \ref{lem goal} holds when $(X,X')=(B,C)$. 
\end{prop}
\begin{proof}
It is equivalent to establish \eqref{eq goal type B}. Consider the partitions appear in Lemma \ref{lem LHS type B}
\[ \underline{q}:=\underline{p}^{-}\ul{C} \sqcup [b^{2d}] \leq  \underline{q}':=\underline{p}^{-}\ul{C} \sqcup [b+1,b^{2d-2},b-1]. \]
It is not hard to see from Definition \ref{def transpose} that
\begin{align*}
    \underline{q}^{\ast}&=[2d^b]+ \underline{p}^{-}\ul{C}{}^{\ast},\\
    (\underline{q}')^{\ast}&=[2d^{b-1},2d-1,1]+ \underline{p}^{-}\ul{C}{}^{\ast}.
\end{align*}
Note that $\underline{q}^{\ast}$ is exactly the right hand side of \eqref{eq goal type B} before taking $C$-collapse. Therefore, if the left hand side of \eqref{eq goal type B} is equal to $\underline{q}^{\ast}$, then $\underline{q}^{\ast}$ is already of type $C$ and hence \eqref{eq goal type B} holds. 

On the other hand, if the left hand side of \eqref{eq goal type B} is equal to $(\underline{q}')^{\ast}$, then $b$ is odd, $|\underline{p}_{>b}|$ is odd and $\underline{p}_{=b}$ is empty by Lemma \ref{lem LHS type B}. In particular, $l((\underline{p}^{-}\ul{C}) \ul{>b})= l(\underline{p}_{>b})$, which is odd since $ \underline{p}$ is of type $B$. Therefore, if we denote $\underline{p}^{-}\ul{C}{}^{\ast}= [p_1,\dots, p_N]$, then $p_{b+1}$ is odd. 

Next, observe that 
\[ (\underline{q}^{\ast})_{\geq p_b+2d}=[p_1+2d,\dots, p_{b}+2d],\ (\underline{q}^{\ast})_{<p_b+2d}= [p_{b+1},\dots, p_{N}], \]
and
\[ (\underline{q}')^{\ast}=(\underline{q}^{\ast})\ul{\geq p_b+2d}{}^{-} \sqcup (\underline{q}^{\ast})\ul{< p_b+2d}{}^{+}. \]
Since $(\underline{q}')^{\ast}=d_{BV}(\underline{p}\sqcup [b^{2d}])$ is of type $C$ and $p_{b+1}+1>p_{b+2}$, we see that $(\underline{q}^{\ast})\ul{\geq p_b+2d}{}^{-} \sqcup [p_{b+1}+1]$ is of type $C$. As $p_{b+1}$ is odd, we conclude that both $(\underline{q}^{\ast})\ul{\geq p_b+2d}{}^{-} $ and $(\underline{q}^{\ast})\ul{< p_b+2d}{}^{+}$ are of type $C$. In particular, $|(\underline{q}^{\ast})\ul{\geq p_b+2d}|$ is odd, and Lemma \ref{lem Ach03} implies that 
\[ (\underline{q}^{\ast})\ul{C}=((\underline{q}^{\ast})\ul{\geq p_b+2d}{}^{-} )\ul{C} \sqcup ((\underline{q}^{\ast})\ul{< p_b+2d}{}^{+})\ul{C}=  (\underline{q}^{\ast})\ul{\geq p_b+2d}{}^{-} \sqcup (\underline{q}^{\ast})\ul{< p_b+2d}{}^{+}=(\underline{q}')^{\ast}. \]
This completes the proof of the proposition.
\end{proof}

\subsection{\texorpdfstring{Type $C$}{}}\label{sec C}
In this subsection, we prove \eqref{eq goal partition} when $\underline{p}$ is of type $C$. We may rewrite the equality as
\begin{align}\label{eq goal type C}
    (\underline{p}\sqcup [b^{2d}])^{+}\underline{\vphantom{p}}_{B}{}^{\ast} =   ([(2d)^b] +  \underline{p}^{+}\underline{\vphantom{p}}_{C}{}^{\ast})\underline{\vphantom{p}}_B.
\end{align} 
For simplicity, throughout this and the next section, we denote the $O$-collapse of a partition $\underline{p}$ by
\[ \underline{p}_O:= \begin{cases} \underline{p}_B &\text{if }|\underline{p}| \text{ is odd,}\\ 
\underline{p}_D &\text{if }|\underline{p}| \text{ is even.}\end{cases}\]
With this notation, we may rephrase Lemma \ref{lem Ach03} in this case as
\[ \underline{p}_{O}= \begin{cases}
 \underline{p}_{>x} {}^{-}\ul{O} \sqcup \underline{p}_{\leq x} {}^{+}\ul{O} & \text{ if }l( \underline{p}_{>x} )+| \underline{p}_{>x} | \text{ is odd},\\
    \underline{p}_{>x,O}  \sqcup  \underline{p}_{\leq x,O}  & \text{ if }l( \underline{p}_{>x} )+| \underline{p}_{>x} | \text{ is even}.
\end{cases}\]

First, we compute $  (\underline{p}\sqcup [b^{2d}])^{+}\underline{\vphantom{p}}_{B}$ explicitly.
\begin{lemma}\label{lem LHS type C}
For any $\underline{p}\in \mathcal{P}(2n)$, we have 
\begin{align}\label{eq LHS type C}
    (\underline{p}\sqcup [b^{2d}])^{+}\underline{\vphantom{p}}_{B}= \underline{p}^{+}\underline{\vphantom{p}}_{B}\sqcup [b^{2d}]
\end{align}
unless all of the following conditions hold.
\begin{enumerate}
    \item [(a)] $b$ is even.
    \item [(b)] $l(\underline{p}_{>b})+|\underline{p}_{>b}|$ is even.
    \item [(c)]$\underline{p}_{=b}$ is empty.
\end{enumerate}
 If \eqref{eq LHS type C} fails, then
\begin{align}\label{eq LHS type C bad}
    (\underline{p}\sqcup [b^{2d}])^{+}\underline{\vphantom{p}}_{B}= \underline{p}^{+}\underline{\vphantom{p}}_{B} \sqcup [b+1,b^{2d-2},b-1].
\end{align}
\end{lemma}
\begin{proof}
First, we deal with the case that $\underline{p}_{\geq b}$ is empty, which gives \[(\underline{p}\sqcup [b^{2d}])^{+}= [b+1,b^{2d-1}] \sqcup \underline{p}.\]
Note that conditions (b) and (c) automatically holds in this case. We have
\begin{align*}
    ( [b+1,b^{2d-1}] \sqcup \underline{p})_B= ([b+1,b^{2d-2},b-1])\underline{\vphantom{p}}_D \sqcup \underline{p}^{+}\underline{\vphantom{p}}_B
\end{align*}
Then \eqref{eq LHS type C} holds if $b$ is odd, and \eqref{eq LHS type C bad} holds if $b$ is even.

Next, we deal with the case that $\underline{p}_{\geq b}$ is non-empty, which gives 
\[(\underline{p}\sqcup [b^{2d}])^{+}= \underline{p}^{+} \sqcup [b^{2d}].\] 
Note that in this case, we have
\[ \underline{p}_{\oldbullet x}= \underline{p}^{+}\underline{\vphantom{p}}_{\oldbullet x}\]
for $x \leq b$ and $\oldbullet \in \{<, \leq \}$.

If $b$ is odd, then \eqref{eq LHS type C} follows from Corollary \ref{cor Ach03}(i). If $l(\underline{p}_{>b})+|\underline{p}_{>b}|$ is odd, then $l(\underline{p}^{+}\ul{>b})+|\underline{p}^{+} \ul{>b}|$ is even, and hence
\begin{align*}
    (\underline{p}^{+}\sqcup [b^{2d}])\ul{B}&= \underline{p}^{+} \ul{>b,O}  \sqcup ([b^{2d}] \sqcup \underline{p}_{\leq b}) \ul{O}\\
    &=\underline{p}^{+} \ul{>b,O} \sqcup [b^{2d}] \sqcup \underline{p}_{\leq b,O} \\
    &=\underline{p}^{+}\ul{B} \sqcup [b^{2d}].
\end{align*}
If $\underline{p}_{=b}=[b^{\alpha}]$ is non-empty, $b$ is even, and $l(\underline{p}_{>b})+|\underline{p}_{>b}|$ is even, then
\begin{align*}
    (\underline{p}^{+}\sqcup [b^{2d}])\ul{B}&= (\underline{p}^{+} \ul{>b} {}^{-}) \ul{O} \sqcup ([b^{2d+\alpha}]  \sqcup \underline{p}_{<b})^{+} \ul{O}\\
    &=(\underline{p}^{+} \ul{>b} {}^{-}) \ul{O} \sqcup [b+1] \sqcup ([b^{2d}] \sqcup ([b^{\alpha-1}]\sqcup \underline{p}_{<b}) )\ul{O}\\
    &=(\underline{p}^{+} \ul{>b} {}^{-}) \ul{O} \sqcup [b^{2d}] \sqcup [b+1] \sqcup ([b^{\alpha-1}]\sqcup \underline{p}_{<b}) \ul{O}\\
    &=(\underline{p}^{+} \ul{>b} {}^{-}) \ul{O} \sqcup [b^{2d}]\sqcup  \underline{p}_{\leq b}{}^{+} \ul{O}\\
    &=\underline{p}^{+}\ul{B} \sqcup [b^{2d}].
\end{align*}

Finally, assuming Conditions (a), (b) and (c) hold, we have 
\begin{align*}
    (\underline{p}^{+}\sqcup [b^{2d}])\ul{B}&=(\underline{p}^{+} \ul{>b} {}^{-}) \ul{O} \sqcup ([b^{2d}] \sqcup \underline{p}_{<b} )^{+} \ul{O}\\
    &=(\underline{p}^{+} \ul{>b} {}^{-}) \ul{O} \sqcup [b+1]\sqcup ([b^{2d-1}] \sqcup \underline{p}_{<b} ) \ul{O}\\
    &= (\underline{p}^{+} \ul{>b} {}^{-}) \ul{O} \sqcup [b+1]\sqcup [b^{2d-2},b-1]\ul{O} \sqcup \underline{p}_{<b}{}^{+} \ul{O}\\
    &=\underline{p}^{+}\underline{\vphantom{p}}_{B} \sqcup [b+1,b^{2d-2},b-1].
\end{align*}
This completes the proof of the lemma.
\end{proof}
Remark that if $\underline{p}\in \mathcal{P}_C(2n)$, then the Condition (b) above is equivalent to $l(\underline{p}_{>b})$ is even since $|\underline{p}_{>x}|$ is even for any $x$ given that $\underline{p}$ is of type $C$.

Similar to the case of type $B$, we give the following corollary.
\begin{cor}\label{cor type C ineq}
Suppose $\underline{p}, \underline{q}$ are partitions in $\mathcal{P}_C(2n)$ such that $\underline{p} \geq \underline{q}$ and $d_{BV}(\underline{p})<d_{BV}(\underline{q})$. Then for any positive integers $b,d$, 
\[ d_{BV}(\underline{p}\sqcup [b^{2d}])< d_{BV}(\underline{q}\sqcup [b^{2d}]). \]
\end{cor}
\begin{proof}
The proof is exactly the same as Corollary \ref{cor type B ineq}, which we omit.
\end{proof}

Using the fact that $d_{BV}(\underline{p} \sqcup [b^{2d}])$ is of type $B$ if $\underline{p}$ is of type $C$, we prove Lemma \ref{lem goal} when $(X,X')=(C,B)$.

\begin{prop}
Lemma \ref{lem goal} holds when $(X,X')=(C,B)$. 
\end{prop}
\begin{proof}
It is equivalent to prove \eqref{eq goal type C}. Consider the partitions appear in Lemma \ref{lem LHS type C}
\[ \underline{q}:=\underline{p}^{+}\ul{B} \sqcup [b^{2d}] \leq  \underline{q}':=\underline{p}^{+}\ul{B} \sqcup  [b+1,b^{2d-2},b-1]. \]
It is not hard to see from Definition \ref{def transpose} that
\begin{align*}
    \underline{q}^{\ast}&=[2d^b]+ \underline{p}^{+}\ul{B}{}^\ast,\\
    (\underline{q}')^{\ast}&=[2d^{b-1},2d-1,1]+ \underline{p}^{+}\ul{B}{}^\ast.
\end{align*}
Note that $\underline{q}^{\ast}$ is exactly the right hand side of \eqref{eq goal type C} before taking $B$-collapse. Therefore, if the left hand side of \eqref{eq goal type C} is equal to $\underline{q}^{\ast}$, then $\underline{q}^{\ast}$ is already of type $B$ and hence \eqref{eq goal type C} holds.

On the other hand, if the left hand side of \eqref{eq goal type C} is equal to $(\underline{q}')^{\ast}$, then $b$ is even, $l(\underline{p}_{>b})$ is even and $\underline{p}_{=b}$ is empty by Lemma \ref{lem LHS type C}. This implies that $l((\underline{p}^{+} \ul{B})\ul{\geq b+1})=l(\underline{p}_{>b}) $, which is even. Therefore, if we denote $\underline{p}^{+}\ul{B}{}^{\ast}=[p_1,\dots, p_N]$, then $p_{b+1}$ is even. Observe that
\[ \underline{q}^{\ast} \ul{\geq p_b+2d}=[p_1+2d,\dots, p_b +2d], \ \underline{q}^{\ast} \ul{< p_b+2d}=[p_{b+1},\dots, p_N ],  \]
and
\[ (\underline{q}')^{\ast}= \underline{q}^\ast \ul{\geq p_b+2d} {}^{-}  \sqcup \underline{q}^\ast \ul{< p_b+2d} {}^{+}.  \]
Since $(\underline{q}')^{\ast}$ is already of type $B$ and $p_{b+1}+1> p_{b+2}$, we see that $\underline{q}^\ast \ul{\geq p_b+2d} {}^{-} \sqcup [p_{b+1}+1]$ is of type $B$ or $D$, and hence so is $\underline{q}^\ast \ul{\geq p_b+2d} {}^{-}$. This implies that $ l(\underline{q}^\ast \ul{\geq p_b+2d} )+ |\underline{q}^\ast \ul{\geq p_b+2d}|$ is odd and 
\[\underline{q}^{\ast} \ul{B}= (\underline{q}^\ast \ul{\geq p_b+2d} {}^{-})\ul{O}  \sqcup (\underline{q}^\ast \ul{< p_b+2d} {}^{+})\ul{O}= (\underline{q}')^{\ast},\] 
which completes the proof of the proposition.
\end{proof}

\subsection{\texorpdfstring{Type $D$}{}}\label{sec D}
In this subsection, we prove \eqref{eq goal partition} when $\underline{p}$ is of type $D$. We prove the following equality directly.
\begin{align}\label{eq goal type D}
    (\underline{p}\sqcup [b^{2d}])^{+-}\ul{C}{}^{\ast} =   ([(2d)^b] +  \underline{p}^{+-}\ul{C}{}^{\ast}) \ul{D}.
\end{align} 

First, we compute $(\underline{p}\sqcup [b^{2d}])^{+-}\underline{\vphantom{p}}_{C}$ explicitly.
\begin{lemma}\label{lem LHS type D}
For any $\underline{p}\in \mathcal{P}(2n)$, we have 
\begin{align}\label{eq LHS type D}
    (\underline{p}\sqcup [b^{2d}])^{+-}\underline{\vphantom{p}}_{C}= \underline{p}^{+-}\underline{\vphantom{p}}_{C}\sqcup [b^{2d}]
\end{align}
unless all of the following conditions hold.
\begin{enumerate}
    \item [(a)] $b$ is odd.
    \item [(b)] $|\underline{p}_{>b}|$ is even.
    \item [(c)]$\underline{p}_{=b}$ is empty.
\end{enumerate}
 If \eqref{eq LHS type D} fails, then
\begin{align}\label{eq LHS type D bad}
    (\underline{p}\sqcup [b^{2d}])^{+-}\underline{\vphantom{p}}_{C}= \underline{p}^{+-}\underline{\vphantom{p}}_{C} \sqcup [b+1,b^{2d-2},b-1].
\end{align}
\end{lemma}
\begin{proof}
First, we deal with the cases that $\underline{p}_{\geq b}$ is empty or $\underline{p}_{\leq b}$ is empty. Note that Conditions (b), (c) automatically hold in these cases.

If $\underline{p}_{\geq b}$ is empty, then
\[ (\underline{p}\sqcup [b^{2d}])^{+-}= [b+1,b^{2d-1}] \sqcup \underline{p}^{-},\]
and hence
\begin{align*}
     (\underline{p}\sqcup [b^{2d}])^{+-}\underline{\vphantom{p}}_{C}= ([b+1,b^{2d-2},b-1])\ul{C} \sqcup (\underline{p}^{-+} \ul{C}).
\end{align*}
Thus \eqref{eq LHS type D} holds if $ b$ is even, and \eqref{eq LHS type D bad} holds if $b$ is odd. Similarly, if $\underline{p}_{\leq b}$ is empty, then
\[ (\underline{p}\sqcup [b^{2d}])^{+-}=  \underline{p}^{+}\sqcup [b^{2d-1},b-1],\]
and hence
\begin{align*}
     (\underline{p}\sqcup [b^{2d}])^{+-}\underline{\vphantom{p}}_{C}= (\underline{p}^{+-} \ul{C}) \sqcup ([b+1,b^{2d-2},b-1])\ul{C}.
\end{align*}
The same conclusion holds. 

Next, we deal with the case that both $\underline{p}_{\geq b}$ and $\underline{p}_{\leq b}$ are non-empty, which gives 
\[  (\underline{p}\sqcup [b^{2d}])^{+-}= \underline{p}^{+-} \sqcup [b^{2d}]. \]

If $b$ is even, then \eqref{eq LHS type D} follows from Corollary \ref{cor Ach03}(ii). If $|\underline{p}_{>b}|$ is odd, then $|\underline{p}^{+} \ul{{>b}}|$ is even, and 
\begin{align*}
    (\underline{p}^{+-} \sqcup [b^{2d}] )\ul{C}&= \underline{p}^{+} \ul{>b, C} \sqcup ([b^{2d}] \sqcup \underline{p}^- \ul{\leq b} )\ul{C}\\
    &= \underline{p}^{+} \ul{>b, C} \sqcup [b^{2d}] \sqcup \underline{p}^- \ul{\leq b,C} \\
    &= \underline{p}^{+-} \ul{C}\sqcup [b^{2d}].
\end{align*}
If $b$ is odd, $|\underline{p}_{>b}|$ is even, and $\underline{p}_{=b}=[b^{\alpha}]$ is non-empty, then $[b^{\alpha-1}] \sqcup \underline{p}_{<b}$ is non-empty since $|[b^{\alpha-1}] \sqcup \underline{p}_{<b}|$ is odd, and hence
\begin{align*}
    (\underline{p}^{+-} \sqcup [b^{2d}] )\ul{C}&=( \underline{p}_{>b}  \sqcup  [ b^{2d+\alpha}] \sqcup \underline{p}_{<b})^{+-}  \ul{C}\\
    &=\underline{p}_{>b} {}^{+-}\ul{C}  \sqcup ([ b^{2d+1}] \sqcup [b^{\alpha-1}] \sqcup \underline{p}_{<b})^{+-}  \ul{C}\\
    &= \underline{p}_{>b} {}^{+-}\ul{C}  \sqcup ( [b+1, b^{2d}]\sqcup ([b^{\alpha-1}] \sqcup \underline{p}_{<b})^{-})  \ul{C}\\
    &= \underline{p}_{>b} {}^{+-}\ul{C}  \sqcup  [b+1, b^{2d}]\sqcup ( [b^{\alpha-1}] \sqcup \underline{p}_{<b})^{-}  \ul{C}\\
    &=\underline{p}_{>b} {}^{+-}\ul{C}  \sqcup  [ b^{2d}]\sqcup ( [b+1,b^{\alpha-1}] \sqcup \underline{p}_{<b})^{-}  \ul{C}\\
    &= (\underline{p}_{>b} {}^{+})^{-}\ul{C} \sqcup (\underline{p}_{\leq b}{}^{-})^{+}\ul{C} \sqcup[b^{2d}]\\
    &=\underline{p}^{+-}\ul{C} \sqcup [ b^{2d}].
\end{align*}
Finally, assuming Conditions (a), (b) and (c) hold and both $\underline{p}_{>b}$ and $\underline{p}_{<b}$ are non-empty, then  
\begin{align*}
     (\underline{p}^{+-} \sqcup [b^{2d}] )\ul{C}&= (\underline{p}_{>b}{}^{+} \sqcup [b^{2d}] \sqcup \underline{p}_{<b}{}^{-})\ul{C}\\
     &= \underline{p}_{>b}{}^{+-} \ul{C} \sqcup ([b+1,b^{2d-1}] \sqcup \underline{p}_{<b}{}^{-}) \ul{C}\\
     &=\underline{p}_{>b}{}^{+-} \ul{C} \sqcup ([b+1,b^{2d-2},b-1])\ul{C} \sqcup \underline{p}_{<b}{}^{+-} \ul{C}\\
     &=\underline{p}^{+-}\ul{C} \sqcup [b+1,b^{2d-2},b-1].
\end{align*}
This completes the proof of the lemma.
\end{proof}

Similar to the cases of type $B,C$, we give the following corollary.

\begin{cor}\label{cor type D ineq}
Suppose $\underline{p}, \underline{q}$ are partitions in $\mathcal{P}_D(2n)$ such that $\underline{p}\geq \underline{q}$ and $d_{BV}(\underline{p})<d_{BV}(\underline{q})$. Then for any positive integers $b,d$, we have
\[ d_{BV}(\underline{p}\sqcup [b^{2d}])< d_{BV}(\underline{q}\sqcup [b^{2d}]). \]
\end{cor}
\begin{proof}
The proof is exactly the same as Corollary \ref{cor type B ineq}, which we omit.
\end{proof}

Using the fact that $d_{BV}(\underline{p} \sqcup [b^{2d}])$ is of type $D$ if $\underline{p}$ is of type $D$, we prove the last case of Lemma \ref{lem goal}.
\begin{prop}
Lemma \ref{lem goal} holds when $(X,X')=(D,D)$.
\end{prop}
\begin{proof}
It is equivalent to prove \eqref{eq goal type D}. Consider the partitions appear in Lemma \ref{lem LHS type D}
\[ \underline{q}:=\underline{p}^{+-}\ul{C} \sqcup [b^{2d}] \leq  \underline{q}':=\underline{p}^{+-}\ul{C} \sqcup  [b+1,b^{2d-2},b-1]. \]
It is not hard to see from Definition \ref{def transpose} that
\begin{align*}
    \underline{q}^{\ast}&=[2d^b]+ \underline{p}^{+-}\ul{C}{}^\ast,\\
    (\underline{q}')^{\ast}&=[2d^{b-1},2d-1,1]+ \underline{p}^{+-}\ul{C}{}^\ast.
\end{align*}
Note that $\underline{q}^{\ast}$ is exactly the right hand side of \eqref{eq goal type D} before taking $D$-collapse. Therefore, if the left hand side of \eqref{eq goal type D} is equal to $\underline{q}^{\ast}$, then $\underline{q}^{\ast}$ is already of type $D$ and hence \eqref{eq goal type D} holds.

On the other hand, if the left hand side of \eqref{eq goal type D} is equal to $(\underline{q}')^{\ast}$, then $b$ is odd, $|\underline{p}_{>b}|$ is even and $\underline{p}_{=b}$ is empty by Lemma \ref{lem LHS type D}. This implies that $l(\underline{p}^{+-} \ul{C, > b})= l(\underline{p}_{> b})$, which is even since $\underline{p}$ is of type $D$. Therefore, if we denote $\underline{p}^{+-}\ul{C}{}^{\ast}=[p_1,\dots, p_N]$, then $p_{b+1}$ is even. Observe that
\[ \underline{q}^{\ast} \ul{\geq p_b+2d}=[p_1+2d,\dots, p_b +2d], \ \underline{q}^{\ast} \ul{< p_b+2d}=[p_{b+1},\dots, p_N ],  \]
and
\[ (\underline{q}')^{\ast}= \underline{q}^\ast \ul{\geq p_b+2d} {}^{-}  \sqcup \underline{q}^\ast \ul{< p_b+2d} {}^{+}.  \]
Since $(\underline{q}')^{\ast}$ is already of type $D$ and $p_{b+1}+1> p_{b+2}$, we see that $\underline{q}^\ast \ul{\geq p_b+2d} {}^{-} \sqcup [p_{b+1}+1]$ is of type $B$ or $D$, and hence so is $\underline{q}^\ast \ul{\geq p_b+2d} {}^{-}$. This implies that $ l(\underline{q}^\ast \ul{\geq p_b+2d} )+ |\underline{q}^\ast \ul{\geq p_b+2d}|$ is odd and 
\[\underline{q}^{\ast} \ul{D}= (\underline{q}^\ast \ul{\geq p_b+2d} {}^{-})\ul{O}  \sqcup (\underline{q}^\ast \ul{< p_b+2d} {}^{+})\ul{O}= (\underline{q}')^{\ast},\] 
which completes the proof of the proposition.
\end{proof}


\begin{thebibliography}{999999}


\bibitem[Ach03]{Ach03} P.~Achar, An order-reversing duality map for conjugacy classes in Lusztig's canonical quotient. {\em Transform. Groups}. \textbf{8}, 107-145 (2003).

\bibitem[ABV92]{ABV92} J. Adams, D. Barbasch, and D. Vogan, {\it The Langlands classification and irreducible characters for real reductive groups,} Progress in Mathematics, vol. 104, Birkh\"auser Boston, Inc., Boston, MA, 1992.



\bibitem[Art89]{Art89} J. Arthur, Unipotent automorphic representations: conjectures. {\em Ast{\'e}risque}. pp. 13-71 (1989), Orbites unipotentes et repr{\' e}sentations, II.

\bibitem[Art13]{Art13}
J. Arthur,
{\it The endoscopic classification of representations: Orthogonal and Symplectic groups.}
{\it Colloquium Publication}
Vol. \textbf{61}, 2013,
American Mathematical Society.





\bibitem[Ato22]{Ato22} H.~Atobe, On the socles of certain parabolically induced representations of p-adic classical groups. {\em Represent. Theory}. \textbf{26} pp. 515-541 (2022).

\bibitem[AC26]{AC26}
{H. Atobe and D. Ciubotaru,
Endoscopic transfer and the wavefront upper bound conjecture.
Preprint. 2026.
} 

\bibitem[AG17]{AG17}H.~Atobe and W.~T.~Gan, On the local Langlands correspondence and Arthur conjecture for even orthogonal groups. {\em Represent. Theory}. \textbf{21} pp. 354-415 (2017).



\bibitem[AM23]{AM20} H. Atobe and A. M{\'i}nguez, The explicit Zelevinsky-Aubert duality. {\em Compos. Math.} \textbf{159}, 380-418 (2023).

\bibitem[Aub95]{Aub95} A.-M. Aubert, Dualit{\' e} dans le groupe de Grothendieck de la cat{\' e}gorie des repr{\' e}sentations lisses de longueur finie d'un groupe r{\' e}ductif p-adique. {\em Trans. Amer. Math. Soc.} \textbf{347}  (1995) 2179-2189.


\bibitem[AMS21]{AMS21} A.-M.~Aubert, A.~Moussaoui, and M.~Solleveld, Affine Hecke algebras for Langlands parameters.  (2021), arXiv:1701.03593v4. 


\bibitem[AMS22]{AMS22}
A.-M.~Aubert, A.~Moussaoui, and M.~Solleveld, Affine Hecke algebras for classical $p$-adic groups. (2022), arXiv:2211.08196.

\bibitem[Bad02]{Bad02} A.~Badulescu, Correspondance de Jacquet-Langlands pour les corps locaux de caract{\'e}ristique non nulle. {\em Ann. Sci. {\'E}cole Norm. Sup. (4)}. \textbf{35}, 695-747 (2002).


\bibitem[BZ77]{BZ77} I. Bernstein and A.  Zelevinsky, Induced representations of reductive  p-adic groups. I. {\em Ann. Sci. {\'E}cole Norm. Sup. (4)}. \textbf{10}, 441-472 (1977).


\bibitem[BV85]{BV85}
D. Barbasch and D. Vogan,
 Unipotent representations of complex semisimple groups.
{\it Ann. of Math.} (2) \textbf{121} (1985), no. 1, 41--110.

\bibitem[Bor79]{Bor79} A. Borel, 
Automorphic L-functions. in \textit{ Automorphic forms, representations and L-functions (Proc. Sympos. Pure Math., Oregon State Univ., Corvallis, Ore., 1977), Part 2,} pp. 27-61, Proc. Sympos. Pure Math., XXXIII, Amer. Math. Soc., Providence, R.I., 1979.


\bibitem[BW80]{BW80}A. Borel and N. Wallach. Continuous cohomology, discrete subgroups, and representations of reductive groups. (Princeton University Press, Princeton, NJ; University of Tokyo Press, Tokyo, 1980).

\bibitem[Cai23]{Cai23} Y. Cai, Quaternionic Speh Representations, {\em Doc. Math.}, 28 (2023), pp. 903–937.  

\bibitem[CJLZ23]{CJLZ23}C.~ Chen, D.~Jiang, D.~Liu, and L.~ Zhang, Arithmetic Branching Law and generic L-packets. {\em Represent. Theory} {\bf 28} (2024), 328--365.


\bibitem[CG16]{CG16}K.~Choiy and D.~ Goldberg, Invariance of R-groups between p-adic inner forms of quasi-split classical groups. {\em Trans. Amer. Math. Soc.} \textbf{368}, 1387-1410 (2016).

\bibitem[CK25]{CK25} D. Ciubotaru and J.-L. Kim, The wavefront set: bounds for the Langlands parameter, Math. Ann. {\bf 393} (2025), no.~2, 1827--1861.



\bibitem[CMBO22]{CMBO22} D. Ciubotaru, L. Mason-Brown, and E. Okada, Some unipotent Arthur
packets for reductive p-adic groups. {\em Int. Math. Res. Not. IMRN.} (2023).

\bibitem[CMBO25]{CMBO23} D. Ciubotaru, L. Mason-Brown, and E. Okada,
Wavefront Sets of Unipotent Representations of Reductive p-adic Groups II. {\em J. Reine Angew. Math.} {\bf 823} (2025), 191--253.




\bibitem[CM93]{CM93} D. Collingwood, and W. McGovern, Nilpotent orbits in semisimple Lie algebras. (Van Nostrand Reinhold Co., New York, 1993).


\bibitem[CFK22]{CFK22} C.~ Cunningham, A.~ Fiori, and N.~ Kitt, Appearance of the Kashiwara-Saito singularity in the representation theory of p-adic GL16, Pacific Journal of Mathematics (2022).

\bibitem[CFMMX22]{CFMMX22} C. Cunningham, A. Fiori, A. Moussaoui, J. Mracek, and B. Xu, Arthur packets for p-adic groups by way of microlocal vanishing cycles of perverse sheaves, with examples. {\em Mem. Amer. Math. Soc.} \textbf{276} (2022), ix+216.


\bibitem[DKV84]{DKV84}P.~Deligne, D.~Kazhdan, and M.~Vign{\'e}ras, Repr{\'e}sentations des alg{\`e}bres centrales simples p-adiques. {\em Representations Of Reductive Groups Over A Local Field}. pp. 33-117 (1984).


\bibitem[GGP12]{GGP12} W. T. Gan, B. H. Gross, and D. Prasad, Symplectic local root numbers, central critical L values, and restriction problems in the representation theory of classical groups. {\em Ast{\'e}risque}. pp. 1-109 (2012), Sur les conjectures de Gross et Prasad. I.


\bibitem[GGP20]{GGP20} W. T. Gan, B. H. Gross, and D. Prasad, Branching laws for classical groups: The non-tempered case, {\it Compos. Math.} 156(11) (2020), 2298–2367.

\bibitem[HC78]{HC78}
Harish-Chandra, 
{\it Admissible invariant distributions on reductive p-adic groups.} Lie theories and their applications (Proc. Ann. Sem. Canad. Math. Congr., Queen's Univ., Kingston, Ont., 1977), pp. 281--347. Queen's Papers in Pure Appl. Math., No. 48, Queen's Univ., Kingston, Ont., 1978.

\bibitem[HT01]{HT01}M. Harris and R. Taylor, The geometry and cohomology of some simple Shimura varieties. (Princeton University Press, 2001).


\bibitem[HLL22]{HLL22} A. Hazeltine, B. Liu, and C.-H. Lo, On the intersection of local Arthur packets for classical groups. (2022),
arXiv.2201.10539.

\bibitem[HLLZ25]{HLLZ22} A. Hazeltine, B. Liu, C.-H. Lo, and Q. Zhang,
The closure ordering conjecture on local $L$-parameters in local Arthur packets of classical groups. {\em J. Reine Angew. Math.} {\bf 823} (2025), 1--60.


\bibitem[Hen00]{Hen00} G. Henniart, Une preuve simple des conjectures de Langlands pour GL(n) sur un corps p-adique. {\em Invent. Math.} \textbf{139}, 439-455 (2000).


\bibitem[Hir04]{Hir04}K.~Hiraga,  On functoriality of Zelevinski involutions. {\em Compos. Math.} \textbf{140}, 1625-1656 (2004).

\bibitem[Ish23]{Ish23} H. Ishimoto, The endoscopic classification of representations of non-quasi-split odd
special orthogonal groups. {\em Int. Math. Res. Not. IMRN} {\bf 2024}, no.~14, 10939--11012.

\bibitem[JL70]{JL70}H.~Jacquet  and R.~Langlands,  Automorphic forms on GL(2). (Springer-Verlag, Berlin-New York,1970).

\bibitem[Jan18]{Jan18}C.~ Jantzen, Duality for classical p-adic groups: the half-integral case. {\em Represent. Theory}. \textbf{22} pp. 160-201 (2018).

\bibitem[Jia14]{Jia14}
D. Jiang,
{\it Automorphic Integral transforms for classical groups I: endoscopy correspondences}.
Automorphic Forms: L-functions and related geometry: assessing the legacy of I.I. Piatetski-Shapiro, 179-242,
Comtemp. Math. \textbf{614}, 2014, AMS.



\bibitem[JLZ22]{JLZ22} D.~ Jiang, D.~ Liu and L.~ Zhang, {\it Arithmetic Wavefront Sets and Generic L-packets.} (2022), arXiv:2207.04700v2.

\bibitem[JLZ25]{JLZ25} D.~ Jiang, D.~ Liu and L.~ Zhang, {\it Arithmetic Wavefront Sets and Generic $L$-packets.} Adv. Math. {\bf 483} (2025), Paper No. 110668, 107 pp.



\bibitem[KMSW14]{KMSW14}
T. Kaletha, A. M{\'i}nguez, S. W. Shin, and P.-J. White,
Endoscopic classification of representations: Inner forms of unitary groups. Preprint. 2014. arXiv:1409.3731.

\bibitem[KL87]{KL87}
D.~Kazhdan and G.~Lusztig, Proof of the Deligne-Langlands conjecture for Hecke algebras. {\em Invent. Math.} \textbf{87}, 153-215 (1987).

\bibitem[Kon03]{Kon03} T. Konno, A note on the Langlands classification and irreducibility of induced representations of p-adic groups. {\em Kyushu J. Math.} \textbf{57}, 383-409 (2003).

\bibitem[La24]{La24}
R.~La,
The Iwahori-Matsumoto dual for tempered representations of Lustig's geometric Hecke algebras. 
Preprint. 2024. arXiv:2403.14528. 

\bibitem[LL24a]{LL23}
B. Liu and C.-H. Lo, On the weak local Arthur packets conjecture for split classical groups, {\em Int. Math. Res. Not. IMRN} (2024), no.~14, 10708--10731.

\bibitem[LLS24b]{LLS23} 
B. Liu, C.-H. Lo, and F. Shahidi, 
{Jiang's conjecture and Fibers of the Barbasch-Vogan duality.} {\em Rad Hrvat. Akad. Znan. Umjet. Mat. Znan.} {\bf 28(558)} (2024), 107--129.

\bibitem[LLS24c]{LLS24b} 
B. Liu, C.-H. Lo, and F. Shahidi, 
On anti-tempered local Arthur packets and a lemma of Arthur. Preprint 2024. arXiv:2405.17407.

\bibitem[LLS]{LLS24a} 
B. Liu, C.-H. Lo, and F. Shahidi, 
{Generalizations of the Shahidi conjecture on local $L$-packets of connected reductive groups.} Preprint.  

\bibitem[LLSZ]{LLSZ26} 
B. Liu, C.-H. Lo, F. Shahidi, and L. Zhang,
{On rational local Jiang conjecture for classical groups.} In preparation. 2026.    

\bibitem[LS23]{LS23} B. Liu and F. Shahidi, Jiang's conjecture on the wavefront sets of local Arthur packets, Submitted. 2023.


\bibitem[Lus84]{Lus84} G.~Lusztig. {\it Characters of reductive groups over finite fields.} {\em Proceedings Of The International Congress Of Mathematicians, Vol. 1, 2 (Warsaw, 1983)}. pp. 877-880 (1984).


\bibitem[Lus95]{Lus95} 
G.~Lusztig, Classification of unipotent representations of simple p-adic groups. {\em Internat. Math. Res. Notices}., 517-589 (1995).




\bibitem[Lus02]{Lus02}
G.~Lusztig, Classification of unipotent representations of simple p-adic groups. II. {\em Represent. Theory}. \textbf{6} pp. 243-289 (2002).












\bibitem[M{\oe}11b]{Moe11b} C. M{\oe}glin, Image des op{\'e}rateurs d'entrelacements normalis{\'e}s et p{\^ o}les des s{\'e}ries d'Eisenstein. {\em Adv. Math.} \textbf{228}  (2011), 1068-1134.



\bibitem[MR18]{MR18} C. M{\oe}glin, D. Renard, Sur les paquets d'Arthur des groupes classiques et unitaires non quasi-d{\'e}ploy{\'e}s. {\em Relative Aspects In Representation Theory, Langlands Functoriality And Automorphic Forms}. \textbf{2221} pp. 341-361 (2018).

\bibitem[MW86]{MW86} 
C. M{\oe}glin and J.-L.~ Waldspurger,
Sur l'involution de Zelevinski. {\em J. Reine Angew. Math.}. \textbf{372} pp. 136-177 (1986).


\bibitem[MW87]{MW87} C. M{\oe}glin and J.-L.~ Waldspurger, Mod{\'e}les de Whittaker d{\`e}g{\'e}n{\'e}r{\'e}s pour des groupes p-adiques. {\em Math. Z.} \textbf{196}, 427-452 (1987).

\bibitem[Mok15]{Mok15} C. Mok, Endoscopic classification of representations of quasi-split unitary groups. {\em Mem. Amer. Math. Soc.} \textbf{235}, vi+248 (2015).

\bibitem[Oka22]{Oka22} E.~Okada, The wavefront set over a maximal unramified field extension. 2022. arXiv:2107.10591.


\bibitem[Rog93]{Rog83}J.~Rogawski, Representations of GL(n) and division algebras over a p-adic field. {\em Duke Math. J.} \textbf{50}, 161-196 (1983).






\bibitem[Sch13]{Sch13}P. Scholze, The local Langlands correspondence for $\GL_n$ over p-adic fields. {\em Invent. Math.} \textbf{192}, 663-715 (2013).


\bibitem[SZ18]{SZ18}A.~Silberger and E.~Zink,  Langlands classification for L-parameters. {\em J. Algebra}. \textbf{511} pp. 299-357 (2018).

\bibitem[Som01]{Som01}E.~Sommers,  Lusztig's canonical quotient and generalized duality. {\em J. Algebra}. \textbf{243}, 790-812 (2001).

\bibitem[Spa82]{Spa82}
N. Spaltenstein, 
Classes Unipotentes et Sous-groupes de Borel, Lecture Notes in Mathematics, no. 946, Springer-Verlag, 1982.


\bibitem[Tad90]{Tad90}M.~Tadi{\'c},  Induced representations of GL(n,A) for p-adic division algebras A. {\em J. Reine Angew. Math.} \textbf{405} pp. 48-77 (1990).

\bibitem[Tad18]{Tad18} M. Tadi{\'c}, On unitarizability in the case of classical p-adic groups. {\em Geometric Aspects Of The Trace Formula}. pp. 405-453 (2018).


\bibitem[Tsa24]{Tsa24} C.-C.~ Tsai, Geometric wave-front set may not be a singleton. {\em J. Amer. Math. Soc.} \textbf{37} (2024), 281-304.

\bibitem[Var14]{Var14}
S. Varma,
On a result of Moeglin and Waldspurger in residual characteristic 2, 
{\it Math. Z.} 277, no. 3--4,
1027--1048 (2014).

\bibitem[Vog93]{Vog93} D. Vogan, The local Langlands conjecture. {\em Representation Theory Of Groups And Algebras}. \textbf{145} (1993), 305-379.


\bibitem[Wal18]{Wal18}J.-L. Waldspurger, Repr{\'e}sentations de r{\'e}duction unipotente pour  SO(2n+1), III: exemples de fronts d'onde. {\em Algebra Number Theory}. \textbf{12}, 1107-1171 (2018).

\bibitem[Wal20]{Wal20} J.-L. Waldspurger, Fronts d'onde des repr{\'e}sentations temp{\'e}r{\'e}es et de r{\'e}duction unipotente pour SO(2n+1). {\em Tunis. J. Math.} \textbf{2}, 43-95 (2020).

\bibitem[Xu17]{Xu17b} B. Xu, On M{\oe}glin's parametrization of Arthur packets for p-adic quasisplit Sp(N) and SO(N). {\em Canad. J. Math.} \textbf{69}, (2017), 890-960.

\bibitem[Xu21b]{Xu21b} B. Xu, A combinatorial solution to M{\oe}glin's parametrization of Arthur packets for p-adic quasisplit Sp(N) and O(N). {\em J. Inst. Math. Jussieu}. \textbf{20}, 1091-1204 (2021).

\bibitem[Zel80]{Zel80} A. Zelevinsky, Induced representations of reductive $p$-adic groups. II. On irreducible representations of  GL(n). {\em Ann. Sci. {\'E}cole Norm. Sup. (4)} \textbf{13}, 165-210 (1980).


\bibitem[Zel81]{Zel81}
A. Zelevinsky, The p-adic analogue of the Kazhdan-Lusztig conjecture. {\em Funktsional. Anal. I Prilozhen.} \textbf{15}, 9-21, 96 (1981).

\end{thebibliography}
\end{document}